\documentclass[11pt]{article}
\usepackage{amsfonts}
\usepackage{amssymb}
\usepackage{amsthm}
\usepackage{amsmath}
\usepackage{graphicx}
\usepackage{empheq}
\usepackage{indentfirst}
\usepackage{cite}
\usepackage{mathrsfs}
\usepackage{graphics}
\usepackage{xcolor}
\usepackage{times}
\usepackage{bm}
\usepackage{makeidx}
\usepackage{tikz}
\newtheorem{theorem}{\textbf{Theorem}}[section]
\newtheorem{lemma}{\textbf{Lemma}}[section]
\newtheorem{proposition}{\textbf{Proposition}}[section]
\newtheorem{corollary}{\textbf{Corollary}}[section]
\newtheorem{remark}{\textbf{Remark}}[section]
\newtheorem{definition}{\textbf{Definition}}[section]
\allowdisplaybreaks[4]

\def\be{\begin{equation}}
\def\ee{\end{equation}}
\def\bea{\begin{eqnarray}}
\def\eea{\end{eqnarray}}
\def\bt{\begin{theorem}}
\def\et{\end{theorem}}
\def\bl{\begin{lemma}}
\def\el{\end{lemma}}
\def\br{\begin{remark}}
\def\er{\end{remark}}
\def\bp{\begin{proposition}}
\def\ep{\end{proposition}}
\def\bc{\begin{corollary}}
\def\ec{\end{corollary}}
\def\bd{\begin{definition}}
\def\ed{\end{definition}}

\def\vp{\varphi}
\def\ve{\varepsilon}

\def\non{\nonumber}

\def\n{\mathbf{n}}
\def\R3{\mathbb{R}^3}
\def\F2o{\overline{F_2}}
\def\N{\mathcal{N}}

\def\d{{\rm d}}
\def\l{\langle}
\def\r{\rangle}

\def\vph{\varphi^h}
\def\muh{\mu^h}
\def\uh{\bm{u}^h}
\def\ph{P^h}
\def\yh{y^h}
\def\zh{z^h}

\def\rh{r^h}
\def\vps{\varphi^*}
\def\ps{P^*}

\def\bv{\bm{v}}
\def\mus{\mu^*}
\def\bus{\bm{u}^*}

\def\u{\bm{u}}
\def\uu{\bm{u}}
\def\vv{\bm{v}}

\def\tu{\widetilde{\bm{u}}}

\def\dt{\partial_t}
\def\dn{\partial_{\bf n}}

\def\vp{\varphi}
\def\ve{\varepsilon}

\def\iQt{\int_{Q_t}}
\def\iO{\int_\Omega}

\def\checkmmode #1{\relax\ifmmode\hbox{#1}\else{#1}\fi}
\def\aeO{\checkmmode{a.\,e.\, in~$\Omega$}}
\def\aeQ{\checkmmode{a.\,e.\ in~$Q$}}

\def\aeS{\checkmmode{a.\,e.\ on~$\Sigma$}}

\begin{document}

\title{Optimal distributed control for a Cahn-Hilliard-Darcy system \\
with mass sources, unmatched viscosities and singular potential
\footnote{This is a full length preprint version with detailed computations.}}

\author{
{Marco Abatangelo\thanks{Dipartimento di Matematica,
Universit\`{a} degli Studi di Milano, 20133 Milano, Italy.
E-mail: m.abatangelo@gmail.com}}
\and
{Cecilia Cavaterra\thanks{
Dipartimento di Matematica,
Universit\`{a} degli Studi di Milano, 20133 Milano, Italy.
Istituto di Matematica Applicata e Tecnologie Informatiche ``Enrico Magenes'', CNR,
27100 Pavia, Italy.
E-mail: cecilia.cavaterra@unimi.it}}
\and
{Maurizio Grasselli\thanks{Dipartimento di Matematica,
Politecnico di Milano, 20133 Milano, Italy.
E-mail: maurizio.grasselli@polimi.it}}
\and
Hao Wu\thanks{Corresponding author. School of Mathematical Sciences
and Shanghai Key Laboratory for Contemporary Applied Mathematics, Fudan University, 200433 Shanghai, China.
Email: haowufd@fudan.edu.cn}
}

\date{\today}
\maketitle
%

\begin{abstract}
\noindent
We study a Cahn-Hilliard-Darcy system with mass sources, which can be considered as a basic, though simplified, diffuse interface model for the evolution of tumor growth. This system is equipped with an impermeability condition for the (volume) averaged velocity $\u$ as well as homogeneous Neumann boundary conditions for the phase function $\varphi$ and the chemical potential $\mu$. The source term in the convective Cahn-Hilliard equation contains a control $R$ that can be thought, for instance, as a drug or a nutrient in applications. Our goal is to study a distributed optimal control problem in the two dimensional setting with a cost functional of tracking-type. In the physically relevant case with unmatched viscosities for the binary fluid mixtures and a singular potential, we first prove the existence and uniqueness of a global strong solution with $\varphi$ being strictly separated from the pure phases $\pm 1$. This well-posedness result enables us to characterize the control-to-state mapping $\mathcal{S}:R \mapsto \varphi$. Then we obtain the existence of an optimal control,  the Fr\'{e}chet differentiability of $\mathcal{S}$ and first-order necessary optimality conditions expressed through a suitable variational inequality for the adjoint variables. Finally, we prove the differentiability of the control-to-costate operator and establish a second-order sufficient condition for the strict local optimality.


\end{abstract}

\tableofcontents
\section{Introduction}

Consider a mixture of two immiscible and incompressible fluids contained between two flat parallel plates that are separated by a narrow gap, i.e., a Hele-Shaw cell. The motion of such a binary fluid flow is driven by the pressure and by the capillary forces acting on the free interface separating the two components. A reasonable model for this phenomenon is based on the diffuse-interface approach (see \cite{LLGa,LLG} and references therein). In this framework, the relative concentration difference of the two fluids, say the phase field variable $\vp$, is governed by an advective Cahn-Hilliard equation, while the (volume) averaged fluid velocity $\u$ satisfies Darcy's law that also contains the Korteweg force accounting for capillary effects.
More precisely, assuming that the mixture density is a constant everywhere (set to be one for simplicity) except in the gravitational term, we have the following coupled system:
\be
\begin{cases}
\label{BCHHS}
\nu(\vp) \u =  -\nabla P+ \mu \nabla \vp + \rho(\vp)\bm{g},\\
\mathrm{div}\, \u=0, \\
\vp_{t}+ \mathrm{div}\,(\vp \u) =   \Delta \mu, \\
\mu= - \Delta \vp +  \Psi'(\vp),
\end{cases}
\qquad \text{in}\ \Omega\times (0,T).
\ee
Here, $\Omega\subset \mathbb{R}^2$ is assumed to be a bounded domain with smooth boundary $\partial \Omega$, and $T\in (0,+\infty)$ is a given final time. Besides, $\nu=\nu(\varphi)$ is the kinematic viscosity that may depend on the composition of the mixture and $\rho(\vp)\bm{g}$ denotes the gravity force.
Other physical constants in the system \eqref{BCHHS} have been set equal to one for the sake of simplicity.
We denote by $\mu$ the chemical potential, which is the Fr\'{e}chet derivative of the Ginzburg-Landau free energy functional
$$
\mathcal{E}(\varphi) = \int_\Omega \Big( \frac12 \vert \nabla\vp\vert^2 + \Psi(\vp) \Big)  \d x.
$$
The nonlinear function $\Psi$ represents the physically relevant Flory-Huggins potential \be
\label{SING}
\Psi(s)=\frac{\Theta}{2}\left[ (1+s)\ln(1+s)+(1-s)\ln(1-s)\right]-\frac{\Theta_0}{2} s^2,
 \qquad s \in (-1,1),
\ee
where $\Theta>0$ denotes the absolute temperature of the binary mixture and $\Theta_0>\Theta$ is the critical temperature.
We recall that the singular potential $\Psi$ accounts for the competition between the Gibbs-Boltzmann convex mixing entropy (i.e., the logarithmic terms) and demixing (anti-convex) effects, that is, the essence of phase separation phenomena.
Indeed, when $0<\Theta<\Theta_0$,  $\Psi$ has a double well structure with two minima between the pure phases $-1$ and $1$.

System \eqref{BCHHS}, known as the Boussinesq-Hele-Shaw-Cahn-Hilliard (or, simply, the Hele-Shaw-Cahn-Hilliard) system, can also be obtained
as an approximation of the Navier-Stokes-Cahn-Hilliard system for binary fluids, when the viscous forces dominate the advective-inertial forces (see, for instance, \cite{LLGa,DGL18} and references therein).
The standard boundary and initial conditions for \eqref{BCHHS} are the following
\be
\label{bdini}
\begin{cases}
\u \cdot \n= \partial_\n \mu=\partial_\n \vp=0 &\quad \text{on\ } \partial \Omega \times (0,T),\\
\vp|_{t=0}=\vp_0 &\quad  \text{in } \Omega.
\end{cases}
\ee
Here, $\n=\n(x)$ is the unit outward normal vector to the boundary $\partial \Omega$ and we denote by $\partial_\n$ the outward normal derivative on $\partial \Omega$.

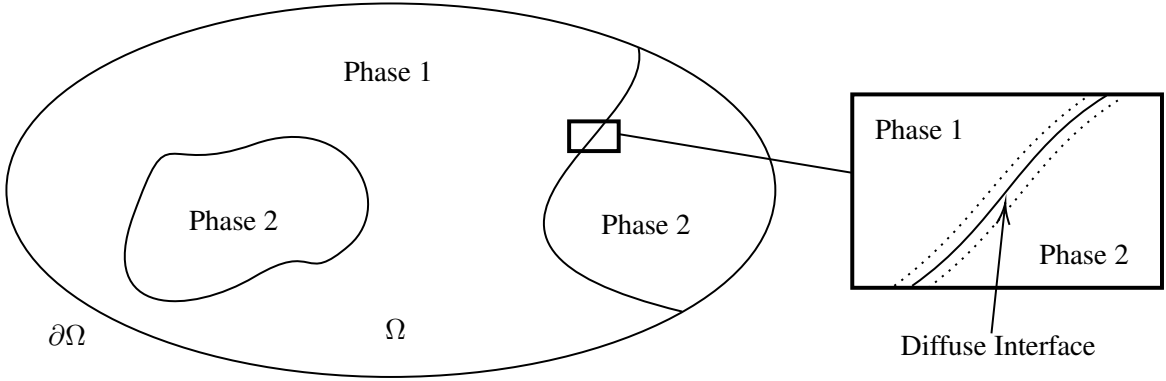
\begin{figure}
  \centering
\tikzset{every picture/.style={line width=0.75pt}} 

\begin{tikzpicture}[x=0.75pt,y=0.75pt,yscale=-1,xscale=1]

\draw   (145.4,95.1) .. controls (193.4,78.1) and (217.4,125.1) .. (193.4,146.1) .. controls (169.4,167.1) and (177.4,143.1) .. (145.4,162.1) .. controls (113.4,181.1) and (65.4,183.1) .. (84.4,131.1) .. controls (103.4,79.1) and (97.4,112.1) .. (145.4,95.1) -- cycle ;
\draw   (20.4,118.36) .. controls (20.4,66.3) and (107.21,24.1) .. (214.3,24.1) .. controls (321.39,24.1) and (408.2,66.3) .. (408.2,118.36) .. controls (408.2,170.41) and (321.39,212.61) .. (214.3,212.61) .. controls (107.21,212.61) and (20.4,170.41) .. (20.4,118.36) -- cycle ;
\draw    (361.2,179.61) .. controls (213.2,143.61) and (348.2,88.61) .. (339.2,46.61) ;
\draw  [line width=1.5]  (304.2,83.61) -- (329.2,83.61) -- (329.2,98.61) -- (304.2,98.61) -- cycle ;
\draw  [line width=1.5]  (447,70) -- (603.2,70) -- (603.2,167.61) -- (447,167.61) -- cycle ;
\draw    (477.2,166.61) .. controls (520.2,136.61) and (531.2,97.61) .. (575.2,70.61) ;
\draw  [dash pattern={on 0.84pt off 2.51pt}]  (485.2,167.61) .. controls (528.2,135.61) and (539.2,98.61) .. (583.2,71.61) ;
\draw  [dash pattern={on 0.84pt off 2.51pt}]  (468.2,166.61) .. controls (511.2,134.61) and (523.2,100.61) .. (566.2,70.61) ;
\draw    (329,90) -- (447.2,109.61) ;
\draw    (517.2,190.61) -- (523.99,126.6) ;
\draw [shift={(524.2,124.61)}, rotate = 96.05] [color={rgb, 255:red, 0; green, 0; blue, 0 }  ][line width=0.75]    (10.93,-3.29) .. controls (6.95,-1.4) and (3.31,-0.3) .. (0,0) .. controls (3.31,0.3) and (6.95,1.4) .. (10.93,3.29)   ;

\draw (189,52) node [anchor=north west][inner sep=0.75pt]   [align=left] {Phase 1};
\draw (470,190) node [anchor=north west][inner sep=0.75pt]   [align=left] {Diffuse Interface};
\draw (210,182) node [anchor=north west][inner sep=0.75pt]   [align=left] {$\displaystyle \Omega $};
\draw (40,185) node [anchor=north west][inner sep=0.75pt]   [align=left] {$\displaystyle \partial \Omega $};
\draw (319,129) node [anchor=north west][inner sep=0.75pt]   [align=left] {Phase 2};
\draw (111,127) node [anchor=north west][inner sep=0.75pt]   [align=left] {Phase 2};
\draw (457,81) node [anchor=north west][inner sep=0.75pt]   [align=left] {Phase 1};
\draw (540,144) node [anchor=north west][inner sep=0.75pt]   [align=left] {Phase 2};

\end{tikzpicture}
  \caption{A diffuse-interface description of two-phase flows in a bounded domain $\Omega\subset \mathbb{R}^2$}
  \label{fig1}
\end{figure}

Neglecting the gravity force $\rho(\vp)\bm{g}$, system \eqref{BCHHS} subject to \eqref{bdini} has recently been analyzed in \cite{Gio2020}, where a nice introduction to the model as well as a detailed story of the related theoretical results can be found (see also \cite{GGW,WW12,WZ13} and references therein and \cite{CFG22,DGG18} for nonlocal models). First of all, the existence of a global weak solution can be obtained by suitably modifying the argument used in \cite{GGW} (see \cite[Theorem 3.2]{Gio2020}). Then, in dimension two, the author proved a conditional uniqueness for weak solutions as well as the existence and uniqueness of a global strong solution, while in dimension three he showed that the initial boundary value problem admits a unique local strong solution (or global if the initial datum is small enough). This is the state-of-the-art of this problem without further simplifications like, e.g., constant kinematic viscosity or a smooth (polynomial) approximation of $\Psi$ (see, e.g., \cite{LTZ,WW12,WZ13}). We recall that in the latter case, uniqueness of global weak solutions in two dimensions has been obtained in \cite[Theorem 4.3]{Gio2020}. On the other hand, it is worth mentioning that a smooth polynomial approximation of $\Psi$ cannot guarantee that the phase function $\vp$ takes its values in the physical range $[-1,1]$ throughout the evolution.

System \eqref{BCHHS} serves as an efficient, simplified model for incompressible
binary fluids moving through a porous medium. The variable viscosity is physically interesting as it is closely related to the well-known Saffman-Taylor instability phenomenon.
When suitable mass sources are taken into account, the system \eqref{BCHHS} can be interpreted as a simplified model of avascular, vascular and metastasis stages of solid
tumour growth (see, e.g., \cite{Fri10,GLSS,Wise08}). In this context, the phase function $\vp$ stands for the difference in volume fractions, where
$\vp=1$ represents the tumor phase and $\vp=-1$ represents the healthy tissue phase. More precisely, we have the following system (neglecting w.l.o.g. the gravitational force)
\be
\begin{cases}
\label{CHHS}
\nu(\vp) \u =  -\nabla P+ \mu \nabla \vp, \\
\mathrm{div}\, \u=S, \\
\vp_{t}+ \mathrm{div}\,(\vp \u) =   \Delta \mu + S + R,  \\
\mu= - \Delta \vp +  \Psi'(\vp),
\end{cases}
\qquad \text{in}\ \Omega \times (0,T).
\ee
 Here, the source terms $S$ and $\widetilde{S} := S +R $ represent possible inter-component mass exchanges as well as gains due
to proliferation of cells and loss due to cell death. This system, also known as the Cahn-Hilliard-Darcy system in the literature, was analyzed in \cite{JWZ} in the easier case that $\nu$ is a positive constant,
$\Psi$ is a double-well polynomial potential, $R = 0$ and $S$ is a given source depending on space and time.
Some well-posedness results were obtained and, in dimension two, the long-time behavior of the global solutions was investigated (i.e., the existence of a minimal pullback attractor and the convergence to a single equilibrium as time goes to infinity).
More recently, under the same assumptions on $\nu$ and $\Psi$, an optimal control problem with respect to the mass source $R$ in dimension two has been studied in \cite{SW21}. The quantity $R$ has been taken as a control variable representing an external mass source (e.g., a drug or a nutrient) that can be supplied to the system to monitor $\vp$ (the size of the tumor).
Given a final time $T>0$, the goal of \cite{SW21} was to analyze the following distributed optimal control problem
$$
\textit{Minimize}\quad \mathcal{J}(\vp,R)
\triangleq
\frac{\alpha_1}{2}\|\vp(T)-\vp_\Omega\|_{L^2(\Omega)}^2
+\frac{\alpha_2}{2}\|\vp-\vp_Q\|_{L^2(Q)}^2
+\frac{\beta}{2}\|R\|^2_{L^2(Q)}
$$
subject to the state system \eqref{bdini}--\eqref{CHHS}, where $R$ belongs to a suitable set $\mathcal{U}_{\mathrm{ad}}$ of admissible controls.
Here, $Q=\Omega \times(0,T)$, $\alpha_1$, $\alpha_2$ and $\beta $ are nonnegative constants (not all identically zero),
 $\vp_\Omega$, $\vp_Q$ denote some prescribed target functions defined in $\Omega$ and $Q$, respectively.
The ratios between the parameters $\alpha_1$, $\alpha_2$ and $\beta$ indicate the importance of the individual targets. In \cite{SW21}, the authors first proved that the optimal control problem admits a solution. Then they showed that the control-to-state operator $\mathcal{S}:R\mapsto\varphi$ is Fr\'echet differentiable between suitable Banach spaces and derived the first-order necessary optimality conditions in terms of the adjoint variables. However, second-order sufficient optimality conditions have not
been derived.

In this study, our aim is to generalize the results of \cite{SW21} on the optimal control problem to the physically relevant case of a variable viscosity depending on $\vp$ and a singular potential like \eqref{SING}. More precisely, we establish the following results:
\begin{itemize}
\item[(1)] The existence of an optimal control (see Theorem \ref{weakcomp});
\item[(2)] The Fr\'{e}chet differentiability of the control-to-state operator $\mathcal{S}$ (see Proposition \ref{1stFD}) and the first-order necessary optimality conditions expressed through a variational inequality for the adjoint variables (see Theorem \ref{NCLO2});
\item[(3)] Differentiability of the control-to-costate operator $\mathcal{T}$ (see Proposition \ref{Diffcotoco}) and a second-order sufficient condition for the strict local optimality (see Theorem \ref{2ndSC}).
\end{itemize}
In order to achieve our goal, a fundamental step is to prove the existence and uniqueness of a global strong solution $(\bm{u},P,\varphi,\mu)$ with $\varphi$ being strictly separated from the pure phases $\pm 1$ (see Theorem \ref{str-well}). The validity of the strict separation property enables us to deal with
the singular potential $\Psi$ and its derivatives, which leads to further regularity properties on the phase function $\varphi$ (cf. \cite{GGG22,GGW} and references therein). These will be crucial to obtain differentiability properties of the associated control-to-state operator $\mathcal{S}$ and the control-to-costate operator $\mathcal{T}$. We note that the second-order analysis is challenging from the mathematical point of view as it  requires that the solution mapping is twice continuously differentiable between suitable Banach spaces.

We conclude the Introduction by mentioning that different choices of the mass source have appeared in recent years in the study of material science, image
processing and biological applications \cite{Lam22,Mi19}. In particular, the Cahn-Hilliard-Darcy system \eqref{CHHS} with source terms depending on $\vp$, suitable boundary conditions and constant viscosity, has recently been studied in \cite{GLRS22,Schi22} (see also \cite{FLRS18,KS22} for multi-species models). A future goal could be the analysis of suitable optimal control problems with
sources depending on $\vp$ (cf. \cite{EK20a,EK20} for a Brinkman version with nutrient) as well as to extend our analysis to nonlocal variants of the system \eqref{CHHS}. Another interesting issue would be the analysis of sparse optimal control problems (see, e.g., \cite{ST21,ST23} and references therein).

\emph{Plan of the paper}. In Section~2, we introduce the functional settings and some analytic tools. In Section 3, we define the control-to-state operator $\mathcal{S}$ and the set of admissible controls. In particular, we state the existence and uniqueness of a strong (and strictly separated) solution, and then derive a continuous dependence estimate for $\mathcal{S}$. The existence of an optimal control is obtained in Section~4. In Section~5, we establish first-order optimality necessary conditions, while in Section~6 we prove a second-order sufficient optimal condition for strict local optimality. Section~7 is an appendix that provides the proof for the existence and uniqueness of a strong solution.


\section{Preliminaries}
\setcounter{equation}{0}

Let $X$ be a real Banach or Hilbert space. Its dual space is indicated by $X'$, and the duality pairing between $X$ and $X'$ is denoted by $\langle \cdot,\cdot\rangle_{X',X}$.
Given an interval $I$ of $\mathbb{R}^+$, we introduce the function space $L^p(I;X)$ with $p\in [1,+\infty]$, which consists of Bochner measurable $p$-integrable functions with
values in $X$. The boldface letter $\mathbf{X}$ denotes the space for vector (or matrix) valued functions.
For the standard Sobolev spaces, we use the notation $W^{k,p} := W^{k,p}(\Omega)$ for any $p \in [1,+\infty]$, $k\in \mathbb{N}$, equipped with the norm $\|\cdot\|_{W^{k,p}(\Omega)}$.
When $k=0$, we denote $W^{0,p}(\Omega)$ by $L^{p}(\Omega)$, while for $p=2$, we denote $W^{m,2}(\Omega)$ by $H^m(\Omega)$.
For convenience, we set $$H_N^2(\Omega)=\{u\in H^2(\Omega)\,:\, \partial_\mathbf{n} u= 0\ \text{a.e. on}\ \partial \Omega\}.$$
For every $f\in (H^1(\Omega))'$, we denote by $\overline{f}$ its generalized mean value over $\Omega$ such that
$$\overline{f}=|\Omega|^{-1}\langle f,1\rangle_{(H^1)',H^1}.$$
If $f\in L^1(\Omega)$, then we simply have $\overline{f}=|\Omega|^{-1}\int_\Omega f \,\mathrm{d}x$. The Poincar\'{e}-Wirtinger inequality gives
\begin{equation}
\label{poincare}
\|f-\overline{f}\|_{L^2(\Omega)}\leq C_P\|\nabla f\|_{L^2(\Omega)},\quad \forall\,
f\in H^1(\Omega),
\end{equation}
where $C_P$ is a positive constant depending only on the spatial dimension and $\Omega$.
Then we introduce the linear spaces
$$
L^2_0(\Omega)=\{ u \in L^2(\Omega):\ \overline{u}=0\}, \quad V_0=\{ u \in H^1(\Omega):\ \overline{u}=0\},\quad
V_0'= \{ u \in (H^1(\Omega))':\ \overline{u}=0 \},
$$
and the linear operator
$A\in \mathcal{L}(H^1(\Omega),(H^1(\Omega))')$ defined by
$$
\l  A u,v \r= \int_\Omega \nabla u \cdot \nabla v \, \d x,
\quad \forall \, u,v  \in H^1(\Omega).
$$
It follows that the restriction of $A$ from $V_0$ onto $V_0'$
is an isomorphism. In particular, $A$ is positively defined on $V_0$
and self-adjoint. We denote its inverse map by $\N =A^{-1}: V_0'
\to V_0$. Note that, for every $f\in V_0'$, $u= \N f \in V_0$ is the unique weak solution to the Neumann problem
$$
\begin{cases}
-\Delta u=f, \quad \text{in} \ \Omega,\\
\partial_\n u=0, \quad \ \  \text{on}\ \partial \Omega.
\end{cases}
$$
For any $f\in V_0'$, we set $\|f\|_{V_0'}=\|\nabla \N f\|_{L^2(\Omega)}$.
According to the Poincar\'{e}-Wirtinger inequality \eqref{poincare}, we
see that  $f\to (\|\nabla f\|_{L^2(\Omega)}^2+|\overline{f}|^2)^\frac12$ is an equivalent norm on $H^1(\Omega)$. Besides, it holds
\begin{align}
\|f\|_{L^2(\Omega)} &\leq \|f\|_{V_0'}^{\frac12} \| \nabla f\|_{L^2(\Omega)}^{\frac12},
\qquad \forall\, f \in V_0. \label{I}
\end{align}

Next, we introduce the Hilbert space for the solenoidal vector field $\u$
$$
\mathbf{H_\sigma}=\{\u\in \mathbf{L}^2(\Omega):
\text{div}\,\u=0\ \text{a.e. in}\ \Omega, \ \ \u\cdot
\mathbf{n}=0 \ \text{a.e. on}\ \partial \Omega\},
$$
endowed with the usual norm $\|\cdot\|_{L^2}$.
Let $\Pi$ be the orthogonal Leray projection from
$\mathbf{L}^2(\Omega)$ onto $\mathbf{H_\sigma}$.
It is well known that every $\uu \in \mathbf{L}^2(\Omega)$
can be uniquely represented as
$
\u= \vv + \nabla P
$ with $\vv= \Pi \,\u \in \mathbf{H_\sigma}$ and $P\in V_0$.  We recall that $\Pi$ is a bounded
operator from $\mathbf{W}^{k,p}(\Omega)$ ($1< p<\infty$, $k\geq 0$) into itself, namely
\begin{equation}
\label{O}
\| \Pi\, \u\|_{W^{k,p}(\Omega)}\leq C \| \u\|_{W^{k,p}(\Omega)},
\qquad \forall\, \u\in \mathbf{W}^{k,p}(\Omega),
\end{equation}
where the positive constant $C$ is independent of $\u$.
Besides, we have the following inequality (see e.g., \cite[Theorem 3.8]{GR})
\be
\label{rot}
\| \u\|_{H^1(\Omega)}\leq
C \left( \| \mathrm{curl}\, \u\|_{L^2(\Omega)}
+\| \u\|_{L^2(\Omega)}\right),
\qquad \forall \,  \u \in \mathbf{H}^1(\Omega)\cap \mathbf{H}_\sigma,
\ee
for some positive constant $C$ independent of $\u$.

In order to handle the mass source term $S$, we recall  the following result on Bogovski's operator (see e.g., \cite[Lemma 2.1.1]{So}):
\bl\label{Bo}
Let $\Omega \subset \mathbb{R}^2$ be a bounded Lipschitz domain, $p\in (1,+\infty)$. For any $g\in L^p(\Omega)$ with $\int_\Omega g\, \mathrm{d}x=0$,
there exists at least one vector $\vv\in \mathbf{W}^{1,p}(\Omega)$ satisfying
$$
\mathrm{div}\,\vv=g \quad \text{a.e. in}\ \Omega, \quad \vv=0\quad \text{a.e. on}\ \partial\Omega.
$$
Moreover, $\|\nabla \vv\|_{L^p(\Omega)}\leq C\|g\|_{L^p(\Omega)}$,
where $C$ is a positive constant only depending on $\Omega$ and $p$.
\el

To estimate the pressure term in Darcy's equation, we recall the following lemma on the homogeneous Neumann problem with a non-constant coefficient (see \cite[Theorem 2.1]{Gio2020}).
\bl\label{press}
Let $\Omega\in\mathbb{R}^2$ be a bounded domain with smooth boundary $\partial\Omega$. Assume that $K\in  C^1(\mathbb{R})$ satisfies $0<\underline{K}\leq K(s)\leq \overline{K}$
for all $s\in \mathbb{R}$, where $\underline{K}$ and $\overline{K}$ are given positive constants. Consider the boundary value problem
\be
\label{Ne}
\begin{cases}
-\mathrm{div}\,(K(\theta)\nabla u)=f,\qquad \, \text{in}\ \Omega, \\
\partial_{\mathbf{n}}u=0,\qquad\qquad \qquad \quad\ \text{on}\ \partial \Omega.
\end{cases}
\ee
Then we have
\begin{itemize}
\item[(1)] Let $\theta$ be a measurable function. For every $f\in V_0'$, there exists a unique weak solution $u\in V_0$ to \eqref{Ne} such that $(K(\theta)\nabla u, \nabla v)=\langle f,v\rangle_{V_0',V_0}$, for all $v\in V_0$.
\item[(2)] Let $\theta\in W^{1,p}(\Omega)$, with $p>2$ and $f\in L_0^2(\Omega)$. Then, $u\in H^2(\Omega)$  and $\partial_{\mathbf{n}}u=0$ almost everywhere on $\partial\Omega$. Moreover,
there exists a positive constant $C$ such that
    $$
    \|u\|_{H^2(\Omega)}\leq C\big(1+\|\theta\|_{W^{1,4}(\Omega)}^2\big)\|f\|_{L^2(\Omega)}.
    $$
\end{itemize}
\el

Next, we report some well-known interpolation inequalities in two dimensions that will be frequently used later.
\begin{itemize}
\item  Ladyzhenskaya's inequality
$$
\|f\|_{L^4(\Omega)}\leq C\|f\|_{L^2(\Omega)}^\frac12\|f\|_{H^1(\Omega)}^\frac12, \qquad \forall\, f\in H^1(\Omega).
$$
\item Agmon's inequality
$$
\|f\|_{L^\infty(\Omega)}\leq C\|f\|_{L^2(\Omega)}^\frac12\|f\|_{H^2(\Omega)}^\frac12, \qquad \forall\, f\in H^2(\Omega).
$$
\item The Br\'{e}zis-Gallouet-Wainger inequality
\begin{align}
\|f\|_{L^\infty(\Omega)}\leq C\big(1+\|f\|_{H^1(\Omega)}\ln^\frac12(e+\|f\|_{W^{1,q}(\Omega)})\big),\qquad \forall\, f\in W^{1,q}(\Omega),\quad q>2.
\label{BGW}
\end{align}
\end{itemize}

Finally, we recall the following Gronwall type lemma (see \cite[Lemma A.1]{Gio2020}):
\bl\label{gron}
Let $f$ be a positive absolutely continuous function on $[0,T]$ and $g, h$ be two summable functions on $[0,T]$ that satisfy the differential inequality
$$
\frac{\d f}{\d t} (t)\leq g(t)f(t)\ln(e+f(t))+h(t),\qquad \text{for a.a.}\ t\in [0,T].
$$
Then we have
$$
f(t)\leq (e+ f(0))^{e^{\int_0^t g(\tau)\d\tau}} e^{\int_0^t e^{\int_\tau^t g(s)\d s}h(\tau)\d\tau},\qquad \forall\, t\in [0,T].
$$
\el

Throughout the paper, we denote by $C$ a generic positive
constant depending only on $\Omega$ and on structural quantities. The constant $C$ may vary from line to line and
even within the same line. Specific dependence will be explicitly pointed out if necessary.
Besides, for the sake of convenience, we set
$$
Q=\Omega \times(0,T),\qquad \Sigma=\partial\Omega\times(0,T),
$$
and
$$
Q_t=\Omega \times(0,t),\qquad \Sigma_t=\partial\Omega\times(0,t),
\qquad \text{for any }\ t\in(0,T).
$$


\section{The control-to-state operator}
\setcounter{equation}{0}
The goal of this section is to define the control-to-state operator. To this end,
we prove the existence and uniqueness of a strictly separated solution and its continuous dependence on the control.

Let us first introduce some basic assumptions on the structure of the state problem \eqref{bdini}--\eqref{CHHS}.
\begin{itemize}
\item[$\mathbf{(A1)}$] The free energy density $\Psi$ can be decomposed into the form
\be
\Psi(s)=F(s)-\frac{\Theta_0}{2}s^2, \qquad \forall \, s \in [-1,1], \label{psiform}
\ee
where the function
$F: [-1,1] \mapsto \mathbb{R}$ satisfies $F\in C([-1,1]) \cap C^{2}(-1,1),$
\begin{align*}
&\lim_{s\rightarrow -1^+}F^{\prime }(s)=-\infty, \quad
\lim_{s\rightarrow 1^-}F^{\prime }(s)=+\infty, \\
&\ \ F^{\prime \prime }(s)\geq \Theta>0,\quad \forall \, s\in (-1,1),
\end{align*}
with the constants $\Theta_0$, $\Theta$ fulfilling
$\alpha:=\Theta_0-\Theta>0$. In addition, there exists $\kappa\in(0,1)$ such that
$F''$ is non-decreasing in $[1-\kappa,1)$ and non-increasing in $(-1,-1+\kappa]$.
Without loss of generality, we assume $F(0)=F'(0)=0$. Moreover, we extend $F(s)=+ \infty$ for all $|s|>1$.
\item[$\mathbf{(A2)}$] The viscosity coefficient $\nu=\nu(s)$ belongs to $C^3(\mathbb{R})$ and satisfies
    $$
    0<\nu_*\leq \nu(s),\ \nu'(s),\ \nu''(s),\ \nu^{(3)}(s)\leq \nu^*,\qquad \forall\, s \in \mathbb{R},
    $$
    where $\nu_*$ and $\nu^*$ are two given positive constants.
\item[$\mathbf{(A3)}$] The external mass source terms satisfy
$$
S\in L^2(0,T;V_0)\cap H^1(0,T;L^2(\Omega)),\quad R\in L^2(0,T;H^1(\Omega))\cap H^1(0,T;(H^1(\Omega))').
$$
\end{itemize}

\br
\label{remarkF}
The assumption $\mathbf{(A1)}$ is fulfilled in the physically relevant case  \eqref{SING} with
\begin{align}
F(s)=\frac{\Theta}{2}\left[ (1+s)\ln(1+s)+(1-s)\ln(1-s)\right],
\qquad \forall \,s \in [-1,1].\label{SINGSS}
\end{align}
Note that the case $\Theta_0-\Theta\leq 0$ is easier, because the potential $\Psi$ is now  convex and we can simply consider $\Psi$ without the decomposition \eqref{psiform}. However, no phase separation takes place in this case.
\er

\subsection{Weak and strong solutions}
\label{exeweak}

We first state a preliminary result on the existence of a global weak solutions to problem \eqref{bdini}--\eqref{CHHS}, which also holds in dimension three.

\begin{proposition}[Weak solutions in dimension two and three]\label{glo-weak}
Let $\Omega\subset \mathbb{R}^d$ ($d\in\{2,3\}$) be a bounded domain with smooth boundary $\partial\Omega$ and $T>0$. Assume that $\mathbf{(A1)}$--$\mathbf{(A2)}$ are satisfied.
For any initial datum $\varphi_0\in H^1(\Omega)$ with $F(\varphi_0)\in L^1(\Omega)$,  $|\overline{\varphi_0}|<1$, and any mass source terms
$S\in L^2(0,T; L^2_0(\Omega))$, $R\in L^\infty(0,T;L^2(\Omega))$ with the following constraint
\begin{equation}
\left|\overline{\vp_0}+\int_0^t\overline{R}(\tau)\,\d\tau\right|\leq 1-\delta_0,\quad \forall\, t\in [0,T],
\label{AR-1}
\end{equation}
for some $\delta_0\in (0,1)$, problem \eqref{bdini}--\eqref{CHHS} admits at least one weak solution $(\u, P, \varphi,\mu)$ on $[0,T]$ in the following sense:

(1) The solution $(\u, P, \varphi,\mu)$ fulfills the regularity properties
\begin{align*}
& \u\in L^2(0,T; \mathbf{L}^2(\Omega))\cap L^s(0,T; \mathbf{H}^1(\Omega)),\quad P\in L^{q_1}(0,T;V_0)\cap L^{q_2}(0,T;H^2(\Omega)),\\
& \varphi\in C([0,T];H^1(\Omega))\cap L^4(0,T;H^2(\Omega))\cap L^2(0,T;W^{2,p}(\Omega))\cap H^1(0,T;(H^1(\Omega))'),\\
& \varphi\in L^\infty(\Omega \times (0,T))\quad \text{with}\ |\varphi(x,t)|<1\ \text{a.e. in}\ \Omega \times (0,T),\\
& \mu\in L^2(0,T;H^1(\Omega)), \quad \Psi'(\varphi)\in L^2(0,T;L^p(\Omega)),
\end{align*}
where $(s,q_1,q_2, p) = (6/5, 8/5, 8/7, 6)$ if $d=3$, $(s,q_1,q_2, p) \in   [1,4/3) \times  [1,2) \times [1,6/5)\times [2,+\infty)$ if $d=2$.

(2) The solution $(\u, P, \varphi,\mu)$ satisfies
\begin{align*}
& \nu(\vp) \u =  -\nabla P+ \mu \nabla \vp, & & \text{a.e. in}\ Q,\\
&\mathrm{div}\, \u=S, & & \text{a.e. in}\ Q,\\
&\langle\vp_{t}, \psi\rangle_{(H^1)',H^1}- (\vp \u, \nabla \psi)+ (\nabla \mu, \nabla \psi) = (S+R, \psi), & & \forall\, \psi\in H^1(\Omega),\ \ \text{for a.a.}\ t\in (0,T),  \\
& \mu= - \Delta \vp +  \Psi'(\vp),& & \text{a.e. in}\ Q.
\end{align*}
Moreover, $\u\cdot
\mathbf{n}= \partial_{\mathbf{n}}\varphi=0$ almost everywhere on $\Sigma$ and $\varphi(\cdot,0)=\varphi_0$ almost everywhere in $\Omega$.

(3) The solution $(\u, P, \varphi,\mu)$ satisfies the energy identity
$$
\frac{\d }{\d t} \mathcal{E}(\varphi) + \|\sqrt{\nu(\varphi)}\u\|_{L^2(\Omega)}^2+\|\nabla \mu\|_{L^2(\Omega)}^2=\int_\Omega S[P+(1-\varphi)\mu]\, \d x+ \int_\Omega R\mu\, \d x,
$$
for almost all $t\in (0,T)$.

\end{proposition}

\begin{remark}
Proposition \ref{glo-weak} can be proved by combining the arguments in \cite{MA2020,Gio2020,GGW,JWZ}. When $S=R=0$, the authors of \cite{GGW} established the existence of global
weak solutions in both two and three dimensions for the case with a constant viscosity and a singular potential like in $\mathbf{(A1)}$. Later in \cite{Gio2020}, the result was extended
to the case with a variable viscosity satisfying $\mathbf{(A2)}$. On the other hand, in \cite{JWZ} the authors analyzed problem \eqref{bdini}--\eqref{CHHS} with a
constant viscosity, a regular potential $\Psi(s)=(1/4)(s^2-1)^2$ and external mass sources $S\neq 0$, $R=0$. Recently, the author of \cite{MA2020} further considered the case with a constant
 viscosity and a singular potential $\Psi$ satisfying $\mathbf{(A1)}$ as well as nonzero mass sources like in Proposition \ref{glo-weak}.
\end{remark}

Here, we will not present the detailed proof of Proposition \ref{glo-weak} but only its sketch.  First, we introduce a family of regular potentials
$\left\lbrace \Psi_\varepsilon\right\rbrace$ that suitably approximates the singular potential $\Psi$ as in \cite[Section 3]{GGW}. Then we establish well-posedness of the
approximating problem with the regular potentials $\Psi_\varepsilon$, by means of the Faedo-Galerkin method (see \cite[Section 3]{JWZ}), where the pressure can be handled as in
\cite[Remark 3.4]{Gio2020}. At last, for the approximate solutions $(\u_\varepsilon, P_\varepsilon, \vp_\varepsilon, \mu_\varepsilon)$ related to the family of regular potentials
$\left\lbrace \Psi_{\varepsilon}\right\rbrace$, we recover compactness by means of uniform energy estimates with respect to the approximation parameter $\varepsilon$ and show
that as $\varepsilon\to 0^+$, the limit quadruplet $(\u, P, \vp,\mu)$ is indeed a weak solution to problem \eqref{bdini}--\eqref{CHHS} on $[0,T]$ (cf. \cite[Section 3]{GGW} and in particular,
\cite[Section 3]{MA2020} for further details).

Unfortunately, the regularity of  weak solutions is not enough to study the optimal control problem (in particular, the differentiability of the state-to-control operator etc).
Moreover, because of the variable viscosity, uniqueness of weak solutions obtained in Proposition \ref{glo-weak} remains an open problem even in two dimensions
(cf. \cite[Section 4]{Gio2020}). On the other hand, due to difficulties from Darcy's equation, the existence and uniqueness of a global strong solution to problem \eqref{bdini}--\eqref{CHHS}
for arbitrary regular initial data are still out of reach in three dimensions (cf. \cite{GGW,Gio2020,WZ13}).

Hence, we confine ourselves to the two dimensional setting so that we are able to establish the strong well-posedness of problem \eqref{bdini}--\eqref{CHHS}.
For this purpose, we impose the following additional assumptions on the singular potential $\Psi$ (cf. \cite{GGW,Gio2020}):

\begin{itemize}
\item[$\mathbf{(A1)'}$] The function
$F: [-1,1] \mapsto \mathbb{R}$ satisfies $F\in C([-1,1]) \cap C^{5}(-1,1)$ and there exists $\kappa\in (0,1)$ such that
$$
F^{(3)}(s)s\geq 0,\quad F^{(4)}(s)>0, \quad \forall\, s\in (-1,-1+\kappa]\cup [1-\kappa, 1).
$$
Besides, it holds
\begin{equation}
|F''(s)|\leq Ce^{C|F'(s)|},\qquad \forall\, s\in (-1,1),
\label{SING-F}
\end{equation}
where $C$ is a positive constant independent of $s$.
\end{itemize}
\begin{remark}
For simplicity, we take the parameter $\kappa$ in $\mathbf{(A1)}$ and  $\mathbf{(A1)'}$ to be the same. It is easy to verify that the Gibbs-Boltzmann mixing entropy \eqref{SINGSS}
also satisfies $\mathbf{(A1)'}$.
\end{remark}

We can prove the following strong well-posedness result.

\begin{theorem}[Strong solutions in dimension two]
\label{str-well}
Let $\Omega\subset \mathbb{R}^2$ be a bounded domain with smooth boundary $\partial\Omega$ and $T>0$. Assume that
$\mathbf{(A1)}$, $\mathbf{(A1)'}$ and $\mathbf{(A2)}$  are satisfied. If $\varphi_0\in H_N^2(\Omega)$ is such that $\overline{\varphi_0}\in (-1,1)$, $\widetilde{\mu}_0=-\Delta\varphi_0+F'(\varphi_0)\in H^1(\Omega)$, and the mass sources satisfy $\mathbf{(A3)}$
with the constraint \eqref{AR-1}, then problem \eqref{bdini}--\eqref{CHHS}  admits a unique strong solution $(\u, P, \varphi,\mu)$ on $[0,T]$ such that
\begin{align*}
& \u\in L^\infty(0,T;\mathbf{H}^1(\Omega)),\quad P\in L^\infty(0,T;H^2(\Omega)\cap L^2_0(\Omega)),\\
& \varphi \in C([0,T]; H^3(\Omega))\cap L^2(0,T;H^5(\Omega))\cap H^1(0,T; H^1(\Omega)),\\
& \mu\in C([0,T];H^1(\Omega))\cap L^2(0,T; H^3(\Omega))\cap H^1(0,T; (H^1(\Omega))'),\\
& \Psi''(\varphi)\in L^\infty(0,T;L^p(\Omega)),
\end{align*}
for any $p\in [2,+\infty)$. The strong solution satisfies the system \eqref{CHHS} almost everywhere in $Q$, while $\u \cdot \n= \partial_\n \mu=\partial_\n \vp=0$ a.e. on $\Sigma$ and
$\varphi(\cdot,0)=\varphi_0$ in $\Omega$.  Moreover, there exists some constant $\delta_1\in (0,1)$ such that
\begin{align}
\|\varphi(t)\|_{C(\overline{\Omega})}\leq 1-\delta_1,\qquad \forall\, t\in [0,T]. \label{sep1}
\end{align}
\end{theorem}

\begin{remark}\label{regP}
 The strict separation property \eqref{sep1} plays a crucial role in the subsequent analysis, since it yields that the phase function $\varphi$ stays away from the pure phases
 $\pm 1$ in the whole interval $[0,T]$. As a direct consequence, the potential $\Psi$ is no longer singular along the evolution for strong solutions and
 $\max_{0\le j\leq 5}\,\left\|\Psi^{(j)}(\vp)\right\|_{C(\overline{Q})}$ is bounded.
\end{remark}

\begin{remark}
The growth condition \eqref{SING-F} in $\mathbf{(A1)'}$ allows us to derive a nonlinear estimate involving $e^{C|F'(\varphi)|}$ and then apply the Trudinger-Moser inequality (in two dimensions) to show the validity of the instantaneous strict separation property \eqref{sep1}, see Lemma \ref{esing} below, cf. also \cite{GGM2017,GGW}.
Concerning the single Cahn-Hilliard equation, the same strict separation property for $\vp$ was established in \cite{GGG22}, with a more direct proof, for a larger class of singular potentials satisfying the growth condition
\begin{equation}
|F''(s)|\leq Ce^{C|F'(s)|^\sigma},\quad \forall\, s\in (-1,1),\quad \sigma\in[1,2),
\label{SING-F1}
\end{equation}
where $C$ is a positive constant independent of $s$.
Further progress in this direction was made in the recent work\cite{GP23}, where the authors proposed a weaker assumption
such that, as $\delta\to 0^+$,
\begin{equation}
\frac{1}{F'(1-2\delta)}=O\left(\frac{1}{|\ln(\delta)|^\sigma}\right),\quad \frac{1}{F'(-1+2\delta)}=O\left(\frac{1}{|\ln(\delta)|^\sigma}\right), \quad \text{for some}\ \sigma>1/2.
\label{SING-F2}
\end{equation}
The new condition \eqref{SING-F2} relaxes previous assumptions on the entropy function, since it only concerns asymptotic behavior of its first derivative $F'$ near the endpoints $\pm 1$, but does not involve any
pointwise relations between $F'$ and $F''$ as in \eqref{SING-F} and \eqref{SING-F1}.
Then the strict separation property for $\varphi$ was obtained by a different method based on De Giorgi's iteration scheme applied to the elliptic equation $\mu= - \Delta \vp +  \Psi'(\vp)$, see \cite[Theorem 3.1]{GP23} for details.
It is easy to check that the assumption \eqref{SING-F2} is satisfied by the logarithmic entropy \eqref{SINGSS} and indeed accommodates entropy densities with milder singularities. Exploiting the argument for \cite[Theorem 3.1]{GP23}, we can replace the assumption \eqref{SING-F} in Theorem \ref{str-well} by \eqref{SING-F2} and recover the strict separation property \eqref{sep1}. As a consequence, all results obtained in this paper are valid under the weaker assumption \eqref{SING-F2} as well. We leave the details to interested readers.
\end{remark}

We postpone the lengthy proof of Theorem \ref{str-well} to Section~7. On the other hand, a crucial ingredient for studying the optimal control problem is the dependence of  solutions to
problem \eqref{bdini}--\eqref{CHHS} on the control $R$. In this respect, we have

\begin{proposition}[Continuous dependence in dimension two]
\label{conti-es1}
Let $(\u_i, P_i, \varphi_i,\mu_i)$, $i=1,2$, be two different strong solutions to problem \eqref{bdini}--\eqref{CHHS} associated with the initial datum
$\varphi_{0}$ and the mass sources $S$, $R_i$ ($i=1,2$) satisfying the assumptions of Theorem \ref{str-well}. Then there exists a constant $C>0$, which depends only on
$\|\varphi_{0}\|_{H^2(\Omega)}$, $\|\widetilde{\mu}_{0}\|_{H^1}$, $\Omega$, $T$, $\|S\|_{L^2(0,T;H^1(\Omega))\cap H^1(0,T;L^2(\Omega))}$ and $\|R_i\|_{L^2(0,T;H^1(\Omega))
\cap H^1(0,T;(H^1(\Omega))')}$, such that the following estimate holds:
\begin{align}\label{stabu1}
&\|\vp_1-\vp_2\|_{H^1(0,t;L^2(\Omega))\cap C([0,t];H^2(\Omega))\cap L^2(0,t;H^4(\Omega))}\,+\,\|\mu_1-\mu_2\|_{L^2(0,t;H^2(\Omega))}\non\\
&\qquad +\|\u_1-\u_2\|
_{L^2(0,t;L^4(\Omega))}+\,\|P_1-P_2\|_{L^2(0,t;W^{2,4/3}(\Omega))}\non\\
&\quad \le C \|R_1-R_2\|_{L^2(0,t;L^2(\Omega))},\qquad \forall\, t\in (0,T].
\end{align}
\end{proposition}
\begin{proof}
Based on Theorem \ref{str-well} and Remark \ref{regP}, we can prove Proposition \ref{conti-es1} by applying an argument similar to that in \cite[Lemma 2.3]{SW21}. Below we sketch the main steps and present necessary modifications due to the variable viscosity in \eqref{CHHS}. Define
 $$R=R_1-R_2,\quad \varphi=\varphi_1-\varphi_2,\quad \mu=\mu_1-\mu_2,\quad \u=\u_1-\u_2,\quad P=P_1-P_2.$$
It follows that the above differences satisfy
\begin{align}
\label{diff3}
\nu(\vp_1)\u\,&=\,-\nabla P+\mu\,\nabla\vp_1+\mu_2\,\nabla\vp-(\nu(\vp_1)-\nu(\vp_2))\u_2, & &\aeQ,\\[1mm]
\label{diff4}
{\rm div}\,\u\,&=\,0, & &\aeQ,\\[1mm]
\label{diff1}
\partial_t\vp-\Delta\mu\,&=\,R-S\vp-\u\cdot\nabla\vp_1 -\u_2\cdot\nabla\vp, & &\aeQ,\\[1mm]
\label{diff2}
\mu\,&=\,-\Delta\vp+\Psi'(\vp_1)-\Psi'(\vp_2), & &\aeQ,\\[1mm]
\label{diff6}
\dn\vp\,&=\,\dn\mu\,=\,\u\cdot{\bf n}\,=\,0, & &\aeS,\\[1mm]
\label{diff7}
\vp|_{t=0}\,&=\,0, & &\aeO.
\end{align}
Moreover, we have
\begin{align}
-\Delta P & =\nu'(\vp_1)\nabla \vp_1\cdot \u + (\nu(\vp_1)-\nu(\vp_2)) S +  \nabla  (\nu(\vp_1)-\nu(\vp_2))\cdot \u_2  \notag \\
&\quad   -\,{\rm div}(\mu\,\nabla\vp_1)-{\rm div}(\mu_2\,\nabla\vp), \qquad \aeQ,\label{diff5}
\end{align}
with the boundary condition $\dn P=0$ almost everywhere on $\Sigma$ and  $\overline{P(t)}=0$ for almost all $t\in (0,T)$.

For any $t\in (0,T]$, following the arguments for \cite[(2.21),\,(2.26)]{SW21} and using $\mathbf{(A2)}$, we easily get
\begin{equation}
\label{diff15}
\int_0^t\|\mu(s)\|_{H^1(\Omega)}^2\,\d s\,\le\,C_1\int_0^t(\|\nabla\mu(s)\|_{L^2(\Omega)}^2 + \|\vp(s)\|_{L^2(\Omega)}^2)\,\d s,
\end{equation}
and
\begin{align}
&\frac 12\,\|\vp(t)\|_{L^2(\Omega)}^2 +\left(1-\gamma\right)
\int_0^t\|\Delta\vp(s)\|_{L^2(\Omega)}^2\,\d s \notag\\
&\quad \leq  \frac{\nu_*}{8} \,\int_0^t \|\u(s)\|_{L^2(\Omega)}^2\,\d s\,+\,C\int_0^t\|R(s)\|_{L^2(\Omega)}^2\, \d s\notag\\
&\qquad +\,C\,(1+\gamma^{-1})\int_0^t \left(1+\,\|\u_2(s)\|_{L^4(\Omega)} + \|S(s)\|_{L^2(\Omega)}\right) \|\vp(s)\|_{H^1(\Omega)}^2\,\d s,
\non
\end{align}
for any $\gamma\in (0,1)$. Next, multiplying \eqref{diff1} by $\mu$ and \eqref{diff3} by $\u$, adding the two resulting identities and integrating over $Q_t$, we obtain
\begin{align}
\label{diff16}
&\frac 12\,\|\nabla\vp(t)\|_{L^2(\Omega)}^2\,+\int_0^t\|\nabla\mu\|_{L^2(\Omega)}^2\,\d s\,+ \int_0^t\nu(\vp_1)|\u|^2\,\d x\,\d s\notag\\
&\quad =  \int_0^t R \mu \,\d x\,\d s\,+\iQt \mu_2\,(\u \cdot\nabla\vp) \,\d x\d s -\iQt\!\mu(S\,\vp\,+\,\u_2\cdot\nabla\vp)\,\d x\,\d s \non\\
&\qquad -\iQt(\Psi'(\vp_1)-\Psi'(\vp_2))\,\dt\vp\,\d x\,\d s
-\iQt (\nu(\vp_1)-\nu(\vp_2))\u_2\cdot \u \,\d x\,\d s.
\end{align}
Thanks to $\mathbf{(A2)}$, the fifth term on the right-hand side of \eqref{diff16} can be estimated as follows
\begin{align}
& \left|\iQt (\nu(\vp_1)-\nu(\vp_2))\u_2\cdot \u \,\d x\,\d s\right| \notag \\
&\quad \leq \iQt \left|\int_0^1 \nu'(\tau \vp_1+(1-\tau)\vp_2) \vp \,\d\tau\right| |\u_2||\u| \,\d x\,\d s \notag \\
&\quad  \leq \nu^* \int_0^t\|\vp(s)\|_{L^4(\Omega)}\|\u_2(s)\|_{L^4(\Omega)} \|\u(s)\|_{L^2(\Omega)}\,\d s\notag\\
&\quad \leq \frac{\nu_*}{8} \,\int_0^t \|\u(s)\|_{L^2(\Omega)}^2\,\d s
+ C \int_0^t \|\u_2(s)\|_{L^4(\Omega)}^2\|\vp(s)\|^2_{H^1(\Omega)}\,\d s.\non
\end{align}
The other four terms can be treated similarly as in \cite[(2.28)--(2.36)]{SW21} by using the strict separation property \eqref{sep1} and Remark \ref{regP}.
Combining these estimates and \eqref{diff15}--\eqref{diff16}, taking $\gamma\in (0,1)$ to be sufficiently small, we arrive at
\begin{align*}
&\frac 12 \,\|\vp(t)\|^2_{H^1(\Omega)} \,+\,
 \frac{\nu_*}{4}\int_0^t\|\u\|_{L^2(\Omega)}^2\,\d s+\frac12 \int_0^t\|\Delta\vp\|_{L^2(\Omega)}^2\,\d s
  + \frac12 \int_0^t\|\nabla\mu\|_{L^2(\Omega)}^2\,\d s \notag\\
&\quad \le\,C \int_0^t\|R(s)\|_{L^2(\Omega)}^2\,\d s\,+\,C \int_0^t G_1(s)\,\|\vp(s)\|^2_{H^1(\Omega)}\,\d s,
\end{align*}
where
$$G_1(\cdot)=1+\|\u_2(\cdot)\|_{L^4(\Omega)}^2 +\|\mu_2(\cdot)\|_{H^1(\Omega)}^2+\|S(\cdot)\|_{L^2(\Omega)}^2\in L^1(0,T).$$
Using standard elliptic estimates and \eqref{diff15}, we deduce from Gronwall's lemma that for any $t\in[0,T]$, it holds
\begin{align}
\label{diff26}
&\|\vp\|_{C([0,t];H^1(\Omega))\cap L^2(0,t;H^2(\Omega))}^2\,+\,\|\mu\|_{L^2(0,t;H^1(\Omega))}^2 \,+\,\|\u\|_{L^2(0,t;L^2(\Omega))}^2 \non\\
&\quad  \leq\,C \int_0^t\|R(s)\|_{L^2(\Omega)}^2\,\d s.
\end{align}
Next, from the Sobolev embedding theorem, \eqref{diff5} and the elliptic estimate for the Neumann problem, we further have
\begin{align}
\label{stime3}
\|\nabla P\|_{L^2(0,t;L^4(\Omega))}\,
& \le\,C\,
\|P\|_{L^2(0,t;W^{2,4/3}(\Omega))}\non\\
& \leq \, C\|\nu'(\vp_1)\nabla \vp_1\cdot \u\|_{L^2(0,t;L^{4/3}(\Omega))} + C\|(\nu(\vp_1)-\nu(\vp_2)) S\|_{L^2(0,t;L^{4/3}(\Omega))} \notag\\
& \quad +
C\| \nabla  (\nu(\vp_1)-\nu(\vp_2))\cdot \u_2\|_{L^2(0,t;L^{4/3}(\Omega))} \notag\\
& \quad + C\, \|\nabla \mu\cdot \nabla\vp_1+\mu_2\Delta \vp\|_{L^2(0,t;L^{4/3}(\Omega))}\notag\\
& \quad + C\, \|\nabla \mu\cdot \nabla\vp_1+\mu_2\Delta\vp\|_{L^2(0,t;L^{4/3}(\Omega))}.
\end{align}
The last two terms on the right-hand side of \eqref{stime3} can be estimated as in \cite[(2.39), (2.40)]{SW21}, that is,
\begin{align}
\int_0^t \left(\|\nabla \mu\cdot \nabla\vp_1+\mu_2\Delta \vp\|_{L^{4/3}(\Omega)}^2
+  \|\nabla \mu\cdot \nabla\vp_1+\mu_2\Delta\vp\|_{L^{4/3}(\Omega)}^2\right)\,\d s\leq\,C \int_0^t\|R(s)\|_{L^2(\Omega)}^2\,\d s.\notag
\end{align}
Besides, from \eqref{diff26} we can deduce that
\begin{align}
\int_0^t\|\nu'(\vp_1)\nabla \vp_1\cdot \u\|_{L^{4/3}(\Omega)}^2\,\d s
\leq \nu^*\int_0^t \|\nabla \vp_1\|_{L^4}^2\|\u\|^2\,\d s\leq C\int_0^t\|R(s)\|_{L^2(\Omega)}^2\,\d s,\notag
\end{align}
\begin{align}
&\int_0^t \|(\nu(\vp_1)-\nu(\vp_2)) S\|_{L^{4/3}(\Omega)}^2\,\d s\non\\
&\quad \leq \|S\|_{C([0,t];L^2(\Omega))}^2
\int_0^t\left\|\int_0^1 \nu'(\tau\vp_1+(1-\tau)\vp_2)\vp\,\d\tau\right\|_{L^4(\Omega)}^2\d s\notag\\
&\quad \leq C\int_0^t\|\vp(s)\|_{L^4(\Omega)}^2\,\d s
\leq C\int_0^t\|R(s)\|_{L^2(\Omega)}^2\,\d s,\notag
\end{align}
and
\begin{align}
& \int_0^t\| \nabla  (\nu(\vp_1)-\nu(\vp_2))\cdot \u_2\|_{L^{4/3}(\Omega)}^2\,\d s\notag\\
&\quad \leq \|\u_2\|_{L^\infty(0,t;L^4(\Omega))}^2
\int_0^t \|\nu'(\phi_1)\nabla \vp\|_{L^2(\Omega)}^2\,\d s\notag\\
&\qquad + \|\u_2\|_{L^\infty(0,t;L^4(\Omega))}^2
\int_0^t \left\|\int_0^1\nu''(\tau\vp_1+(1-\tau)\vp_2)\vp\,\d\tau\nabla \vp_2\right\|_{L^2(\Omega)}^2\,\d s\notag\\
&\quad \leq C\int_0^t \|\vp(s)\|_{H^1(\Omega)}^2\,\d s\leq C\int_0^t\|R(s)\|_{L^2(\Omega)}^2\,\d s.\notag
\end{align}
Collecting the above estimates, we find
\begin{align}
\int_0^t \|P(s)\|_{W^{2,4/3}(\Omega)}^2\,\d s\leq C\int_0^t\|R(s)\|_{L^2(\Omega)}^2\,\d s,\non
\end{align}
which combined with $\mathbf{(A2)}$, \eqref{diff3} and \eqref{diff26} yields
\begin{align}
\int_0^t \|\u(s)\|_{L^4(\Omega)}^2\,\d s\leq C\int_0^t\|R(s)\|_{L^2(\Omega)}^2\,\d s. \non
\end{align}
Finally, using the same argument as in \cite[pp. 502--503]{SW21}, we can derive the following inequality
\begin{align*}
& \frac 12\|\Delta\vp(t)\|_{L^2(\Omega)}^2\,+\,\frac 12\int_0^t\|\Delta^2\vp(s)\|_{L^2(\Omega)}^2\, \d s\\
&\quad \leq C\int_0^t G_2(s)\|\Delta \vp(s)\|_{L^2(\Omega)}^2\, ds + C\int_0^t\|R(s)\|_{L^2(\Omega)}^2\, \d s,\non
\end{align*}
where
$$
G_2(\cdot)=1+\|S(\cdot)\|_{L^2(\Omega)}^2+ \|\u_2(\cdot)\|_{L^4(\Omega)}^2\in L^1(0,T).
$$
Thus, an application of Gronwall's lemma yields
\begin{align}
\|\Delta\vp\|_{C([0,t]; L^2(\Omega))}^2+ \|\Delta^2 \vp\|_{L^2(0,t; L^2(\Omega))}^2\leq C\int_0^t\|R(s)\|_{L^2(\Omega)}^2\, \d s,\quad \forall\, t\in [0,T].\non
\end{align}
The remaining part of \eqref{stabu1} follows from the same arguments as for \cite[(2.46)--(2.48)]{SW21}.
\end{proof}

\subsection{Admissible controls and the control-to-state operator}

On account of Theorem \ref{str-well}, we can specify the settings for the optimal control problem.

\begin{itemize}
\item[$\mathbf{(S1)}$] The potential function $\Psi$ satisfies $\mathbf{(A1)}$ and $\mathbf{(A1)'}$.
\item[$\mathbf{(S2)}$] The viscosity coefficient $\nu$ satisfies $\mathbf{(A2)}$, the external mass source $S$ satisfies $\mathbf{(A3)}$.
\item[$\mathbf{(S3)}$] The initial datum $\vp_0$ satisfies $\vp_0\in H_N^2(\Omega)$ and $\widetilde{\mu}_0=-\Delta\varphi_0+F'(\varphi_0)\in H^1(\Omega)$.
\end{itemize}
\begin{remark}\label{iniseprem}
When the spatial dimension is two, we infer from $\mathbf{(S3)}$ and Lemma \ref{esing} (see Appendix) that $\vp_0\in H^3(\Omega)$ and
$\|\varphi_0\|_{L^{\infty}} \leq 1-\widetilde{\delta}$ for some $\widetilde{\delta}\in(0,1)$ depending on $\|\widetilde{\mu}_0\|_{H^1}$, $\Omega$ and $F$.
As a consequence, it holds $|\overline{\vp_0}|\leq 1-\widetilde{\delta}<1$.
\end{remark}

Let us introduce the space
$$
\mathcal{U}=\big\{R(x,t)\,|\, R\in L^2(0,T;H^1(\Omega))\cap H^1(0,T;(H^1(\Omega))')\big\},
$$
whose (equivalent) norm is given by
$$
\|R\|_{\mathcal{U}}=\|R\|_{L^2(0,T;H^1(\Omega))}+\|\partial_t R\|_{L^2(0,T;(H^1(\Omega))')}.
$$
Thanks to Remark \ref{iniseprem}, for any given initial datum $\varphi_0$ satisfying $\mathbf{(S3)}$, we can fix some positive constant
$$\delta_0\in \left(0,\,\frac12\left(1-|\overline{\vp_0}|\right)\right)$$
and set
\begin{equation}
r_0=\left(1-2\delta_0-|\overline{\vp_0}|\right)|\Omega|>0,\qquad \widetilde{r}_0=r_0+\delta_0|\Omega|>0.
\label{r0}
\end{equation}
Then the set of admissible controls is defined as follows:
\begin{definition}
Let $\varphi_0$ be an initial datum satisfying $\mathbf{(S3)}$.
The set
\begin{align*}
\mathcal{U}_{\mathrm{ad}}
& =\big\{R\in \mathcal{U}\cap L^\infty(Q)\ |\  \|R\|_{\mathcal{U}}\leq r_1,\,\|R\|_{L^1(Q)}\leq r_0, \ R_\mathrm{min}\leq R\leq R_\mathrm{max} \ \  \text{a.e. in}\ Q\big\}
\end{align*}
is referred to as \textbf{the set of admissible controls} for problem \eqref{bdini}--\eqref{CHHS}, where $r_1\in (0,+\infty)$ and $r_0$ is given by \eqref{r0}.
Here, $R_\mathrm{min}, R_\mathrm{max} \in L^\infty(Q)$ are  two given functions with $R_\mathrm{min}\leq R_\mathrm{max} $ almost everywhere in $Q$.
\end{definition}
 Obviously, $\mathcal{U}_{\mathrm{ad}} $ is a bounded, convex and closed subset of the Banach space $\mathcal{U}\subset L^2(Q)$. Throughout the paper, we assume that
 $\mathcal{U}_{\mathrm{ad}} $ is nonempty.
 Moreover, in order to derive optimality conditions for the control problem, we need to properly enlarge $\mathcal{U}_{\mathrm{ad}}$ and consider the open set
$$
\widetilde{\mathcal{U}} =\big\{R\in \mathcal{U}\cap L^\infty(Q)\ |\  \|R\|_{\mathcal{U}}< 2 r_1,\,\|R\|_{L^1(Q)}< \widetilde{r}_0\big\} \quad \text{such that}\quad
\mathcal{U}_{\mathrm{ad}}\subset \widetilde{\mathcal{U}}.
$$

Thanks to Theorem \ref{str-well} and Proposition \ref{conti-es1}, we find that problem \eqref{bdini}--\eqref{CHHS} admits a unique strong solution
$(\bm{u},P,\vp,\mu)$ for every $R\in \widetilde{\mathcal{U}}$ (not only for admissible controls). This enables us to define the control-to-state operator that maps
the control function $R$ to its \textbf{associated state} $(\bm{u},P,\vp,\mu)$. Set the function spaces
\begin{align}
&\mathcal{X}:= C([0,T];H^3(\Omega))\cap L^2(0,T;H^5(\Omega))\cap H^1(0,T;H^1(\Omega)),\non\\
&\mathcal{Y}:= C([0,T];H^2(\Omega))\cap L^2(0,T;H^4(\Omega))\cap H^1(0,T;L^2(\Omega)).\non
\end{align}
Then we have
\begin{proposition}\label{StoCo}
Assume that $\Omega\subset \mathbb{R}^2$ is a bounded domain with smooth boundary $\partial\Omega$ and $T>0$. Let the assumptions $\mathbf{(S1)}$--$\mathbf{(S3)}$ be satisfied.

(1) For any function $R\in\widetilde{\mathcal{U}}$, the \textbf{control-to-state operator} given by
\begin{align}
{\cal S}:\, \widetilde{\mathcal{U}}\to \mathcal{X},\quad    R \mapsto {\cal S}(R)=\varphi,
\label{ctos}
\end{align}
is well defined, where $\varphi$ is the unique global strong solution to the state system \eqref{bdini}--\eqref{CHHS} on $[0,T]$.

(2) There exist constants $K_1, K_2>0$ and $\widetilde{\delta}_1\in (0,1)$ that depend on $\Omega$, $T$, $\nu_*$, $\nu^*$, $\|\vp_0\|_{H^2}$, $\|\widetilde{\mu}_0\|_{H^1}$,
$|\overline{\vp_0}|$, $\delta_0$, $\|S\|_{L^2(0,T:H^1(\Omega))\cap H^1(0,T;L^2(\Omega))}$, $r_1$ and on the other parameters of the system, but not on the choice of
$R\in \widetilde{\mathcal{U}}$, such that
\begin{align}
& \|\vp\|_{C([0,T]; H^3(\Omega))\cap L^2(0,T;H^5(\Omega))\cap H^1(0,T; H^1(\Omega))}  +\|\u\|_{L^\infty(0,T;H^1(\Omega))} + \|P\|_{L^\infty(0,T;H^2(\Omega))}\notag\\
&\quad + \|\mu\|_{C([0,T];H^1(\Omega))\cap L^2(0,T; H^3(\Omega))\cap H^1(0,T; (H^1(\Omega))')}\leq K_1,
\label{K1}
\end{align}
\begin{align}
\|\vp\|_{C(\overline{Q})}\leq 1-\widetilde{\delta}_1,
\label{delta1}
\end{align}
and
\begin{align}
\max_{0\le j\leq 5} \big\|\Psi^{(j)}(\vp)\big\|_{C(\overline{Q})}\leq K_2.\label{K2}
\end{align}

(3) The control-to-state operator ${\cal S}$ is locally Lipschitz continuous as a mapping from $\widetilde{\mathcal{U}}$ into $\mathcal{X}$ equipped with weaker topologies
induced by $L^2(Q)$ and $\mathcal{Y}$, respectively.
\end{proposition}

\section{Existence of a globally optimal control}
\label{exOP}
\setcounter{equation}{0}
Assume in addition,
\begin{itemize}
\item[$(\mathbf{S4})$] the target functions satisfy $\vp_\Omega\in H^1(\Omega)$, $\vp_Q\in L^2(Q)$;
    \item[$(\mathbf{S5})$] the coefficients $\alpha_1$, $\alpha_2$ and $\beta $ are nonnegative constants (not all equal to zero).
\end{itemize}
Let us consider the following \textbf{optimal control problem}:
\begin{align}
 \textit{Minimize}\quad \mathcal{J}(\vp,R)
:=
\frac{\alpha_1}{2}\|\vp(T)-\vp_\Omega\|_{L^2(\Omega)}^2
+\frac{\alpha_2}{2}\|\vp-\vp_Q\|_{L^2(Q)}^2
+\frac{\beta}{2}\|R\|^2_{L^2(Q)}
\label{costJ}
\end{align}
\textit{subject to the following conditions:}
\begin{itemize}
\item $R$ \textit{is an admissible control, that is,} $R\in \mathcal{U}_{\mathrm{ad}}$\textit{;}
\item $\varphi$ \textit{is the unique strong solution to problem} \eqref{bdini}--\eqref{CHHS} \textit{corresponding to} $R$.
\end{itemize}

Then we introduce the following

\begin{definition}\label{gloloc}
Let $R^*\in \mathcal{U}_{\mathrm{ad}}$.

(1) $R^*$ is called a \textbf{globally optimal control} for problem \eqref{costJ}, if $\mathcal{J}(\mathcal{S}(R^*),R^*)\leq \mathcal{J}(\mathcal{S}(R),R)$
for all $R\in \mathcal{U}_{\mathrm{ad}}$.

(2)  $R^*$ is called a \textbf{locally optimal control} for problem \eqref{costJ} in the sense of $\mathcal{U}$, if there exists some $\lambda>0$ such that
$\mathcal{J}(\mathcal{S}(R^*),R^*)\leq \mathcal{J}(\mathcal{S}(R),R)$ for all $R\in \mathcal{U}_{\mathrm{ad}}$ with $\|R-R^*\|_{\mathcal{U}}\leq\lambda$.

Here, $\vps=\mathcal{S}(R^*)$ is called the associated globally (or locally) optimal state for $R^*$.
\end{definition}

\begin{remark}
Proposition \ref{StoCo} states that the control-to-state operator ${\cal S}$ is well defined for every $R\in \widetilde{\mathcal{U}}$.
Hence, we can reformulate the optimal control problem \eqref{costJ} and minimize the \textbf{reduced cost functional}
\begin{align}
\widehat{\mathcal{J}}(R):=\mathcal{J}(\mathcal{S}(R),R)\quad \text{over }\ \ \mathcal{U}_{\mathrm{ad}},\label{ReJ}
\end{align}
and subject to the state system \eqref{bdini}--\eqref{CHHS}. For a given optimal control $R^*$, we also call the corresponding unique strong solution
$(\uu^*,P^*,\vp^*,\mu^*)$ to problem \eqref{bdini}--\eqref{CHHS} the associated optimal state.
\end{remark}

We begin the study of the optimal control problem \eqref{costJ}  by showing the existence of a globally optimal control.

\begin{theorem}\label{weakcomp}
Assume that $\Omega\subset \mathbb{R}^2$ is a bounded domain with smooth boundary $\partial\Omega$ and $T>0$. Let the assumptions $\mathbf{(S1)}$--$\mathbf{(S5)}$ be satisfied.
The optimal control problem \eqref{costJ} admits at least one globally optimal control $R^*$.
\end{theorem}
\begin{proof}
The conclusion can be proved by the direct method of calculus of variations. We work with the reduced cost functional $\widehat{\mathcal{J}}(R)$.
From the fact that $\widehat{\mathcal{J}}(R)\geq 0$ on $\mathcal{U}_{\mathrm{ad}}$, we infer that its infimum $J_*:=\inf_{R\in \mathcal{U}_{\mathrm{ad}}}\widehat{\mathcal{J}}(R)$
exists. Besides, there exists a minimizing sequence $\{R_n\}_{n\in \mathbb{N}}$ such that
$$ \lim_{n\to +\infty}\widehat{\mathcal{J}}(R_n)=J_*.$$
Since the set $\mathcal{U}_{\mathrm{ad}}$ is weakly (star) compact, we can find some $R^*\in \mathcal{U}_{\mathrm{ad}}$ and a convergent subsequence $\{R_n\}_{n\in \mathbb{N}}$
(not relabelled for simplicity) such that $R_n \to R^*$ weakly star in $\mathcal{U}$. Denote the strong solutions of problem \eqref{bdini}--\eqref{CHHS}
with the mass source term $R_n$ by $(\u_n,P_n,\varphi_n,\mu_n)$. Using the uniform bounds in Proposition \ref{StoCo} and standard compactness arguments, we can conclude that
there exists $(\u^*,P^*,\varphi^*,\mu^*)$ satisfying
\begin{align*}
& \u^*\in L^\infty(0,T;\mathbf{H}^1(\Omega)),\quad P^*\in L^\infty(0,T;H^2(\Omega)\cap L^2_0(\Omega)),\\
& \varphi^* \in C([0,T]; H^3(\Omega))\cap L^2(0,T;H^5(\Omega))\cap H^1(0,T; H^1(\Omega)),\\
& \mu^*\in C([0,T];H^1(\Omega))\cap L^2(0,T; H^3(\Omega))\cap H^1(0,T; (H^1(\Omega))'),
\end{align*}
such that up to a subsequence (again not relabelled for simplicity),
\begin{align*}
& \u_n\to \u^* \quad \text{weakly star in}\quad L^\infty(0,T;\mathbf{H}^1(\Omega)),\\
&P_n \to P^* \quad \text{weakly star in}\quad L^\infty(0,T;H^2(\Omega)\cap L^2_0(\Omega)),\\
& \varphi_n\to \varphi^*\quad \text{weakly in}\quad   L^2(0,T;H^5(\Omega))\cap H^1(0,T; H^1(\Omega)),\\
& \mu_n\to \mu^* \quad \text{weakly in}\quad  L^2(0,T; H^3(\Omega))\cap H^1(0,T; (H^1(\Omega))'),\\
& \varphi_n\to \varphi^*\quad \text{strongly in}\quad C([0,T];H^{3-\epsilon}(\Omega))\cap L^2(0,T;H^{5-\epsilon}(\Omega)),\\
& \mu_n\to \mu^* \quad \text{strongly in}\quad  C([0,T]; H^{1-\epsilon}(\Omega))\cap L^2(0,T; H^{3-\epsilon}(\Omega)),
\end{align*}
for any $\epsilon\in (0,1/2)$. The above (sequential) convergence results enable us to conclude that the limit $(\u^*,P^*,\varphi^*,\mu^*)$ is
indeed a strong solution to problem \eqref{bdini}--\eqref{CHHS} with the mass source term $R^*$ (see \cite[Theorem 4.1]{SW21}).
As a result, $\varphi^*=\mathcal{S}(R^*)$ and $(R^*,\mathcal{S}(R^*))\in \mathcal{U}_{\mathrm{ad}}\times \mathcal{X} $ is an admissible control-state pair.
Thanks to the lower semi-continuity of the reduced cost function $\widehat{\mathcal{J}}$, it holds
 $$
  J_*\leq \widehat{\mathcal{J}}(R^*)\leq \liminf_{n\to +\infty}\widehat{\mathcal{J}}(R_n)=\lim_{n\to +\infty}\widehat{\mathcal{J}}(R_n)=J_*.
 $$
This implies $\widehat{\mathcal{J}}(R^*)=J_*$ so that $R^*$ is a globally optimal control.
\end{proof}

\section{First-order necessary optimality conditions} \setcounter{equation}{0}
\label{FNOC}
Although the cost functional $\mathcal{J}(\vp, R)$ is convex with respect to its components $\vp$ and $R$, our optimal control problem \eqref{costJ} is not a convex one.
This is due to the nonlinear feature of the state system \eqref{bdini}--\eqref{CHHS} (or the nonlinear control-to-state operator $\mathcal{S}$).

The optimal control problem \eqref{costJ} may admit several optimal controls (global or local). In this section, we aim to derive first-order necessary optimality conditions
for a locally optimal control.

\subsection{Differentiability of the control-to-state operator}
\label{1-Diff}

Our first task is to prove the Fr\'{e}chet differentiability of the control-to-state operator $\mathcal{S}$ between suitable Banach spaces.

Let $R^*\in \widetilde{\mathcal{U}}$ be a given control. We denote by $(\u^*,P^*,\varphi^*,\mu^*)$ the unique strong solution corresponding to the control
$R^*$ (recall Proposition \ref{StoCo}) such that $\varphi^*=\mathcal{S}(R^*)$. Consider the following linear initial boundary value problem:
\begin{align}
\label{lsphi}
&\nu(\vp^*)\vv= -\nabla q + \eta \nabla\vp^* + \mu^*\nabla\xi - \nu'(\vp^*)\xi\u^* +\bm{f}_1,  &\text{in}\ Q,\\[1mm]
\label{lsdivu}
&\mathrm{div}\, \vv  = 0,&\text{in}\ Q,\\[1mm]
\label{lsbc}
&\partial_t\xi + \mathrm{div}\,(\varphi^* \vv) + \mathrm{div}\,(\xi\,\u^*) =
\Delta\eta + f_2, &\text{in}\ Q,\\[1mm]
\label{lsmu}
&\eta = -\Delta\xi + \Psi''(\vp^*)\xi+f_3,
&\text{in}\ Q,\\[1mm]
\label{lsu}
&\partial_{\mathbf{n}}\xi=\partial_{\mathbf{n}}\eta= \vv\cdot\mathbf{n}=0, &\text{on}\ \Sigma,\\[1mm]
\label{lsini}
&\xi|_{t=0} =  0, &\text{in}\ \Omega,
\end{align}
where $\bm{f}_1$, $f_2$ and $f_3$ are some vector or scalar functions with $\bm{f}_1\cdot \mathbf{n}=0$ on $\Sigma$.
We note that the pressure variable $q=q(t)$ (formally) solves for almost all $t\in (0,T)$ the elliptic boundary value problem:
\begin{align}
& -\Delta q = -\mathrm{div}\,(\eta\nabla\vp^*) - \mathrm{div}\,(\mu^*\nabla\xi)+ \mathrm{div}\,(\nu(\vp^*) \vv) + \mathrm{div}\,(\nu'(\vp^*)\xi\u^*)-\mathrm{div}\,\bm{f}_1,
&\mbox{in }\,\Omega,\label{ellqa}\\
&\partial_\mathbf{n} q=0,&\mbox{on }\,\Gamma.\label{ellq}
\end{align}
Due to the homogeneous Neumann boundary condition, $q$ is uniquely determined up to a constant. Thus, we simply require that $\overline{q(t)}=0$ for almost all $t\in (0,T)$.

\begin{remark}
Setting $\bm{f}_1=\bm{0}$, $f_2=h$ and $f_3=0$ in  \eqref{lsphi}--\eqref{lsmu}, we arrive at the linearized system of the state system \eqref{bdini}--\eqref{CHHS} at
$(\u^*,P^*,\varphi^*,\mu^*)$. For the convenience of later analysis, here we treat the linearized system \eqref{lsphi}--\eqref{lsmu} that has a slightly more general form.
\end{remark}


\begin{lemma}\label{line}
Assume that $\Omega\subset \mathbb{R}^2$ is a bounded domain with smooth boundary $\partial\Omega$, $T>0$ and the assumptions $\mathbf{(S1)}$--$\mathbf{(S3)}$ are satisfied.
Let $R^* \in \widetilde{\mathcal{U}}$ be given with its associated state denoted by $(\u^*,P^*,\varphi^*,\mu^*)$. Then, for every
$$\bm{f}_1\in L^2(0,T;\mathbf{L}^4(\Omega)) \quad \text{satisfying}\quad \mathrm{div}\, \bm{f}_1 \in L^2(0,T;L^{4/3}(\Omega)),\quad  \bm{f}_1\cdot \mathbf{n}=0\ \
\text{a.e. on}\ \Sigma,
$$
$$
 f_2\in L^2(Q),\quad f_3\in L^2(0,T;H_N^2(\Omega)),
 $$ the linear problem \eqref{lsphi}--\eqref{lsini}
admits a unique strong solution $(\vv,q,\xi,\eta)$ such that
\begin{align*}
&\vv\in L^2(0,T;\mathbf{L}^4(\Omega)\cap \mathbf{H}_\sigma),
\qquad q\in L^2(0,T;W^{2,4/3}(\Omega)),\\[1mm]
&\xi\in C([0,T];H^2(\Omega))\cap L^2(0,T;H^4(\Omega))\cap H^1(0,T;L^2(\Omega)),\\[1mm]
&\eta\in L^2(0,T;H^2(\Omega)),
\end{align*}
with $\overline{q(t)}=0$ for almost all $t\in (0,T)$.
The equations \eqref{lsphi}--\eqref{lsmu} are satisfied almost everywhere in $Q$.
Besides, the boundary conditions
$\partial_{\mathbf{n}}\xi=\partial_{\mathbf{n}}\eta= \vv\cdot\mathbf{n}=\partial_\mathbf{n} q=0$ hold
almost everywhere on $\Sigma$ and the initial condition  $\xi|_{t=0} =  0$ is satisfied almost everywhere in $\Omega$. Moreover, the following estimate holds
\begin{align}
&\|\xi\|^2_{C([0,T];H^2(\Omega))\cap L^2(0,T;H^4(\Omega))\cap H^1(0,T;L^2(\Omega))} +\|\eta\|_{L^2(0,T;H^2(\Omega))}^2 +\|\vv\|_{L^2(0,T;L^4(\Omega))}^2\non \\
&\qquad + \|q\|_{ L^2(0,T;W^{2,4/3}(\Omega))}^2\non\\
&\quad \leq C \|\bm{f}_1\|^2_{L^2(0,T;\mathbf{L}^4(\Omega))} + C\|\mathrm{div}\,\bm{f}_1\|^2_{L^2(0,T;L^{4/3}(\Omega))} + C\|f_2\|^2_{L^2(Q)} + C\|f_3\|^2_{L^2(0,T;H^2(\Omega))}.
\label{estiLine}
\end{align}
In addition, if $f_3\in L^\infty(0,T; L^2(\Omega))$, then it holds $\eta\in L^\infty(0,T;L^2(\Omega))$.
\end{lemma}
\begin{proof}
The proof of existence follows from a standard Faedo-Galerkin scheme, see e.g., \cite[Lemma 3.2]{SW21}. Different from the problem studied therein, extra efforts have to be made
to handle the nonconstant viscosity (see \eqref{lsphi} and \eqref{ellqa}). Below we only perform \textit{a priori} estimates for the solutions, which can be justified rigorously
within the Galerkin approximation.

Let us keep in mind that $(\u^*,P^*,\varphi^*,\mu^*)$ satisfy the estimates \eqref{K1}--\eqref{K2}.  First, from \eqref{lsmu}--\eqref{lsu} and the Poincar\'{e}-Wirtinger inequality,
we obtain
\begin{align}
\int_0^t \|\eta(s)\|_{H^1(\Omega)}^2\mathrm{d}s
&\leq C\int_0^t \|\nabla \eta(s)\|_{L^2(\Omega)}^2\mathrm{d}s+ C\int_0^t |\overline{\eta(s)}|^2\mathrm{d}s\non \\
&\leq C\int_0^t \big(\|\nabla \eta(s)\|_{L^2(\Omega)}^2 + \|\xi(s)\|_{L^2(\Omega)}^2+\|f_3(s)\|_{L^2(\Omega)}^2\big)\mathrm{d}s,\quad \forall\, t\in (0,T].\label{eseta1}
\end{align}

\textbf{Lower order estimate for $\xi$}. Multiplying \eqref{lsbc} by $\xi$, integrating over $Q_t$, arguing as to get \cite[(3.30)]{SW21}, we find
\begin{align}
&\|\xi(t)\|_{L^2(\Omega)}^2
+ \int_0^t \|\Delta\xi(s)\|_{L^2(\Omega)}^2 \mathrm{d}s\non\\
&\quad \le \int_0^t\|h_1(s)\|_{L^2(\Omega)}^2\mathrm{d}s +  \frac{\nu_*}{6K_2} \int_0^t\|\vv(s)\|_{L^2(\Omega)}^2\mathrm{d}s  + C\int_0^t
\|\xi(s)\|_{H^1(\Omega)}^2\mathrm{d}s\non\\
&\qquad + C\int_0^t\|f_2(s)\|_{L^2(\Omega)}^2\,\mathrm{d}s,
\label{fe5}
\end{align}
where $C>0$ depends on $K_1$, $K_2$ and coefficients of the system.
Next, multiplying \eqref{lsphi} by $\vv$ and \eqref{lsbc} by $\eta$, respectively, integrating over $Q_t$ and after an integration by parts, we get
\begin{align}
\label{fe6}
&\frac 12 \|\nabla\xi(t)\|_{L^2(\Omega)}^2  +\int_0^t\|\nabla\eta(s)\|_{L^2(\Omega)}^2\mathrm{d}s + \int_0^t \nu(\vp^*)\|\vv(s)\|_{L^2(\Omega)}^2\mathrm{d}s\non\\
&\quad = \iQt \eta \,(f_2  - \nabla\xi \cdot\u^* - S\xi) \,\mathrm{d}x\mathrm{d}s
+ \iQt \mu^*\,\nabla\xi\cdot \vv\,\mathrm{d}x\mathrm{d}s \non\\
&\qquad + \iQt  (\bm{f}_1-\nu'(\vp^*)\xi\u^*)\cdot \vv\,\mathrm{d}x\mathrm{d}s- \iQt \Psi''(\vp^*)\,\xi\,\partial_t\xi\,\mathrm{d}x \mathrm{d}s
\non\\
&\quad =:\,\sum_{j=1}^4 J_j.
\end{align}
The terms $J_1$ can be estimated as in \cite{SW21} by using \eqref{K1} and \eqref{eseta1}, namely,
\begin{align*}
J_1 &\leq \frac{1}{2}\int_0^t\|\nabla \eta(s)\|_{L^2(\Omega)}^2\,\mathrm{d}s
+C \int_0^t \|\xi(s)\|^2_{H^1(\Omega)}\,\mathrm{d}s
+ C\int_0^t\|f_2(s)\|_{L^2(\Omega)}^2\,\mathrm{d}s.\non
\end{align*}
For $J_2$ and $J_3$, we infer from $(\mathbf{S2})$, \eqref{K1}  and Young's inequality that
\begin{align*}
J_2&\leq C\int_0^t \|\vv(s)\|_{L^2(\Omega)} \|\mu^*(s)\|_{L^4(\Omega)} \|\nabla \xi(s)\|_{L^2(\Omega)}^\frac12 \|\nabla \xi(s)\|_{H^1(\Omega)} ^\frac12 \,\mathrm{d}s \non\\
&\leq \frac{\nu_*}{6}\int_0^t \|\vv(s)\|_{L^2(\Omega)}^2\mathrm{d}s
+
\frac{K_2}{2} \int_0^t  \|\Delta \xi(s)\|_{L^2(\Omega)}^2\,\mathrm{d}s
+ C \int_0^t \|\xi(s)\|_{H^1(\Omega)}^2\,\mathrm{d}s,
\end{align*}
and
\begin{align*}
J_3&\leq \nu^*\int_0^t \|\vv(s)\|_{L^2(\Omega)} \big(\|\bm{f}_1(s)\|_{L^2(\Omega)}+\|\u^*(s)\|_{L^4(\Omega)} \|\xi(s)\|_{L^4(\Omega)}\big) \,\mathrm{d}s \non\\
&\leq \frac{\nu_*}{6}\int_0^t \|\vv(s)\|_{L^2(\Omega)}^2\mathrm{d}s
+ C \int_0^t \|\xi(s)\|_{H^1(\Omega)}^2\,\mathrm{d}s + C\int_0^t\|\bm{f}_1(s)\|_{L^2(\Omega)}^2\,\mathrm{d}s.
\end{align*}
For $J_4$,
using \eqref{lsbc} and integration by parts in time, we have
\begin{align}
\non
J_4\,&=\,-\frac 12\iO\xi^2(t)\Psi''(\vp^*(t))\,\mathrm{d}x +
\frac 12\iQt\xi^2\,\Psi^{(3)}(\vp^*)\,\partial_t\vp^*\,\mathrm{d}x\mathrm{d}s
=: J_5+J_6.
\end{align}
Then it follows from \eqref{K1} and \eqref{K2} that
\begin{align}
\non
J_5\,\le \frac{K_2}{2}\|\xi(t)\|_{L^2(\Omega)}^2,
\end{align}
and
\begin{align}
\non
J_6& \le C\int_0^t\|\xi(s)\|_{L^4(\Omega)}^2\,\|\partial_t\vp^*(s)\|_{L^2(\Omega)}\,\mathrm{d}s
\le C\int_0^t\|\xi(s)\|_{H^1(\Omega)}^2\,\mathrm{d}s.
\end{align}
 Combining the above estimates with \eqref{eseta1}--\eqref{fe5}, we can deduce that,
for all $t\in (0,T]$, it holds (see \cite[(3.37)]{SW21}):
\begin{align}
&\frac 12  \|\xi(t)\|_{L^2(\Omega)}^2
+\frac{1}{2K_2} \|\nabla\xi (t)\|_{L^2(\Omega)}^2
 + \frac{1}{2K_2} \int_0^t \|\nabla\eta(s)\|_{L^2(\Omega)}^2\,\mathrm{d}s\non\\
&\qquad +\,
\frac12 \int_0^t\|\Delta\xi(s)\|_{L^2(\Omega)}^2\,\mathrm{d}s +  \frac{\nu_*}{2K_2} \int_0^t\|\vv(s)\|_{L^2(\Omega)}^2\,\mathrm{d}s\non\\
&\quad \le C \int_0^t \|\xi(s)\|^2_{H^1(\Omega)}\,\mathrm{d}s
+ C\int_0^t\|\bm{f}_1(s)\|_{L^2(\Omega)}^2\,\mathrm{d}s
+ C\int_0^t\|f_2(s)\|_{L^2(\Omega)}^2\,\mathrm{d}s,\non
\end{align}
where the constant $C>0$ depends on $K_1$, $K_2$, and coefficients of the system. Hence, it follows from Gronwall's lemma and \eqref{eseta1} that
\begin{align}
\label{esti1}
&\|\xi\|^2_{L^\infty(0,t;H^1(\Omega))\cap L^2(0,t;H^2(\Omega))} + \|\eta\|^2_{L^2(0,t;H^1(\Omega))} + \|\vv\|^2_{L^2(0,t;L^2(\Omega))}\non\\[1mm]
&\quad \le C\int_0^T\|\bm{f}_1(s)\|_{L^2(\Omega)}^2\,\mathrm{d}s
+ C\int_0^T\|f_2(s)\|_{L^2(\Omega)}^2\,\mathrm{d}s,\qquad \forall\, t\in (0,T].
\end{align}
Moreover, from \eqref{K2}, \eqref{lsmu}, \eqref{esti1} and the elliptic estimate, we find
 \begin{align}
\label{esti1a}
\|\xi\|^2_{L^2(0,t;H^3(\Omega))}
& \le C\int_0^T\|\bm{f}_1(s)\|_{L^2(\Omega)}^2\,\mathrm{d}s
+ C\int_0^T\|f_2(s)\|_{L^2(\Omega)}^2\,\mathrm{d}s\non\\
&\quad + C\int_0^T\|f_3(s)\|_{H^1(\Omega)}^2\,\mathrm{d}s,
\,\qquad \forall\, t\in (0,T].
\end{align}

\textbf{Estimates for $\vv$ and $q$}. Using \eqref{esti1}--\eqref{esti1a} and arguing as in \cite{SW21}, we obtain
\begin{align}
\label{fe14a}
&\int_0^t\|\mathrm{div}\,(\eta(s) \nabla\vp^*(s)) + \mathrm{div}\,(\mu^*(s)\nabla\xi(s))\|_{L^{4/3}(\Omega)}^2\,\mathrm{d}s
+ \int_0^t\|\eta (s)\nabla\vp^*(s) + \mu^*(s)\nabla\xi (s)\|_{L^4(\Omega)}^2\,\mathrm{d}s \non\\ &\quad  \le C\int_0^T\|\bm{f}_1(s)\|_{L^2(\Omega)}^2\,\mathrm{d}s
+ C\int_0^T\|f_2(s)\|_{L^2(\Omega)}^2\,\mathrm{d}s
 + C\int_0^T\|f_3(s)\|_{H^1(\Omega)}^2\,\mathrm{d}s,
\end{align}
for all $t\in (0,T]$.
Besides, using Minkowski's inequality, H\"{o}lder's inequality and \eqref{K1}, \eqref{esti1}, we find
\begin{align}
& \int_0^t \|\mathrm{div}\,(\nu(\vp^*(s))\vv(s))+ \mathrm{div}\,(\nu'(\vp^*(s))\xi(s)\u^*(s))\|_{L^{4/3}(\Omega)}^2\,\mathrm{d}s \non\\
&\quad \leq C \int_0^t \|\nabla \vp^*(s)\|_{L^4(\Omega)}^2\|\vv(s)\|_{L^2(\Omega)}^2\,\mathrm{d}s
 + C\int_0^t \|\nabla \vp^*(s)\|_{L^4(\Omega)}^2\|\xi(s)\|_{L^4(\Omega)}^2\|\u^*(s)\|_{L^4(\Omega)}^2\,\mathrm{d}s\non\\
&\qquad + C\int_0^t \|\nabla\xi(s)\|_{L^2(\Omega)}^2\|\u^*(s)\|_{L^4(\Omega)}^2\,\mathrm{d}s
+ C\int_0^t \|\xi(s)\|_{L^4(\Omega)}^2\|\nabla \u^*(s)\|_{L^2(\Omega)}^2\,\mathrm{d}s\non\\
&\quad \leq   C\int_0^T\|\bm{f}_1(s)\|_{L^2(\Omega)}^2\,\mathrm{d}s
+ C\int_0^T\|f_2(s)\|_{L^2(\Omega)}^2\,\mathrm{d}s,\quad \forall\, t\in (0,T].
\label{fe14b}
\end{align}
It follows from \eqref{fe14a}--\eqref{fe14b} and the elliptic estimate for \eqref{ellqa} that
\begin{align}
\|q\|^2_{L^2(0,t;W^{2,4/3}(\Omega))}
& \leq C\int_0^T\|\bm{f}_1(s)\|_{L^2(\Omega)}^2\,\mathrm{d}s
+C \int_0^T\|\mathrm{div}\,\bm{f}_1(s)\|_{L^{4/3}(\Omega)}^2\,\mathrm{d}s
 \non\\
&\quad + C\int_0^T\|f_2(s)\|_{L^2(\Omega)}^2\,\mathrm{d}s + C\int_0^T\|f_3(s)\|_{H^1(\Omega)}^2\,\mathrm{d}s,\quad \forall\, t\in (0,T].
\label{fe14c}
\end{align}
Then we infer from \eqref{lsphi}, \eqref{fe14c} and the continuous embedding $W^{1,\frac43}(\Omega)\hookrightarrow L^4(\Omega)$  that
\begin{align}
\|\vv\|^2_{L^2(0,t;L^4(\Omega))}
&\leq C \int_0^t\|\nabla q(s)\|_{L^4(\Omega)}^2\,\mathrm{d}s + \int_0^t\|\eta (s)\nabla\vp^*(s) + \mu^*(s)\nabla\xi (s)\|_{L^4(\Omega)}^2\,\mathrm{d}s \non\\
&\quad +C \int_0^t \|\xi(s)\|_{L^8(\Omega)}\|\u^*(s)\|_{L^8(\Omega)}^2\,\mathrm{d}s
+C\int_0^T\|\bm{f}_1(s)\|_{L^4(\Omega)}^2\,\mathrm{d}s\non\\
&\leq  C\int_0^T\|\bm{f}_1(s)\|_{L^4(\Omega)}^2\,\mathrm{d}s
+C \int_0^T\|\mathrm{div}\,\bm{f}_1(s)\|_{L^{4/3}(\Omega)}^2\,\mathrm{d}s
 \non\\
&\quad + C\int_0^T\|f_2(s)\|_{L^2(\Omega)}^2\,\mathrm{d}s + C\int_0^T\|f_3(s)\|_{H^1(\Omega)}^2\,\mathrm{d}s,\quad \forall\, t\in (0,T].\label{fe14d}
\end{align}

\textbf{Higher order estimates for $\xi$ and $\eta$}. Similar to \cite[(3.48)]{SW21}, using \eqref{K1}, we get
\begin{align}
& \|\Delta \xi(t)\|_{L^2(\Omega)}^2
+ \int_0^t \|\Delta^2 \xi(s)\|_{L^2(\Omega)}^2\,\mathrm{d}s \non\\
&\leq C\int_0^t \|\Delta \xi(s)\|_{L^2(\Omega)}^2 \,\mathrm{d}s +C\int_0^T\|\bm{f}_1(s)\|_{L^4(\Omega)}^2\,\mathrm{d}s
+C \int_0^T\|\mathrm{div}\,\bm{f}_1(s)\|_{L^{4/3}(\Omega)}^2\,\mathrm{d}s
 \non\\
&\quad + C\int_0^T\|f_2(s)\|_{L^2(\Omega)}^2\,\mathrm{d}s + C\int_0^T\|f_3(s)\|_{H^2(\Omega)}^2\,\mathrm{d}s.\non
\end{align}
By Gronwall's lemma, \eqref{K2}, \eqref{esti1} and the elliptic estimate  for \eqref{lsmu}, we deduce that
\begin{align}
& \|\xi\|^2_{L^\infty(0,t;H^2(\Omega))\cap L^2(0,t;H^4(\Omega))} + \|\eta\|^2_{L^2(0,t;H^2(\Omega))}\non\\
&\quad \leq C\int_0^T\|\bm{f}_1(s)\|_{L^4(\Omega)}^2\,\mathrm{d}s
+C \int_0^T\|\mathrm{div}\,\bm{f}_1(s)\|_{L^{4/3}(\Omega)}^2\,\mathrm{d}s
 \non\\
&\qquad + C\int_0^T\|f_2(s)\|_{L^2(\Omega)}^2\,\mathrm{d}s + C\int_0^T\|f_3(s)\|_{H^2(\Omega)}^2\,\mathrm{d}s,\quad \forall\, t\in (0,T].
\label{esti2}
\end{align}
Hence, by comparison in \eqref{lsbc}, we conclude
\begin{align}
 \|\partial_t \xi\|^2_{L^2(0,t;L^2(\Omega))}
  & \leq  C\int_0^T\|\bm{f}_1(s)\|_{L^4(\Omega)}^2\,\mathrm{d}s
+C \int_0^T\|\mathrm{div}\,\bm{f}_1(s)\|_{L^{4/3}(\Omega)}^2\,\mathrm{d}s
 \non\\
&\quad + C\int_0^T\|f_2(s)\|_{L^2(\Omega)}^2\,\mathrm{d}s + C\int_0^T\|f_3(s)\|_{H^2(\Omega)}^2\,\mathrm{d}s,\quad \forall\, t\in (0,T].
\label{esti4}
\end{align}
Finally, from \eqref{lsmu} and \eqref{esti2}, we have $\eta\in L^\infty(0,T;L^2(\Omega))$,
provided that $f_3\in L^\infty(0,T; L^2(\Omega))$.

 Collecting the estimates \eqref{fe14c}--\eqref{esti4}, we easily arrive at the conclusion \eqref{estiLine}.
 Next, we observe that if $(\vv,q,\xi,\eta)$ solves problem \eqref{lsphi}--\eqref{lsini} with $\bm{f}_1=\bm{0}$, $f_2=f_3=0$, then \eqref{fe14c}, \eqref{fe14d} and \eqref{esti2}
 imply  $(\vv,q,\xi,\eta)=(\bm{0},0,0,0)$. This yields the uniqueness of solution due to the linearity of the system.
\end{proof}

\begin{remark}
Concerning the special case $\bm{f}_1=\bm{0}$, $f_2=h$ and $f_3=0$,
by virtue of the weak sequential lower semicontinuity of norms, we can conclude from the estimates \eqref{esti2} and \eqref{esti4} that the linear mapping $h\mapsto\xi$ is
continuous from $L^2(Q)$ into $\mathcal{Y}$.
\end{remark}

With the aid of Lemma \ref{line}, we can show the Fr\'{e}chet differentiability of the control-to-state operator $\mathcal{S}$.

\begin{proposition}\label{1stFD}
Assume that $\Omega\subset \mathbb{R}^2$ is a bounded domain with smooth boundary $\partial\Omega$, $T>0$ and the assumptions $\mathbf{(S1)}$--$\mathbf{(S3)}$ are satisfied.
For any given $R^*\in \widetilde{\mathcal{U}}$, let $(\u^*,P^*,\varphi^*,\mu^*)$ be the unique strong solution to problem \eqref{bdini}--\eqref{CHHS} corresponding to $R^*$.

(1) The control-to-state operator $\mathcal{S}: \mathcal{U}\to \mathcal{X}\subset \mathcal{Y}$ defined in \eqref{ctos} is Fr\'{e}chet differentiable
at $R^*$ as a mapping from $\mathcal{U}$ into $\mathcal{Y}$. The Fr\'{e}chet derivative $D\mathcal{S}(R^*)\in \mathcal{L}(\mathcal{U},\mathcal{Y})$ can be determined as follows.
For any $h\in \mathcal{U}\subset L^2(Q)$, it holds $$D\mathcal{S}(R^*)h=\xi^h,$$
where $(\vv^h,q^h,\xi^h,\eta^h)$ is the unique solution to the linear problem \eqref{lsphi}--\eqref{lsini} at $(\u^*,P^*,\varphi^*,\mu^*)$
with $\bm{f}_1=\bm{0}$, $f_2=h$, $f_3=0$, subject to the constraint $\overline{q(t)} = 0$ for almost all $t \in (0, T )$.

(2) The Fr\'{e}chet derivative of the control-to-state operator $\mathcal{S}$ is Lipschitz continuous in $\widetilde{\mathcal{U}}$, that is, for any
$R^*, R^\sharp \in \widetilde{\mathcal{U}}$, it holds
\begin{align}
\|D\mathcal{S}(R^*)-D\mathcal{S}(R^\sharp)\|_{\mathcal{L}(\mathcal{U},\mathcal{Y})}\leq C\|R^*-R^\sharp\|_{L^2(Q)},
\label{DeLip}
\end{align}
where $C>0$ may depend on $K_1$, $K_2$, $r_1$ and on the parameters of the system.
\end{proposition}

\begin{proof}
Since $R^*\in \widetilde{\mathcal{U}}$, there is some $\lambda>0$ sufficiently small such that $R^*+h\in \widetilde{\mathcal{U}}$ whenever $h\in \mathcal{U}$ and
$\|h\|_{\mathcal{U}}< \lambda$. In the following, we shall only consider such small perturbations $h$.
We denote
by  $(\uh,\ph,\vph,\muh)$ the associated state corresponding to $R^*+h$, that is the unique strong solution to problem \eqref{bdini}--\eqref{CHHS}. Next, we define the differences
\begin{align*}
&\yh=\vph-\vps-\xi^h, \quad \zh=\muh-\mus-\eta^h,\quad \bm{\theta}^h=\uh-\bus-\bv^h,
\quad \rh=\ph-\ps-q^h.
\end{align*}
From Proposition \ref{StoCo} and Lemma \ref{line}, we have,
for all admissible perturbations $h$,
the following regularity properties for $(\yh, \zh, \bm{\theta}^h, \rh)$:
\begin{equation*}
 \yh \in \mathcal{Y},\quad \zh\in L^2(0,T;H^2(\Omega)), \quad
\bm{\theta}^h\in L^2(0,T;L^4(\Omega)),\quad \rh\in L^2(0,T;W^{2,4/3}(\Omega)).
\end{equation*}
Moreover, estimates \eqref{K1}--\eqref{K2} hold
for both $(\bus,P^*,\vps,\mus)$ and $(\uh,\ph,\vph,\muh)$. Observe also that estimate \eqref{stabu1} yields
\begin{align}\label{stabu3}
&\|\vph-\vps\|_{H^1(0,T;L^2(\Omega))\cap C([0,T];H^2(\Omega))\cap L^2(0,T;H^4(\Omega))} + \|\mu^h-\mus\|_{L^2(0,T;H^2(\Omega))}\non\\
&\qquad +\|\uh-\bus\|
_{L^2(0,T;L^4(\Omega))} + \|\ph-P^*\|_{L^2(0,T;W^{2,4/3}(\Omega))}\non\\
&\quad \le C\|h\|_{L^2(Q)},
\end{align}
for some $C>0$ independent of $h$.

Using Taylor's formula,
we have
\begin{align*}
\Psi'(\vph)&=\Psi'(\vps)+\Psi''(\vps)(\vph-\vps)+\frac12\int_0^1\int_0^1 \Psi^{(3)}(sz\vph+(1-sz)\vps)(\vph-\vps)^2\mathrm{d}s\mathrm{d}z,\\
\nu(\vph)&=\nu(\vps)+\nu'(\vps)(\vph-\vps)+\frac12\int_0^1\int_0^1 \nu''(sz\vph+(1-sz)\vps)(\vph-\vps)^2\mathrm{d}s\mathrm{d}z.
\end{align*}
Then, by definition,  $(\yh,\zh,\bm{\theta}^h,\rh)$ is a strong solution to the following problem:
\begin{align}
\non
&\nu(\vp^*)\bm{\theta}^h\,=\,-\nabla\rh + \zh\nabla\vps+\mus\nabla\yh -\nu'(\vps)\yh\bus+\bm{f}_1,
&\text{in}\ Q,\\[1mm]
\non
&\mathrm{div}\,\bm{\theta}^h =0,&\text{in}\ Q,\\[1mm]
\non
&\dt\yh +\mathrm{div}\,(\varphi^*\bm{\theta}^h)+\mathrm{div}\,(\yh\bus) = \Delta \zh+f_2, &\text{in}\ Q,\\[1mm]
\non
&\zh\,=\,-\Delta \yh + \Psi''(\vps)\yh+ f_3, &\text{in}\ Q,\\[1mm]
\non
&\bm{\theta}^h\cdot{\bf n}=\dn \rh=\dn\yh=\dn\zh=0, &\text{on}\ \Sigma,\\[1mm]
\non
&\yh|_{t=0}=0, &\text{in}\ \Omega,
\end{align}
with
\begin{align*}
\bm{f}_1&= (\muh-\mus)\nabla(\vph-\vps)-(\uh-\bus)\int_0^1\nu'(s\vph+(1-s)\vps)(\vph-\vps)\,\mathrm{d}s  \non\\
&\quad
-\frac12\bus\int_0^1\int_0^1 \nu''(sz\vph+(1-sz)\vps)(\vph-\vps)^2\,\mathrm{d}s\mathrm{d}z,\\
f_2&= -(\uh-\bus) \cdot\nabla(\vph-\vps),\\
f_3&= \frac12\int_0^1\int_0^1 \Psi^{(3)}(sz\vph+(1-sz)\vps)(\vph-\vps)^2\,\mathrm{d}s\mathrm{d}z.
\end{align*}
Applying Lemma \ref{line}, from \eqref{estiLine} we get
\begin{align}
&\|y^h\|^2_{C([0,T];H^2(\Omega))\cap L^2(0,T;H^4(\Omega))\cap H^1(0,T;L^2(\Omega))} +\|\zh\|_{L^2(0,T;H^2(\Omega))}^2 +\|\bm{\theta}^h\|_{L^2(0,T;L^4(\Omega))}^2\non \\
&\qquad + \|r^h\|_{ L^2(0,T;W^{2,4/3}(\Omega))}^2\non\\
&\quad \leq C \Big(\|\bm{f}_1\|^2_{L^2(0,T;L^4(\Omega))} +  \|\mathrm{div}\, \bm{f}_1\|^2_{L^2(0,T;L^{4/3}(\Omega))} + \|f_2\|^2_{L^2(Q)}   + \|f_3\|^2_{L^2(0,T;H^2(\Omega))}\Big).
\label{estiyh}
\end{align}
Keeping in mind the bounds \eqref{K1}--\eqref{K2} for $(\bus,P^*,\vps,\mus)$, $(\uh,\ph,\vph,\muh)$ and the estimate \eqref{stabu3}, using the Sobolev embedding theorem in
two dimensions, we can deduce the following estimates:
\begin{align*}
\|\bm{f}_1\|^2_{L^2(0,T;L^4(\Omega))}
& \leq C \|\nabla(\vph-\vps)\|_{C([0,T];L^8(\Omega))}^2\int_0^T \|\muh-\mus\|_{L^8(\Omega)}^2\,\mathrm{d}t\non\\
&\quad
+ C\|\vph - \vps\|_{C([0,T];L^\infty(\Omega))}^2\int_0^T  \|\uh-\bus\|_{L^4(\Omega)}^2\,\mathrm{d}t\non\\
&\quad
+  C\|\vph-\vps\|_{C([0,T];L^{\infty}(\Omega))}^4\int_0^T \|\bus\|_{L^{4}(\Omega)}^2\,\mathrm{d}t\non\\
&\leq C\|h\|_{L^2(Q)}^4,
\end{align*}
\begin{align*}
\|\mathrm{div}\, \bm{f}_1\|^2_{L^2(0,T;L^{4/3}(\Omega))}
&\leq C\|\nabla(\vph-\vps)\|_{C([0,T];L^4(\Omega))}^2 \int_0^T \|\nabla (\muh-\mus)\|_{L^2(\Omega)}^2\,\mathrm{d}t\non\\
&\quad + C\|\Delta(\vph-\vps)\|_{C([0,T];L^2(\Omega))}^2 \int_0^T  \|\muh-\mus \|_{L^4(\Omega)}^2\,\mathrm{d}t\non\\
&\quad +C\|\vph-\vps\|_{C([0,T];H^1(\Omega))}^2 \int_0^T \|\uh-\bus\|_{L^4(\Omega)}^2\,\mathrm{d}t\non\\
&\quad +C\|\vph-\vps\|_{C([0,T];L^{4}(\Omega))}^4 \int_0^T \big(\|S\|_{L^4(\Omega)}^2+ \|\bus\|_{L^4(\Omega)}^2\big)\,\mathrm{d}t\non\\
&\quad + C\|\vph-\vps\|_{C([0,T];W^{1,4}(\Omega))}^2\int_0^T \|\bus\|_{L^4(\Omega)}^2\,\mathrm{d}t\non\\
&\leq C\|h\|_{L^2(Q)}^4,
\end{align*}
\begin{align*}
\|f_2\|^2_{L^2(Q)} \leq C \|\nabla(\vph-\vps)\|_{C([0,T];L^4(\Omega))}^2\int_0^T \|\uh-\bus\|_{L^4(\Omega)}^2 \,\mathrm{d}t \leq  C\|h\|_{L^2(Q)}^4,
\end{align*}
and
\begin{align*}
&\|f_3\|^2_{L^2(0,T;H^2(\Omega))} \non\\
&\quad \leq C\int_0^T \left(\int_0^1\int_0^1 \|\Psi^{(3)}(sz\vph+(1-sz)\vps)\|_{H^2(\Omega)}\,\mathrm{d}s\mathrm{d}z \|\vph-\vps\|_{H^2(\Omega)}^2\right)^2\mathrm{d}t\non\\
&\quad \leq  C\|h\|_{L^2(Q)}^4.
\end{align*}
The above estimates together with \eqref{estiyh} yield
\begin{align}
&\|y^h\|_{C([0,T];H^2(\Omega))\cap L^2(0,T;H^4(\Omega))\cap H^1(0,T;L^2(\Omega))}^2 +\|\zh\|_{L^2(0,T;H^2(\Omega))}^2 +\|\bm{\theta}^h\|_{L^2(0,T;L^4(\Omega))}^2\non \\
&\quad + \|r^h\|_{ L^2(0,T;W^{2,4/3}(\Omega))}^2
\leq C\|h\|_{L^2(Q)}^4,
\label{DDD1}
\end{align}
for any $h\in \mathcal{U}\subset L^2(Q)$ with $\|h\|_{\mathcal{U}}< \lambda$.
As a consequence, we can verify that
$$
\frac{\|\mathcal{S}(R^*+h)-\mathcal{S}(R^*)-\xi^h\|_{\mathcal{Y}}}{\|h\|_{\mathcal{U}}}=\frac{\|y^h\|_{\mathcal{Y}}}{\|h\|_{\mathcal{U}}}
\leq C\|h\|_{\mathcal{U}}\to 0\quad \text{as}\ \ \ \|h\|_{\mathcal{U}}\to 0.
$$
This completes the proof of the assertion (1).

We are left to prove the Lipschitz continuity of the Fr\'{e}chet derivative of $\mathcal{S}$ (i.e., assertion (2)).

For two given controls $R^*, R^\sharp\in \widetilde{\mathcal{U}}$, we denote  by $(\bus,P^*,\vp^*,\mu^*)$, $(\u^\sharp,P^\sharp,\vp^\sharp,\mu^\sharp)$ their associate states,
and by $(\vv^*,q^*,\xi^*,\eta^*)$, $(\vv^\sharp,q^\sharp,\xi^\sharp,\eta^\sharp)$ the corresponding solutions to the linear system \eqref{lsphi}--\eqref{lsini}
at $(\u^*,P^*,\varphi^*,\mu^*)$ and $(\u^\sharp,P^\sharp,\vp^\sharp,\mu^\sharp)$ with $\bm{f}_1=\bm{0}$, $f_2=h$, $f_3=0$, respectively. Then we have
$$
D\mathcal{S}(R^*)h=\xi^*,\quad D\mathcal{S}(R^\sharp)h=\xi^\sharp.
$$
Setting
$$
\vv=\vv^*-\vv^\sharp,\qquad q=q^*-q^\sharp,\quad \xi=\xi^*-\xi^\sharp,\quad \eta=\eta^*-\eta^\sharp,
$$
it is easy to realize that the differences $(\vv,q,\xi,\eta)$ is a solution to the following problem
\begin{align*}
&\nu(\vp^*)\vv = -\nabla q + \eta \nabla\vp^* + \mu^*\nabla\xi - \nu'(\vp^*)\xi\u^* + \bm{f}_1,  &\text{in}\ Q,\\[1mm]
&\mathrm{div}\, \vv  = 0,&\text{in}\ Q,\\[1mm]
&\partial_t\xi + \mathrm{div}\,(\varphi^* \vv) + \mathrm{div}\,(\xi\,\u^*) =
\Delta\eta + f_2, &\text{in}\ Q,\\[1mm]
&\eta = -\Delta\xi + \Psi''(\vp^*)\xi+f_3,
&\text{in}\ Q,\\[1mm]
& \vv\cdot\mathbf{n}=\partial_{\mathbf{n}}q=\partial_{\mathbf{n}}\xi=\partial_{\mathbf{n}}\eta=0, &\text{on}\ \Sigma,\\[1mm]
&\xi|_{t=0} =  0, &\text{in}\ \Omega,
\end{align*}
with
\begin{align*}
\bm{f}_1
&= -\vv^\sharp\int_0^1\nu'(s\vp^*+(1-s)\vp^\sharp)(\vp-\vp^\sharp)\,\mathrm{d}s\\
&\quad +\eta^\sharp\nabla(\vp^*-\vp^\sharp)
+(\mu^*-\mu^\sharp)\nabla \xi^\sharp
-\nu'(\vp^*)(\u^*-\u^\sharp)\xi^\sharp\\
&\quad -\xi^\sharp\u^\sharp \int_0^1\nu''(\vp^*)(s\vp^*+(1-s)\vp^\sharp)(\vp^*-\vp^\sharp)\,\mathrm{d}s,\\
f_2&= -\mathrm{div}\,((\vp^*-\vp^\sharp)\vv^\sharp)-\mathrm{div}\,(\xi^\sharp(\u^*-\u^\sharp)),\\
f_3&=\xi^\sharp \int_0^1\Psi^{(3)}(s\vp^*+(1-s)\vp^\sharp)(\vp-\vp^\sharp)\,\mathrm{d}s.
\end{align*}
Using estimate \eqref{estiLine} for $(\vv^*,q^*,\xi^*,\eta^*)$, $(\vv^\sharp,q^\sharp,\xi^\sharp,\eta^\sharp)$,
the facts that $\eta^*, \eta^\sharp\in  L^\infty(0,T;L^2(\Omega))$ and the continuous dependence estimate \eqref{stabu1}, we obtain
\begin{align*}
&\|\bm{f}_1\|^2_{L^2(0,T;L^4(\Omega))} \\
&\quad \leq C\|\vp-\vp^\sharp\|_{C([0,T];L^\infty(\Omega))}^2 \int_0^T\|\vv^\sharp\|_{L^4(\Omega)}^2\,\mathrm{d}t
+ C\|\eta^\sharp\|_{L^\infty(0,T;L^2(\Omega))}^2\int_0^T \|\nabla(\vp^*-\vp^\sharp)\|_{L^\infty(\Omega)}^2\,\mathrm{d}t\\
&\qquad + C\|\nabla \xi^\sharp \|_{C([0,T];L^8(\Omega))}^2\int_0^T \|\mu^*-\mu^\sharp\|_{L^8(\Omega)}^2  \,\mathrm{d}t
+ C\|\xi^\sharp\|_{C([0,T];L^\infty(\Omega))}^2\int_0^T \|\u^*-\u^\sharp\|_{L^4(\Omega)}^2\,\mathrm{d}t\\
&\qquad + C \|\xi^\sharp\|_{C([0,T];L^\infty(\Omega))}^2 \|\vp^*-\vp^\sharp\|_{C([0,T];L^\infty(\Omega))}^2 \int_0^T \|\u^\sharp\|_{L^4(\Omega)}^2\,\mathrm{d}t\\
&\quad \leq C\|R^*-R^\sharp\|_{L^2(Q)}^2\|h\|_{L^2(Q)}^2,
\end{align*}
\begin{align*}
& \|\mathrm{div}\, \bm{f}_1\|^2_{L^2(0,T;L^{4/3}(\Omega))} \\
&\quad \leq  C\|\vp^*-\vp^\sharp\|_{C([0,T];H^1(\Omega))}^2 \int_0^T\|\vv^\sharp\|_{L^4(\Omega)}^2\,\mathrm{d}t
 +C \|\nabla (\vp^*-\vp^\sharp)\|_{C([0,T];L^4(\Omega))}^2 \int_0^T \|\nabla \eta^\sharp\|_{L^2(\Omega)}^2\,\mathrm{d}t\\
&\qquad +C \|\vp^*-\vp^\sharp\|_{C([0,T];H^2(\Omega))}^2 \int_0^T \|\eta^\sharp\|_{L^4(\Omega)}^2\,\mathrm{d}t
 +C \|\nabla \xi^\sharp \|_{C([0,T];L^4(\Omega))}^2 \int_0^T \|\nabla (\mu^*-\mu^\sharp)\|_{L^2(\Omega)}^2 \,\mathrm{d}t\\
&\qquad +C \|\xi^\sharp\|_{C([0,T];H^2(\Omega))}^2 \int_0^T \|\mu^*-\mu^\sharp\|_{L^4(\Omega)}^2 \,\mathrm{d}t
+C \|\xi^\sharp\|_{C([0,T];H^1(\Omega))}^2\int_0^T\|\u^*-\u^\sharp \|_{L^4(\Omega)}^2 \,\mathrm{d}t \\
&\qquad +C \|\xi^\sharp\|_{C([0,T];W^{1,4}(\Omega))}^2\|\vp^*-\vp^\sharp\|_{C([0,T];L^\infty(\Omega))}^2 \int_0^T\big(\|\u^\sharp \|_{L^2(\Omega)}^2
+\|S\|_{L^2(\Omega)}^2\big) \,\mathrm{d}t\\
&\qquad +C \|\xi^\sharp\|_{C([0,T];L^\infty(\Omega))}^2\|\nabla(\vp^*-\vp^\sharp)\|_{C([0,T];L^4(\Omega))}^2 \int_0^T \|\u^\sharp \|_{L^2(\Omega)}^2  \,\mathrm{d}t\\
&\quad \leq C\|R^*-R^\sharp\|_{L^2(Q)}^2\|h\|_{L^2(Q)}^2,
\end{align*}
\begin{align*}
\|f_2\|_{L^2(Q)}&\leq C\|\nabla (\vp^*-\vp^\sharp)\|_{C([0,T];L^4(\Omega))}^2\int_0^T \|\vv^\sharp\|_{L^4(\Omega)}^2  \,\mathrm{d}t\\
&\quad +C \|\nabla \xi^\sharp\|_{C([0,T];L^4(\Omega))}^2\int_0^T\|\u^*-\u^\sharp\|_{L^4(\Omega)}^2  \,\mathrm{d}t\\
& \leq C\|R^*-R^\sharp\|_{L^2(Q)}^2\|h\|_{L^2(Q)}^2,
\end{align*}
\begin{align*}
&\|f_3\|_{L^2(0,T;H^2(\Omega))}^2 \\
&\quad \leq C\int_0^T
\left( \|\vp^*-\vp^\sharp\|_{H^2(\Omega)} \|\xi^\sharp\|_{H^2(\Omega)}\int_0^1 \|\Psi^{(3)}(s\vp^*+(1-s)\vp^\sharp)\|_{H^2(\Omega)}\,\mathrm{d}s\right)^2 \mathrm{d}t\non\\
 &\quad \leq  C\|R^*-R^\sharp\|_{L^2(Q)}^2\|h\|_{L^2(Q)}^2.
\end{align*}
 Applying Lemma \ref{line} again, from the above estimates and \eqref{estiLine} , we find that
 \begin{align}
&\|\xi\|^2_{C([0,T];H^2(\Omega))\cap L^2(0,T;H^4(\Omega))\cap H^1(0,T;L^2(\Omega))} +\|\eta\|_{L^2(0,T;H^2(\Omega))}^2 \non \\
&\qquad +\|\vv\|_{L^2(0,T;L^4(\Omega))}^2 + \|q\|_{ L^2(0,T;W^{2,4/3}(\Omega))}\non\\
&\quad \leq C\|R^*-R^\sharp\|_{L^2(Q)}^2\|h\|_{L^2(Q)}^2,\label{DeLipa}
\end{align}
which yields the desired conclusion, i.e., \eqref{DeLip}.
\end{proof}

\subsection{First-order necessary optimality conditions}

Thanks to the Fr\'{e}chet differentiability of $\mathcal{S}$ in $\widetilde{\mathcal{U}}$, the Fr\'{e}chet differentiability of the cost functional $\mathcal{J}$ easily
follows by chain rule.
This enables us to establish first-order necessary optimality conditions for the optimal control problem \eqref{costJ}, namely,

\begin{theorem}\label{NCLO1}
Assume that $\Omega\subset \mathbb{R}^2$ is a bounded domain with smooth boundary $\partial\Omega$, $T>0$ and the assumptions $\mathbf{(S1)}$--$\mathbf{(S5)}$ are satisfied.
Let $R^*\in\mathcal{U}_{\mathrm{ad}}$ be a locally optimal control of
problem \eqref{costJ} with the associated state $\vps={\cal S}(R^*)$. Then the following variational inequality holds
\begin{align}\label{vug1}
& \alpha_1\iO\left(\vps(T)-\vp_\Omega\right)\xi(T)\,\mathrm{d}x + \alpha_2\int_Q\left(\vps-\vp_Q\right)\xi\,\mathrm{d}x\mathrm{d}t
+ \beta \int_Q R^*(R-R^*)\,\mathrm{d}x\mathrm{d}t\,\ge\,0,
\end{align}
for all $R\in\mathcal{U}_{\mathrm{ad}}$, where $\xi$ is the unique strong solution to the linear problem
\eqref{lsphi}--\eqref{lsini} with $\bm{f}_1=\bm{0}$, $f_2=R-R^*$ and $f_3=0$.
\end{theorem}
\begin{proof}

Recalling the definition of the reduced cost functional $\widetilde{\mathcal{J}}$ (see \eqref{ReJ}) and invoking the convexity of $\mathcal{U}_{\mathrm{ad}}$,
we obtain (cf. \cite[Lemma 2.21]{To})
$$\big(\widetilde{\mathcal{J}}'(R^*), R-R^*\big)\geq 0,\qquad \forall\, R\in \mathcal{U}_{\mathrm{ad}}.$$
On the other hand, we infer from the chain rule that
$$
\widetilde{\mathcal{J}}'(R)= \mathcal{J}'_{\mathcal{S}(R)}(\mathcal{S}(R), R)\circ D\mathcal{S}(R)+\mathcal{J}'_{R}(\mathcal{S}(R), R),
$$ where for every fixed $R\in \mathcal{U}$,
$\mathcal{J}'_{\vp}(\vp, R)$ is the Fr\'echet derivative of $\mathcal{J}(\vp,R)$  with respect to $\vp$ at $\vp\in \mathcal{Y}$, and for every fixed $\vp\in \mathcal{Y}$,
$\mathcal{J}'_{R}(\vp, R)$ is the Fr\'echet derivative with respect to $R$ at $R\in \mathcal{U}$. Hence, by a straightforward computation and using the fact
$D\mathcal{S}(R^*)(R-R^*)=\xi$ (see Proposition \ref{1stFD}), we obtain the variational inequality \eqref{vug1}.
\end{proof}

Our next aim is to simplify the variational inequality \eqref{vug1} by introducing an adjoint state. Let $R^*\in \widetilde{\mathcal{U}}$ be a control with its associated
state denoted by  $(\bus,P^*,\vps,\mus)$. We consider the following \textbf{adjoint system}:
\begin{align}
\label{asu}
& \nu(\vps)\bm{w} =\nabla \pi - \rho\nabla\vps,&\mbox{in $Q$,}\\[1mm]
\label{asp}
&\mathrm{div} \,\bm{w} =0,&\mbox{in  $Q$,}\\[1mm]
\label{asphi}
&-\dt \rho-\bus\cdot\nabla \rho+\Delta \zeta -\Psi''(\vps)\,\zeta+\bm{w} \cdot \nabla \mus \non\\
&\qquad \qquad  +\nu'(\vp^*) \bus\cdot \bm{w} = \alpha_2(\vps-\vp_Q),&\mbox{in $Q$,}\\[1mm]
\label{asmu}
& \zeta = \Delta \rho +  \bm{w}\cdot \nabla \vps,&\mbox{in  $Q$},
\end{align}
subject to the boundary and endpoint conditions:
\begin{align}
\label{asbc}
&\dn \rho = \dn \zeta =\bm{w}\cdot{\bf n} = 0,&\mbox{on \,$\Sigma$,}\\[1mm]
\label{asini}
&\rho|_{t=T} = \alpha_1\,(\vps(T)-\vp_\Omega),&\mbox{in $\Omega$.}
\end{align}
Besides, we see from \eqref{asu}--\eqref{asbc} that $\pi$ (formally) satisfies the elliptic problem
\begin{align}
&\Delta \pi =\mathrm{div}\,(\rho\nabla\vps) + \nu'(\vps)\nabla \vps\cdot \bm{w}, &\mbox{in $Q$,}\label{asp4}\\[1mm]
&\dn \pi=0, &\mbox{on $\Sigma$}.\label{asp4bc}
\end{align}
Like before, in order to uniquely determine $\pi$, we require
$\overline{\pi(t)}=0$ for almost all $t\in (0,T)$.

\begin{remark}
The adjoint system \eqref{asu}--\eqref{asp4bc} can be derived by using
the formal Lagrangian method (see \cite{To}). Its solution, if exists, is called the \textbf{adjoint state} or \textbf{costate} associated with $R^*$.
\end{remark}

For the convenience of later analysis, we study a slightly more general linearized system. More precisely, we have
\begin{lemma}
\label{adexe}
Assume that $\Omega\subset \mathbb{R}^2$ is a bounded domain with smooth boundary $\partial\Omega$, $T>0$ and the assumptions $\mathbf{(S1)}$--$\mathbf{(S3)}$ are satisfied.
Let $R^*\in \widetilde{\mathcal{U}}$ be any control with its associated state denoted by  $(\bus,P^*,\vps,\mus)$.
Then, for every
$$\bm{g}_1\in L^2(0,T;\mathbf{H}^1(\Omega)) \quad \text{with} \quad  \bm{g}_1\cdot \mathbf{n}=0\ \ \text{a.e. on}\ \Sigma,
$$
$$
 g_2\in L^2(0,T;L^{4/3}(\Omega)),\quad g_3\in L^2(0,T;H^1(\Omega)),\quad g_4\in H^1(\Omega),
 $$
there exists a unique quadruple $(\bm{w},\pi,\rho,\zeta)$ with  the following regularity
\begin{align}
& \rho \in  C([0,T];H^1(\Omega))\cap L^2(0,T;H^3(\Omega))\cap H^1(0,T;(H^1(\Omega))'),
\label{wreg1}\\[1mm]
& \zeta\in L^2(0,T;H^1(\Omega)), \quad \bm{w}\in L^2(0,T; \mathbf{H}^1(\Omega)\cap \mathbf{H}_\sigma), \quad \pi\in L^2(0,T;H^2(\Omega)\cap L^2_0(\Omega)),
\label{wreg2}
\end{align}
such that it satisfies
\begin{align}
\label{asu1}
& \nu(\vps)\bm{w} =\nabla \pi - \rho\nabla\vps +\bm{g}_1,&\mbox{a.e in $Q$,}\\[1mm]
\label{asp1}
&\mathrm{div} \,\bm{w} =0,&\mbox{a.e in  $Q$,}\\[1mm]
&\Delta \pi =\mathrm{div}\,(\rho\nabla\vps) + \nu'(\vps)\nabla \vps\cdot \bm{w} -\mathrm{div}\,\bm{g}_1, &\mbox{a.e in $Q$,}\label{asp41}\\[1mm]
\label{asmu1}
& \zeta = \Delta \rho + \bm{w} \cdot \nabla \vps + g_3, &\mbox{a.e in  $Q$,}\\[1mm]
\label{asbc1}
&\dn \rho  =\bm{w}\cdot{\bf n}= \dn \pi= 0,&\mbox{a.e on $\Sigma$,}\\[1mm]
\label{asini1}
&\rho|_{t=T} = g_4,&\mbox{a.e in $\Omega$,}
\end{align}
and
\begin{align}
\label{asphi1}
&-\langle\dt \rho, \psi\rangle_{(H^1)',H^1}-\int_\Omega (\bus\cdot\nabla \rho)\psi\,\mathrm{d}x -\int_\Omega \nabla  \zeta \cdot\nabla \psi\,\mathrm{d}x
- \int_\Omega \Psi''(\vps) \zeta\psi\,\mathrm{d}x \non\\
&\qquad \qquad +\int_\Omega  (\bm{w} \cdot \nabla \mus)\psi\,\mathrm{d}x +\int_\Omega (\nu'(\vp^*) \bus\cdot \bm{w})\psi\,\mathrm{d}x = \int_\Omega g_2\psi\,\mathrm{d}x,
\end{align}
for almost all $t\in(0,T)$ and all $\psi\in H^1(\Omega)$.
Moreover, the following estimate holds
\begin{align}
&\|\rho\|^2_{C([0,T];H^1(\Omega))\cap L^2(0,T;H^3(\Omega))\cap H^1(0,T;(H^1(\Omega))')} +\|\zeta\|_{L^2(0,T;H^1(\Omega))}^2 \non \\
&\qquad +\|\bm{w}\|_{L^2(0,T;H^1(\Omega))}^2 + \|\pi\|_{ L^2(0,T;H^2(\Omega))}^2\non\\
&\quad \leq C \Big(\|\bm{g}_1\|^2_{L^2(0,T;H^1(\Omega))}
+ \|g_2\|^2_{L^2(0,T;L^{4/3}(\Omega))} +  \|g_3\|^2_{L^2(0,T;H^1(\Omega))} + \|g_4\|^2_{H^1(\Omega)}\Big).
\label{estiLineb}
\end{align}
In addition, if $\bm{g}_1\in L^\infty(0,T;\mathbf{L}^2(\Omega))$, then
$\bm{w}\in L^\infty(0,T; \mathbf{L}^2(\Omega))$; if $\bm{g}_1\in L^4(0,T;\mathbf{L}^2(\Omega))$ and $\mathrm{curl}\,\bm{g}_1\in L^4(0,T;L^2(\Omega))$,
then $\bm{w}\in L^4(0,T; \mathbf{H}^1(\Omega))$.
\end{lemma}
\begin{proof}
The existence result again follows from a standard Faedo–Galerkin method.  Therefore, we omit the implementation of the approximation scheme and just perform
necessary \emph{a priori} estimates.

\textbf{First estimate}. Multiplying \eqref{asu1} by $\bm{w}$, integrating over $\Omega$, using \eqref{asp1} and H\"{o}lder's inequality, we get
\begin{align*}
\nu_*\|\bm{w}\|_{L^2(\Omega)}^2  & \leq  C \big(  \| \rho\|_{L^2(\Omega)}\|\nabla\vps\|_{L^\infty(\Omega)} + \|\bm{g}_1\|_{L^2(\Omega)}\big)\|\bm{w}\|_{L^2(\Omega)},\non
\end{align*}
which implies
\begin{align}
\|\bm{w}\|_{L^2(\Omega)}&\leq C    \| \rho\|_{L^2(\Omega)} + C\|\bm{g}_1\|_{L^2(\Omega)}.\label{eswL2}
\end{align}
 Testing \eqref{asphi1} with $\rho$,   we obtain
\begin{align}
&-\frac12\frac{\mathrm{d}}{\mathrm{d}t}\|\rho\|_{L^2(\Omega)}^2 + \|\Delta \rho\|_{L^2(\Omega)}^2\non\\
&\quad = - \iO (\bm{w}\cdot \nabla \vps)\Delta \rho\, \mathrm{d}x
+ \iO \big[\Psi''(\vps)\,\Delta \rho+\bus\cdot\nabla \rho\big] \rho\,\mathrm{d}x\non\\[1mm]
&\qquad + \iO \Psi''(\vps)(\bm{w}\cdot \nabla \vps)\rho  \,\mathrm{d}x
- \iO (\bm{w}\cdot \nabla \mus)\rho \, \mathrm{d}x
- \int_\Omega (\nu'(\vp^*) \bus\cdot \bm{w})\rho\,\mathrm{d}x\non\\[1mm]
&\qquad +\iO [\Psi''(\vps) g_3 +g_2] \rho  \,\mathrm{d}x  -\int_\Omega g_3\Delta \rho\,\mathrm{d}x
\non\\
&\quad =:\sum_{j=1}^7I_j.
\label{esline}
\end{align}
Keeping in mind that the estimates \eqref{K1}--\eqref{K2} hold  for $(\bus,P^*,\vps,\mus)$, we handle terms on the right-hand side of \eqref{esline}.
$I_2$ and $I_4$ can be estimated as in \cite[(4.19), (4.21)]{SW21} such that
\begin{align*}
I_2 & \leq \frac{1}{12} \|\Delta \rho\|_{L^2(\Omega)}^2
+ C(1+\|S\|^2_{L^2(\Omega)})\|\rho\|_{L^2(\Omega)}^2,
\end{align*}
\begin{align*}
I_4 &\leq \|\bm{w}\|_{L^2(\Omega)}\|\nabla \mus\|_{L^4(\Omega)}\|\rho\|_{L^4(\Omega)}\non\\
&\leq  C \|\nabla \mus\|_{L^4(\Omega)}\big( \| \rho\|_{L^2(\Omega)} + C\|\bm{g}_1\|_{L^2(\Omega)}\big) \| \rho\|_{H^1(\Omega)} \non\\
&\leq  \frac{1}{12} \|\Delta \rho\|_{L^2(\Omega)}^2
+ C(1+\|\nabla \mus\|_{L^4(\Omega)}^2)\|\rho\|_{L^2(\Omega)}^2
+ C\|\bm{g}_1\|_{L^2(\Omega)}^2.
\end{align*}
For the remaining terms, we have
\begin{align*}
I_1&\leq \|\bm{w}\|_{L^2(\Omega)}\|\nabla \varphi^*\|_{L^\infty(\Omega)}\|\Delta \rho\|_{L^2(\Omega)}\non\\
&\leq \frac{1}{12}\|\Delta \rho\|_{L^2(\Omega)}^2+ C    \| \rho\|_{L^2(\Omega)}^2 + C\|\bm{g}_1\|_{L^2(\Omega)}^2,
\end{align*}
\begin{align*}
I_3&\leq \|\Psi''(\vps)\|_{L^\infty(\Omega)}\|\bm{w}\|_{L^2(\Omega)}\|\nabla \vps\|_{L^\infty(\Omega)}\|\rho\|_{L^2(\Omega)}\non\\
&\leq C    \| \rho\|_{L^2(\Omega)}^2 + C\|\bm{g}_1\|_{L^2(\Omega)}^2,
\end{align*}
\begin{align*}
I_5&\leq \|\nu'(\vp^*)\|_{L^\infty(\Omega)} \|\bus\|_{L^4(\Omega)} \|\bm{w}\|_{L^2(\Omega)}\|\rho\|_{L^4(\Omega)}\non\\
&\leq  \frac{1}{12} \|\Delta \rho\|_{L^2(\Omega)}^2+ C\|\rho\|_{L^2(\Omega)}^2 + C\|\bm{g}_1\|_{L^2(\Omega)}^2,
\end{align*}
\begin{align*}
I_6  &\leq \|\Psi''(\vps)\|_{L^\infty(\Omega)} \|g_3\|_{L^2(\Omega)}\| \rho  \|_{L^2(\Omega)} + \|g_2\|_{L^{4/3}(\Omega)}\|\rho\|_{L^4(\Omega)}\non\\
&\leq  \frac{1}{12} \|\Delta \rho\|_{L^2(\Omega)}^2 + C\|\rho\|_{L^2(\Omega)}^2 + C\|g_2\|_{L^{4/3}(\Omega)}^2
+  C\|g_3\|_{L^2(\Omega)}^2,
\end{align*}
\begin{align*}
I_7 &\leq \|g_3\|_{L^2(\Omega)}\|\Delta\rho\|_{L^2(\Omega)}
\leq \frac{1}{12} \|\Delta \rho\|_{L^2(\Omega)}^2
+  C\|g_3\|_{L^2(\Omega)}^2.
\end{align*}
As a consequence, we deduce from \eqref{esline} and $(\mathbf{A3})$ that
\begin{align}
&-\frac12\frac{\mathrm{d}}{\mathrm{d}t}\|\rho\|_{L^2(\Omega)}^2 + \frac12\|\Delta \rho\|_{L^2(\Omega)}^2\non\\
&\quad \leq  C\big(1+\|\nabla \mus\|_{L^4(\Omega)}^2\big)\|\rho\|_{L^2(\Omega)}^2 + C\|\bm{g}_1\|_{L^2(\Omega)}^2
+ C\|g_2\|_{L^{4/3}(\Omega)}^2 +  C\|g_3\|_{L^2(\Omega)}^2.
\non
\end{align}
Then the (backward) Gronwall's inequality yields
\begin{align}
&\|\rho(t)\|_{L^2(\Omega)}^2+ \int_t^T\|\Delta \rho(s)\|_{L^2(\Omega)}^2\,\mathrm{d}s\non\\
&\quad  \leq C\Big(\|g_4\|_{L^2(\Omega)}^2 + \|\bm{g}_1\|_{L^2(Q)}^2
+  \|g_2\|_{L^2(0,T;L^{4/3}(\Omega))}^2 +   \|g_3\|_{L^2(Q)}^2\Big),\quad \forall\, t\in [0,T].
\label{eslinea}
\end{align}
Besides, it follows from \eqref{eswL2} that
\begin{align}
& \|\bm{w}\|_{L^2(Q)}^2 \leq C\Big(\|g_4\|_{L^2(\Omega)}^2 + \|\bm{g}_1\|_{L^2(Q)}^2
+ \|g_2\|_{L^2(0,T;L^{4/3}(\Omega))}^2 +  \|g_3\|_{L^2(Q)}^2\Big),\label{eslineb}
\end{align}
and, due to \eqref{asmu1}, it holds
\begin{align}
& \|\zeta\|_{L^2(Q)}^2 \leq C\Big(\|g_4\|_{L^2(\Omega)}^2 + \|\bm{g}_1\|_{L^2(Q)}^2
+ \|g_2\|_{L^2(0,T;L^{4/3}(\Omega))}^2
+ \|g_3\|_{L^2(Q)}^2\Big).\label{eslinec}
\end{align}
From \eqref{eswL2} and \eqref{eslinea},  we have
$\bm{w}\in L^\infty(0,T; \mathbf{L}^2(\Omega))$, provided that   $\bm{g}_1\in L^\infty(0,T;\mathbf{L}^2(\Omega))$.

\textbf{Second estimate}. From \eqref{asp41}, \eqref{eslineb}, the assumption  $\overline{\pi}=0$ and the elliptic estimate, we obtain
\begin{align}
\|\pi\|_{H^2(\Omega)}
& \leq C\|\mathrm{div}\,(\rho\nabla\vps)\|_{L^{2}(\Omega)}
+ C \|\nu'(\vps)\nabla \vps \cdot \bm{w}\|_{L^{2}(\Omega)}
+ C \|\mathrm{div}\,\bm{g}_1\|_{L^{2}(\Omega)}\non\\
& \leq C\|\rho\|_{H^1(\Omega)} + C\|\bm{w}\|_{L^{2}(\Omega)}+ C \|\mathrm{div}\,\bm{g}_1\|_{L^{2}(\Omega)}\non\\
&\leq C \Big(\|\rho\|_{H^1(\Omega)} +   \|\mathrm{div}\,\bm{g}_1\|_{L^{2}(\Omega)} +  \|\bm{g}_1\|_{L^2(\Omega)}\Big). \non
\end{align}
This together with \eqref{eslinea} implies
\begin{align}
&\|\pi\|_{L^2(0,T;H^2(\Omega))}^2\non\\
&\quad \leq C\Big(\|g_4\|_{L^2(\Omega)}^2 + \|\bm{g}_1\|_{L^2(Q)}^2 + \|\mathrm{div}\,\bm{g}_1\|_{L^2(Q)}^2
+ \|g_2\|_{L^2(0,T;L^{4/3}(\Omega))}^2 +  \|g_3\|_{L^2(Q)}^2\Big). \label{eslined}
\end{align}
On the other hand, using   the identity
$$
\nu(\vps)\mathrm{curl}\,\bm{w}
+\nu'(\vps)\nabla \vps \cdot \bm{w}^\perp =\nabla \vps \cdot (\nabla \rho)^\perp,
$$
where $\vv^\perp=(v_2,-v_1)^T$ for any vector $\vv=(v_1,v_2)^T$, we infer from \eqref{asu1} that
\begin{align}
\|\mathrm{curl}\,\bm{w}\|_{L^2(\Omega)}
& \leq C  \|\nu'(\vps)\nabla \vps\|_{L^\infty(\Omega)} \|\bm{w}^\perp\|_{L^2(\Omega)}
+ C\|\nabla \vps\|_{L^\infty(\Omega)}\|(\nabla \rho)^\perp\|_{L^2(\Omega)}
  \non\\
 &\quad + C \|\mathrm{curl}\,\bm{g}_1\|_{L^2(\Omega)}\non\\
&\leq    C \Big(\|\bm{w} \|_{L^2(\Omega)} +  \|\rho\|_{H^1(\Omega)}
+  \|\mathrm{curl}\,\bm{g}_1\|_{L^2(\Omega)}\Big).\label{curlw}
\end{align}
Thus, it follows from \eqref{rot} and \eqref{eslinea}--\eqref{eslineb} that
\begin{align}
&\|\bm{w}\|_{L^2(0,T;H^1(\Omega))}^2 \non\\
&\quad \leq  C\|\bm{w}\|_{L^2(Q)}^2 +C \|\rho\|_{L^2(0,T;H^1(\Omega))}^2 + C \|\mathrm{curl}\,\bm{g}_1\|_{L^2(Q)}^2 \non\\
&\quad  \leq C\Big(\|g_4\|_{L^2(\Omega)}^2
+ \|\bm{g}_1\|_{L^2(Q)}^2
+ \|\mathrm{curl}\,\bm{g}_1\|_{L^2(Q)}^2
+ \|g_2\|_{L^2(0,T;L^{4/3}(\Omega))}^2
+ \|g_3\|_{L^2(Q)}^2\Big).
\label{eslinee}
\end{align}

\textbf{Third estimate}. Testing \eqref{asphi1} with $-\Delta \rho$ gives
\begin{align}
& -\frac12\frac{\mathrm{d}}{\mathrm{d}t}\|\nabla \rho\|_{L^2(\Omega)}^2
+ \|\nabla \Delta \rho\|_{L^2(\Omega)}^2  \non\\
&\quad =
-\int_\Omega (\bus\cdot\nabla \rho)\Delta \rho\,\mathrm{d}x
-\int_\Omega \nabla (\bm{w}\cdot\nabla \vps)\cdot \nabla \Delta \rho\,\mathrm{d}x \non\\
&\qquad
- \int_\Omega \Psi''(\vps) (\Delta \rho + \bm{w} \cdot \nabla \vps + g_3  )\Delta \rho\,\mathrm{d}x
+ \int_\Omega  (\bm{w} \cdot \nabla \mus)\Delta \rho\,\mathrm{d}x\non\\
&\qquad +\int_\Omega (\nu'(\vp^*) \bus\cdot \bm{w})\Delta \rho\,\mathrm{d}x
-\int_\Omega \nabla g_3\cdot \nabla \Delta \rho\,\mathrm{d}x
-\int_\Omega g_2\Delta \rho\,\mathrm{d}x\non\\
&=: \sum_{j=8}^{14}I_j.
\label{ajesH1}
\end{align}
The terms on the right-hand side of \eqref{ajesH1} can be estimated as follows
\begin{align*}
I_8&\leq \|\bus\|_{L^4(\Omega)}\|\nabla \rho\|_{L^4(\Omega)}\|\Delta \rho\|_{L^2(\Omega)}\non\\
&\leq C\|\nabla \rho\|_{L^2(\Omega)}^\frac12\|\rho\|_{H^2(\Omega)}^\frac32\non\\
&\leq \frac{1}{12} \|\nabla \Delta \rho\|_{L^2(\Omega)}^2 + C\|\nabla \rho\|_{L^2(\Omega)}^2+ C\|\rho\|_{L^2(\Omega)}^2,
\end{align*}
\begin{align*}
I_9&\leq \|\nabla (\bm{w}\cdot\nabla \vps)\|_{L^2(\Omega)}\| \nabla \Delta \rho\|_{L^2(\Omega)}\non\\
&\leq C
\|\bm{w}\|_{L^4(\Omega)}\|\Delta \vps\|_{L^4(\Omega)}\|\nabla \Delta \rho\|_{L^2(\Omega)}
+ C \|\nabla \bm{w}\|_{L^2(\Omega)}\|\nabla  \vps\|_{L^\infty(\Omega)}\|\nabla \Delta \rho\|_{L^2(\Omega)}  \non\\
&\leq \frac{1}{12} \|\nabla \Delta \rho\|_{L^2(\Omega)}^2
+ C\|\bm{w}\|_{H^1(\Omega)}^2,
\end{align*}
\begin{align*}
I_{10} &\leq \|\Psi''(\vps)\|_{L^\infty(\Omega)} \big(\|\Delta \rho\|_{L^2(\Omega)} + \|\bm{w}\|_{L^2(\Omega)}\|\nabla \vps\|_{L^\infty(\Omega)}
+ \|g_3\|_{L^2(\Omega)} \big)\|\Delta \rho\|_{L^2(\Omega)}\non\\
&\leq \frac{1}{12} \|\nabla \Delta \rho\|_{L^2(\Omega)}^2 + C\|\nabla \rho\|_{L^2(\Omega)}^2 + C \|\bm{w}\|_{L^2(\Omega)}^2 +C \|g_3\|_{L^2(\Omega)}^2,
\end{align*}
\begin{align*}
I_{11}  & = - \int_\Omega   \mus(\bm{w}  \cdot \nabla\Delta \rho)\,\mathrm{d}x\non\\
&\leq \|\mus\|_{L^4(\Omega)}\|\bm{w}\|_{L^4(\Omega)}\|\nabla \Delta \rho\|_{L^2(\Omega)}
 \non\\
&\leq \frac{1}{12} \|\nabla \Delta \rho\|_{L^2(\Omega)}^2 + C\|\bm{w}\|_{L^4(\Omega)}^2,
\end{align*}
\begin{align*}
 I_{12} &\leq \|\nu'(\vp^*)\|_{L^\infty(\Omega)}\|\bus\|_{L^4(\Omega)}\| \bm{w}\|_{L^4(\Omega)}\|\Delta \rho\|_{L^2(\Omega)}\non\\
&\leq \frac{1}{12} \|\nabla \Delta \rho\|_{L^2(\Omega)}^2
+ C\|\nabla \rho\|_{L^2(\Omega)}^2 + C\|\bm{w}\|_{L^4(\Omega)}^2,
\end{align*}
\begin{align*}
I_{13}+I_{14}&\leq  \|\nabla g_3\|_{L^2(\Omega)} \| \nabla \Delta \rho\|_{L^2(\Omega)} + \|g_2\|_{L^{4/3}(\Omega)}\|\Delta \rho\|_{L^4(\Omega)}\non\\
&\leq \frac{1}{12} \|\nabla \Delta \rho\|_{L^2(\Omega)}^2
+ C\|\nabla \rho\|_{L^2(\Omega)}^2 + C\|\nabla g_3\|^2_{L^2(\Omega)}+ C\|g_2\|^2_{L^{4/3}(\Omega)}.
\end{align*}
Hence, we can deduce from \eqref{ajesH1} that
\begin{align}
& -\frac12\frac{\mathrm{d}}{\mathrm{d}t}\|\nabla \rho\|_{L^2(\Omega)}^2
+ \frac12 \|\nabla \Delta \rho\|_{L^2(\Omega)}^2  \non\\
&\quad  \leq  C\|\nabla \rho\|_{L^2(\Omega)}^2
+ C\|\rho\|_{L^2(\Omega)}^2 + C\|\bm{w}\|_{H^1(\Omega)}^2
+ C\|g_2\|_{L^{4/3}(\Omega)}^2 + C\| g_3\|_{H^1(\Omega)}.
\non
\end{align}
The (backward) Gronwall's lemma together with \eqref{eslinea} and  \eqref{eslinee} yield
\begin{align}
&\|\nabla \rho(t)\|_{L^2(\Omega)}^2+ \int_t^T\|\nabla \Delta \rho(s)\|_{L^2(\Omega)}^2\,\mathrm{d}s\non\\
&\quad  \leq C \|g_4\|_{H^1(\Omega)}^2
+ C\|\bm{g}_1\|_{L^2(Q)}^2
+ C\|\mathrm{curl}\,\bm{g}_1\|_{L^2(Q)}^2 + C\|g_2\|_{L^2(0,T;L^{4/3}(\Omega))}^2
\non\\
&\qquad +  C\|g_3\|_{L^2(0,T;H^1(\Omega))}^2,
\qquad \forall\, t\in [0,T],\label{eslinef}
\end{align}
which combined with \eqref{asmu1}, \eqref{eslinee} and \eqref{eslinef} implies
\begin{align}
\|\zeta\|_{L^2(0,T;H^1(\Omega))}^2
&\leq   C \|g_4\|_{H^1(\Omega)}^2
+ C\|\bm{g}_1\|_{L^2(Q)}^2
+ C\|\mathrm{curl}\,\bm{g}_1\|_{L^2(Q)}^2
\non\\
&\quad + C\|g_2\|_{L^2(0,T;L^{4/3}(\Omega))}^2
+ C\|g_3\|_{L^2(0,T;H^1(\Omega))}^2.
\label{eslineg}
\end{align}
Next, we infer from \eqref{asphi1} that
\begin{align}
\|\partial_t\rho\|_{(H^1(\Omega))'}
&\leq C\|\bus\cdot\nabla \rho\|_{L^{4/3}(\Omega)} + C\|\nabla  \zeta\|_{L^2(\Omega)}
+ C\|\Psi''(\vps)\|_{L^\infty(\Omega)}\| \zeta\|_{L^2(\Omega)} \non\\
&\quad + C\|\bm{w} \cdot \nabla \mus\|_{L^{4/3}(\Omega)} +  C\|\nu'(\vp^*)\|_{L^\infty(\Omega)}\| \bus\cdot \bm{w}\|_{L^{4/3}(\Omega)}
+ C\|g_2\|_{L^{4/3}(\Omega)}\non\\
&\leq C\|\bus\|_{L^4(\Omega)}\|\nabla \rho\|_{L^{2}(\Omega)}
+C\| \zeta\|_{H^1(\Omega)}
+ C\|\bm{w}\|_{L^4(\Omega)}\| \nabla \mus\|_{L^{2}(\Omega)}
 \non\\
&\quad  +  C\| \bus\|_{L^4(\Omega)}\| \bm{w}\|_{L^{2}(\Omega)}
+ C\|g_2\|_{L^{4/3}(\Omega)}\non\\
&\leq C\Big(\|\rho\|_{H^1(\Omega)} +   \|\zeta\|_{H^1(\Omega)}
+  \|\bm{w}\|_{L^4(\Omega)}+ \|g_2\|_{L^{4/3}(\Omega)}\Big).\non
\end{align}
Thus, we get
\begin{align}
\|\partial_t\rho\|_{L^2(0,T;(H^1(\Omega))')} ^2
&\leq   C \|g_4\|_{H^1(\Omega)}^2
+ C\|\bm{g}_1\|_{L^2(Q)}^2
+ C\|\mathrm{curl}\,\bm{g}_1\|_{L^2(Q)}^2
 \non\\
&\quad + C\|g_2\|_{L^2(0,T;L^{4/3}(\Omega))}^2 +  C\|g_3\|_{L^2(0,T;H^1(\Omega))}^2.
\label{eslineh}
\end{align}
Finally, from \eqref{curlw}, \eqref{eslinef}, we deduce
$\bm{w}\in L^4(0,T; \mathbf{H}^1(\Omega))$, provided that $\bm{g}_1\in L^\infty(0,T;\mathbf{L}^2(\Omega))$ and $\mathrm{curl}\,\bm{g}_1\in L^4(0,T;L^2(\Omega))$.

Collecting the estimates \eqref{eslined}, \eqref{eslinee} and   \eqref{eslinef}--\eqref{eslineh}, we arrive at the conclusion \eqref{estiLineb}.
Next, we observe that if $(\bm{w},\pi,\rho,\zeta)$ solves problem \eqref{wreg1}--\eqref{asphi1} with $\bm{g}_1=\bm{0}$, $g_2=g_3=g_4=0$, then it follows
from \eqref{estiLineb} that $(\bm{w},\pi,\rho,\zeta)=(\bm{0},0,0,0)$. This yields the uniqueness of solution due to linearity of the system.
 \end{proof}

We are ready to show that the adjoint system is uniquely solvable.

\begin{proposition}\label{adexe2}
Assume that $\Omega\subset \mathbb{R}^2$ is a bounded domain with smooth boundary $\partial\Omega$, $T>0$ and the assumptions $\mathbf{(S1)}$--$\mathbf{(S5)}$ are satisfied.
Let $R^*\in \widetilde{\mathcal{U}}$ be any control with its associated state denoted by  $(\bus,P^*,\vps,\mus)$. The adjoint system \eqref{asu}--\eqref{asp4bc} admits a
 weak solution $(\bm{w}^*,\pi^*,\rho^*,\zeta^*)$ satisfying \eqref{wreg1}, \eqref{wreg2} and the following estimate
\begin{align}
&\|\rho^*\|^2_{C([0,T];H^1(\Omega))\cap L^2(0,T;H^3(\Omega))\cap H^1(0,T;(H^1(\Omega))')} +\|\zeta^*\|_{L^2(0,T;H^1(\Omega))}^2 \non \\
&\qquad +\|\bm{w}^*\|_{L^4(0,T;H^1(\Omega))\cap L^\infty(0,T;L^2(\Omega))}^2 + \|\pi^*\|_{ L^2(0,T;H^2(\Omega))}^2\non\\
&\quad \leq C \|\alpha_2(\vps-\vp_Q)\|^2_{L^2(Q)} + C\|\alpha_1\,(\vps(T)-\vp_\Omega)\|^2_{H^1(\Omega)}.
\label{estiLineaj}
\end{align}
\end{proposition}
\begin{proof}
Take
$$
\bm{g}_1=\bm{0}, \quad g_2=\alpha_2(\vps-\vp_Q),\quad g_3=0,\quad g_4=\alpha_1\,(\vps(T)-\vp_\Omega).
$$
From the regularity of $\vps$, we easily verify that $g_2\in L^2(Q)$ and $g_4\in H^1(\Omega)$. Then the conclusion is a direct consequence of Lemma \ref{adexe}.
\end{proof}

Now we are able to eliminate the function $\xi$ from the variational inequality
\eqref{vug1} and, alternatively, establish a first-order necessary optimality condition via the adjoint state.

\begin{theorem}\label{NCLO2}
Assume that $\Omega\subset \mathbb{R}^2$ is a bounded domain with smooth boundary $\partial\Omega$, $T>0$ and the assumptions $\mathbf{(S1)}$--$\mathbf{(S5)}$ are satisfied.
Let $R^*\in\mathcal{U}_{\mathrm{ad}}$ be a locally optimal control of problem \eqref{costJ} with the associated state $(\bm{u}^*,P^*,\vps,\mu^*)$ and the adjoint
state $(\bm{w}^*,\pi^*,\rho^*,\zeta^*)$. Then $R^*$ satisfies the following variational inequality
\begin{align}
\label{vug2}
\int_Q\left(\rho^* + \beta R^*\right)(R-R^*)\,\mathrm{d}x\mathrm{d}t \ge 0, \qquad\forall\,R\in\mathcal{U}_{\mathrm{ad}}.
\end{align}
\end{theorem}
\begin{proof}
The proof follows the argument for \cite[Corollary 4.4]{SW21} with minor modifications.
Taking $(\vv, q, \xi, \eta)$ that is the unique solution to the linearized system \eqref{lsphi}--\eqref{lsini} (with $\bm{f}_1=\bm{0}$, $f_2=R-R^*$ and $f_3=0$)
as test functions in the adjoint system \eqref{asu}--\eqref{asini}, adding the results together and using integration by parts, we have
\begin{align}
&\alpha_1\iO\left(\vps(T)-\vp_\Omega\right)\xi(T)\,\mathrm{d}x +\alpha_2\int_Q\,\left(\vps-\vp_Q\right)\xi\,\mathrm{d}x\mathrm{d}t\non\\
&\quad = \int_0^T \frac{\mathrm{d}}{\mathrm{d}t}\left(\int_\Omega \rho^* \xi\,\mathrm{d}x\right) \mathrm{d}t
-\int_0^T\,\langle\dt \rho^*, \xi\rangle_{(H^1(\Omega))',H^1(\Omega)}\, \mathrm{d}t
- \int_Q \nabla \zeta^*\cdot\nabla \xi\,\mathrm{d}x\mathrm{d}t\non\\
&\qquad -\int_Q \Big[\bus\cdot\nabla \rho^*+\Psi''(\vps)\,\zeta^* -\bm{w}^* \cdot \nabla \mus -\nu'(\vps)\bus\cdot\bm{w}^*\Big]\,\xi\,\mathrm{d}x\mathrm{d}t\non\\
&\qquad +\int_Q \big( \zeta^* - \Delta \rho^* -  \bm{w}^*\cdot \nabla \vps\big) \eta\,\mathrm{d}x\mathrm{d}t
- \int_Q (\mathrm{div}\,\bm{w}^*)q\,\mathrm{d}x\mathrm{d}t\non\\
&\qquad +\int_Q \Big[\nu(\vps)\bm{w}^* - \nabla \pi^* + \rho^* \nabla\vps\Big]\cdot \vv\,\mathrm{d}x\mathrm{d}t
\non\\
&\quad =
 \int_Q \Big[\dt \xi  +  \mathrm{div}\,(\vps \vv) + \mathrm{div}\,(\xi \bus)
 - \Delta\eta\Big]\,\rho^* \, \mathrm{d}x\mathrm{d}t\non\\
&\qquad  + \int_Q \Big[\eta + \Delta\xi - \Psi''(\vps)\xi\Big]\, \zeta^*\,\mathrm{d}x\mathrm{d}t
+  \int_Q  (\mathrm{div}\,\vv)\pi^*\,\mathrm{d}x\mathrm{d}t\non\\
&\qquad  + \int_Q\Big[\nu(\vps)\vv  + \nabla q - \eta\,\nabla\vps - \mus\nabla\xi +\nu'(\vps)\xi\bus\Big] \cdot \bm{w}^*\,\mathrm{d}x\mathrm{d}t\non\\
&\quad =   \int_Q (R-R^*)\, \rho^*\,\mathrm{d}x\mathrm{d}t.
\label{be2}
\end{align}
 Substitution of the identity \eqref{be2} into \eqref{vug1} yields the conclusion \eqref{vug2}.
 \end{proof}

Recalling that  $\mathcal{U}_{\mathrm{ad}}$ is a nonempty, closed convex subset of $L^2(Q)$, in the case $\beta>0$ the necessary condition \eqref{vug2} is
equivalent to the following projection formula (cf. \cite{To}):
 \begin{corollary}
 When $\beta>0$, the locally optimal control $R^*$ is the $L^2(Q)$-orthogonal projection of
$-\beta^{-1}\rho^*$ onto $\mathcal{U}_{\mathrm{ad}}$:
\begin{align}
R^*(x,t)=\max\left\{R_{\mathrm{min}}(x,t),\ \min\{-\beta^{-1}\rho^*(x,t),\,R_{\mathrm{max}}(x,t)\}\right\},\quad \text{for a.a.}\ (x,t)\in Q.\non
\end{align}
 \end{corollary}

\section{A second-order sufficient condition for strict local optimality}
\setcounter{equation}{0}

Since our optimal control problem is not convex, the
first-order necessary optimality conditions are not sufficient.
The aim of this section is to establish a second-order sufficient condition for strict local optimality.

\subsection{Differentiability of the control-to-costate operator}

In view of Proposition \ref{adexe2}, we are able to define the control-to-costate operator that maps any control $R\in\widetilde{\mathcal{U}}$ onto its corresponding adjoint state.

\begin{definition}\label{def:cotoco}
Let the assumptions of Proposition \ref{adexe2} be satisfied. Set
$$
\mathcal{Z}:= C([0,T];H^1(\Omega))\cap L^2(0,T;H^3(\Omega))\cap H^1(0,T;(H^1(\Omega))').
$$
We define the \textbf{control-to-costate operator}
\begin{align}
\mathcal{T}:  \widetilde{\mathcal{U}} \to \mathcal{Z},\quad R\mapsto \mathcal{T}(R)=\rho, \label{coToco}
\end{align}
where $(\bm{w},\pi,\rho,\zeta)$ is the unique weak solution to the adjoint system \eqref{asu}--\eqref{asp4bc} corresponding to the given control function
$R\in \widetilde{\mathcal{U}}$.
\end{definition}

First, we establish the Lipschitz continuity of the control-to-costate operator.

\begin{proposition}\label{Lipcotoco}
Assume that $\Omega\subset \mathbb{R}^2$ is a bounded domain with smooth boundary $\partial\Omega$, $T>0$ and the assumptions $\mathbf{(S1)}$--$\mathbf{(S5)}$ are satisfied.
For any given controls $R^*, R^\sharp\in \widetilde{\mathcal{U}}$, we denote their adjoint states by
$(\bm{w}^*,\pi^*,\rho^*,\zeta^*)$, $(\bm{w}^\sharp,\pi^\sharp,\rho^\sharp,\zeta^\sharp)$, respectively. Then it holds
\begin{align}
&\|\rho^*-\rho^\sharp\|_{C([0,T];H^1(\Omega))\cap L^2(0,T;H^3(\Omega))\cap H^1(0,T;(H^1(\Omega))')}
 + \|\zeta^*-\zeta^\sharp\|_{L^2(0,T;H^1(\Omega))}\non\\
&\qquad +\|\bm{w}^*-\bm{w}^\sharp\|
_{L^\infty(0,T;L^2(\Omega))\cap L^4(0,T;H^1(\Omega))}
+ \|\pi^*-\pi^\sharp\|_{L^2(0,T;H^2(\Omega))}\non\\
&\quad \le  C\|R^*-R^\sharp\|_{L^2(Q)},
\label{stabuLip}
\end{align}
where $C>0$ may depend on $K_1$, $K_2$, $r_1$ and on the parameters of the system.
\end{proposition}
\begin{proof}
Let us set
$$
\bm{w}=\bm{w}^*-\bm{w}^\sharp,\quad \pi=\pi^*-\pi^\sharp,\quad \rho=\rho^*-\rho^\sharp,\quad \zeta=\zeta^*-\zeta^\sharp.
$$
Besides, we denote by $(\u^*,P^*,\vp^*,\mu^*)$, $(\u^\sharp,P^\sharp,\vp^\sharp,\mu^\sharp)$ the associated states corresponding to $R^*$, $R^\sharp$, respectively.
Then $(\bm{w},\pi,\rho,\zeta)$ satisfies the following problem
\begin{align*}
& \nu(\vp^*)\bm{w} =\nabla \pi - \rho\nabla\vp^* +\bm{g}_1,&\mbox{a.e. in $Q$,}\\[1mm]
&\mathrm{div} \,\bm{w} =0,&\mbox{a.e. in  $Q$,}\\[1mm]
&\Delta \pi =\mathrm{div}\,(\rho\nabla\vp^*) + \nu'(\vp^*)\nabla \vp^*\cdot \bm{w} -\mathrm{div}\,\bm{g}_1, &\mbox{a.e. in $Q$,}\\[1mm]
& \zeta = \Delta \rho + \bm{w} \cdot \nabla \vp_1 + g_3, &\mbox{a.e. in  $Q$,}\\[1mm]
&\dn \rho  =\bm{w}\cdot{\bf n}= \dn \pi= 0,&\mbox{a.e. on $\Sigma$,}\\[1mm]
&\rho|_{t=T} = g_4,&\mbox{a.e. in $\Omega$,}
\end{align*}
and
\begin{align*}
&-\langle\dt \rho, \psi\rangle_{(H^1)',H^1}
-\int_\Omega (\u^*\cdot\nabla \rho)\psi\,\mathrm{d}x
-\int_\Omega \nabla  \zeta \cdot\nabla \psi\,\mathrm{d}x
-\int_\Omega \Psi''(\vp^*) \zeta\psi\,\mathrm{d}x \non\\
&\qquad \qquad
+\int_\Omega  (\bm{w} \cdot \nabla \mu^*)\psi\,\mathrm{d}x
+\int_\Omega (\nu'(\vp^*) \u_1\cdot \bm{w})\psi\,\mathrm{d}x = \int_\Omega g_2\psi\,\mathrm{d}x,
\end{align*}
for almost all $t\in(0,T)$ and all $\psi\in H^1(\Omega)$.
Here, we have
\begin{align*}
\bm{g}_1&= -\bm{w}^\sharp\int_0^1\nu'(s\vp^*+(1-s)\vp^\sharp)(\vp^*-\vp^\sharp)\,\mathrm{d}s -\rho^\sharp\nabla(\vp^*-\vp^\sharp),\\
g_2 &=  (\u^*-\u^\sharp)\cdot\nabla \rho^\sharp
+ \zeta^\sharp \int_0^1\Psi^{(3)}(s\vp^*+(1-s)\vp^\sharp)(\vp^*-\vp^\sharp)\,\mathrm{d}s\\
&\quad -\bm{w}^\sharp\cdot\nabla (\mu^*-\mu^\sharp) -\nu'(\vp^*)\bm{w}^\sharp\cdot(\u^*-\u^\sharp)\non\\
&\quad - (\u^\sharp\cdot\bm{w}^\sharp)\int_0^1\nu''(s\vp^*+(1-s)\vp^\sharp)(\vp^*-\vp^\sharp)\,\mathrm{d}s  +\alpha_2(\vp^*-\vp^\sharp),\\
g_3&=\bm{w}^\sharp\cdot\nabla(\vp^*-\vp^\sharp),\qquad g_4=\alpha_1(\vp^*(T)-\vp^\sharp(T)).
\end{align*}
Again, we exploit Taylor's formula in $\bm{g}_1$ and $g_2$. Then using estimates \eqref{stabu1}, \eqref{K1}, \eqref{K2} and \eqref{estiLineaj}, we can obtain
\begin{align*}
\|\bm{g}_1\|_{L^2(0,T;H^1(\Omega))}^2
&\leq C\|\vp^*-\vp^\sharp\|_{C([0,T];L^\infty(\Omega))}^2  \int_0^T \|\bm{w}^\sharp\|_{H^1(\Omega)}^2\,\mathrm{d}t\\
&\quad + C \|\bm{w}^\sharp\|_{L^\infty(0,T;L^2(\Omega))}^2 \int_0^T \|\vp^*-\vp^\sharp\|_{W^{1,\infty}(\Omega)}^2 \,\mathrm{d}t\non\\
&\quad + C \|\nabla (\vp^*-\vp^\sharp)\|_{C([0,T];L^4(\Omega))}^2  \int_0^T \|\rho^\sharp\|_{W^{1,4}(\Omega)}^2\,\mathrm{d}t\non\\
&\quad + C \|\rho^\sharp\|_{C([0,T];L^4(\Omega))}^2 \int_0^T \|\vp^*-\vp^\sharp\|_{W^{2,4}(\Omega)}^2\,\mathrm{d}t\\
&\leq C\|R^*-R^\sharp\|_{L^2(Q)}^2,
\end{align*}
\begin{align*}
\|\bm{g}_1\|_{L^\infty(0,T;L^2(\Omega))}^2
&\leq C \|\bm{w}^\sharp\|_{L^\infty(0,T;L^2(\Omega))}^2 \|\vp^*-\vp^\sharp\|_{C([0,T];L^\infty(\Omega))}^2\\
&\quad +C \|\rho^\sharp\|_{C([0,T];L^4(\Omega))}^2 \|\nabla (\vp^*-\vp^\sharp)\|_{C([0,T];L^4(\Omega))}^2\\
&\leq   C\|R^*-R^\sharp\|_{L^2(Q)}^2,
\end{align*}
\begin{align*}
\|\mathrm{curl}\,\bm{g}_1\|_{L^4(0,T;L^2(\Omega))}^4
&\leq C \|\vp^*-\vp^\sharp\|_{C([0,T];L^\infty(\Omega))}^4 \int_0^T \|\bm{w}^\sharp\|_{H^1(\Omega)}^4\,\mathrm{d}t \\
&\quad +C \|\bm{w}^\sharp\|_{L^\infty(0,T;L^2(\Omega))}^4 \int_0^T \|\vp^*-\vp^\sharp\|_{W^{1,\infty}(\Omega)}^4 \,\mathrm{d}t \\
&\quad +C \|\rho^\sharp\|_{C([0,T];H^1(\Omega))}^4\int_0^T \|\nabla (\vp^*-\vp^\sharp)\|_{L^\infty(\Omega)}^4 \,\mathrm{d}t\\
&\leq   C\|R^*-R^\sharp\|_{L^2(Q)}^4,
\end{align*}
\begin{align*}
\|g_2\|_{L^2(0,T; L^{4/3}(\Omega))}^2
& \leq C\|\nabla \rho^\sharp\|_{C([0,T];L^2(\Omega))}^2 \int_0^T\|\u^*-\u^\sharp\|_{L^4(\Omega)}^2\,\mathrm{d}t\\
&\quad +C \|\vp^*-\vp^\sharp\|_{C([0,T];L^2(\Omega))}^2\int_0^T\|\zeta^\sharp \|_{L^4(\Omega)}^2\,\mathrm{d}t\\
&\quad +C \|\bm{w}^\sharp\|_{L^\infty(0,T;L^2(\Omega))}^2 \int_0^T\|\nabla (\mu^*-\mu^\sharp)\|_{L^4(\Omega)}^2\,\mathrm{d}t\\
&\quad +C \|\bm{w}^\sharp\|_{L^\infty(0,T;L^2(\Omega))}^2 \int_0^T\|\u^*-\u^\sharp\|_{L^4(\Omega)}^2\,\mathrm{d}t\\
&\quad +C \|\bm{w}^\sharp\|_{L^\infty(0,T;L^2(\Omega))}^2 \|\vp^*-\vp^\sharp\|_{C([0,T];L^\infty(\Omega))}^2
\int_0^T \|\u^\sharp\|_{L^4(\Omega)}^2\,\mathrm{d}t\\
&\quad +C \|\vp^*-\vp^\sharp\|_{L^2(Q)}^2\\
&\leq C\|R^*-R^\sharp\|_{L^2(Q)}^2,
\end{align*}
\begin{align*}
\|g_3\|_{L^2(0,T;H^1(\Omega))}^2
&\leq C\int_0^T\|\bm{w}^\sharp\|_{H^1(\Omega)}^2 \|\nabla(\vp^*-\vp^\sharp)\|_{L^2(\Omega)} \|\nabla(\vp^*-\vp^\sharp)\|_{H^2(\Omega)}\,\mathrm{d}t\\
&\quad +C \|\bm{w}^\sharp\|_{L^\infty(0,T;L^2(\Omega))}^2 \int_0^T \|\nabla(\vp^*-\vp^\sharp)\|_{W^{1,\infty}(\Omega)}^2\,\mathrm{d}t\\
&\leq C\|R^*-R^\sharp\|_{L^2(Q)}^2,
\end{align*}
\begin{align*}
\|g_4\|_{H^1(\Omega)}^2\leq \|\vp^*-\vp^\sharp\|^2_{C([0,T];H^1(\Omega))} \leq C\|R^*-R^\sharp\|_{L^2(Q)}^2.
\end{align*}
Taking the above estimates into account, we can apply Lemma \ref{adexe} and deduce   that
\begin{align}
&\|\rho\|^2_{C([0,T];H^1(\Omega))\cap L^2(0,T;H^3(\Omega))\cap H^1(0,T;(H^1(\Omega))')} +\|\zeta\|_{L^2(0,T;H^1(\Omega))}^2 \non \\
&\qquad +\|\bm{w}\|_{L^\infty(0,T;L^2(\Omega))\cap L^4(0,T;H^1(\Omega))}^2
+ \|\pi\|_{ L^2(0,T;H^2(\Omega))}^2\non\\
&\quad \leq C  \|\bm{g}_1\|^2_{L^2(0,T;H^1(\Omega))} +  C  \|\bm{g}_1\|^2_{L^\infty(0,T;L^2(\Omega))} + C  \|\mathrm{curl}\,\bm{g}_1\|^2_{L^4(0,T;L^2(\Omega))}
\non \\
&\qquad + C\|g_2\|^2_{L^2(0,T;L^{4/3}(\Omega))} + C \|g_3\|^2_{L^2(0,T;H^1(\Omega))} + C\|g_4\|^2_{H^1(\Omega)} \non\\
 &\quad \le  C\|R^*-R^\sharp\|^2_{L^2(Q)},
\label{estiLineg}
\end{align}
which yields the required estimate \eqref{stabuLip}.
\end{proof}

Next, we show the Fr\'{e}chet differentiability of the control-to-costate operator.

\begin{proposition}\label{Diffcotoco}
Assume that $\Omega\subset \mathbb{R}^2$ is a bounded domain with smooth boundary $\partial\Omega$, $T>0$ and the assumptions $\mathbf{(S1)}$--$\mathbf{(S5)}$ are satisfied.

(1) For any $R^*\in \widetilde{\mathcal{U}}$, the control-to-costate operator $\mathcal{T}: \widetilde{\mathcal{U}}\to \mathcal{Z}$ defined in \eqref{coToco} is
Fr\'{e}chet differentiable at $R^*$ as a mapping from $\mathcal{U}$ into $\mathcal{Z}$. The Fr\'{e}chet derivative $D\mathcal{T}(R^*)\in \mathcal{L}(\mathcal{U},\mathcal{Z})$
can be determined as follows.
Given $R^*\in \widetilde{\mathcal{U}}$, we denote by $(\u^*,P^*,\varphi^*,\mu^*)$ its associated state and by $(\bm{w}^*,\pi^*,\rho^*,\zeta^*)$ its adjoint state.
Besides, for any $h\in \mathcal{U}$, we denote by $(\vv^h,q^h,\xi^h,\eta^h)$ the unique solution to problem \eqref{lsphi}--\eqref{lsini} at $(\u^*,P^*,\varphi^*,\mu^*)$
with $\bm{f}_1=\bm{0}$, $f_2=h$, $f_3=0$.
 Then, it holds
\begin{align}
D\mathcal{T}(R^*)h=\widetilde{\rho}^h,\label{Dtt}
\end{align}
where $(\widetilde{\bm{w}}^h,\widetilde{\pi}^h,\widetilde{\rho}^h,\widetilde{\zeta}^h)$ with the regularity
\begin{align*}
& \widetilde{\rho}^h \in \mathcal{Z},\quad \widetilde{\zeta}^h\in L^2(0,T;H^1(\Omega)), \\
&  \widetilde{\bm{w}}^h\in L^\infty(0,T;L^2(\Omega))\cap L^2(0,T; \mathbf{H}^1(\Omega)\cap \mathbf{H}_\sigma), \quad \widetilde{\pi}^h\in L^2(0,T;H^2(\Omega)\cap L^2_0(\Omega)),
\end{align*}
is a weak solution to the following linear system
\begin{align}
\label{asuz}
& \nu(\vps)\widetilde{\bm{w}} =\nabla \widetilde{\pi} - \widetilde{\rho}\nabla\vps +\bm{g}_1,&\mbox{a.e. in $Q$,}\\[1mm]
\label{aspz}
&\mathrm{div} \,\widetilde{\bm{w}} =0,&\mbox{a.e. in  $Q$,}\\[1mm]
\label{asmuz}
& \widetilde{\zeta} = \Delta \widetilde{\rho} + \widetilde{\bm{w}} \cdot \nabla \vps + g_3, &\mbox{a.e. in  $Q$,}\\[1mm]
\label{asbcz}
&\dn \widetilde{\rho} =\widetilde{\bm{w}}\cdot{\bf n}= \dn \widetilde{\pi}= 0,&\mbox{a.e. on $\Sigma$,}\\[1mm]
\label{asiniz}
&\widetilde{\rho}|_{t=T} = g_4,&\mbox{a.e. in $\Omega$,}
\end{align}
and
\begin{align}
\label{asphiz}
&-\langle\dt \widetilde{\rho}, \psi\rangle_{(H^1)',H^1}
-\int_\Omega (\bus\cdot\nabla \widetilde{\rho})\psi\,\mathrm{d}x -\int_\Omega \nabla  \widetilde{\zeta} \cdot\nabla \psi\,\mathrm{d}x
-\int_\Omega \Psi''(\vps) \widetilde{\zeta}\psi\,\mathrm{d}x \non\\
&\qquad \qquad +\int_\Omega  (\widetilde{\bm{w}} \cdot \nabla \mus)\psi\,\mathrm{d}x
+\int_\Omega (\nu'(\vp^*) \bus\cdot \widetilde{\bm{w}})\psi\,\mathrm{d}x = \int_\Omega g_2\psi\,\mathrm{d}x,
\end{align}
 for almost all $t\in(0,T)$ and all $\psi\in H^1(\Omega)$.
Here,
\begin{align}
\bm{g}_1 &= -\nu'(\vps)\xi^h\bm{w}^*-\rho^*\nabla \xi^h, \label{dg1}\\
g_2&= \vv^h\cdot\nabla\rho^* + \Psi^{(3)}(\vps)\xi^h\zeta^*
-\bm{w}^*\cdot\nabla\eta^h-\nu''(\vps)\xi^h\u^*\cdot\bm{w}^*\non\\
&\quad -\nu'(\vps)\vv^h\cdot\bm{w}^*+\alpha_2\xi^h, \label{dg2}\\
g_3&= \bm{w}^*\cdot\nabla \xi^h,\label{dg3}\\
g_4&= \alpha_1\xi^h(T).\label{dg4}
\end{align}

(2) The Fr\'{e}chet derivative of the control-to-costate operator $\mathcal{T}$ is Lipschitz continuous in $\widetilde{\mathcal{U}}$, i.e.,
for any $R^*, R^\sharp \in \widetilde{\mathcal{U}}$, it holds
\begin{align}
\|D\mathcal{T}(R^*)  - D\mathcal{T}(R^\sharp) \|_{\mathcal{L}(\mathcal{U},\mathcal{Z})}
\leq C\|R^*-R^\sharp\|_{L^2(Q)},
\label{DeLipT}
\end{align}
where $C>0$ may depend on $K_1$, $K_2$, $r_1$ and on the parameters of the system.
\end{proposition}
\begin{proof}
By definition, for the given control function $R^*$, $(\u^*,P^*,\varphi^*,\mu^*)$ is the unique strong solution to problem \eqref{bdini}--\eqref{CHHS},
while  $(\bm{w}^*,\pi^*,\rho^*,\zeta^*)$ is the unique weak solution to the adjoint system \eqref{asu}--\eqref{asp4bc} with $\bm{g}_1=\bm{0}$,
$g_2=\alpha_2(\vps-\vp_Q)$, $g_3=0$ and $g_4=\alpha_1\,(\vps(T)-\vp_\Omega)$ therein.

  Let us first establish the solvability of problem \eqref{asuz}--\eqref{dg4}. Thanks to Lemma \ref{line}, we easily find  $g_4=\alpha_1\xi^h(T) \in H^1(\Omega)$.
  It is straightforward to check that $\bm{g}_1\cdot \mathbf{n}=0$ on $\Sigma$.
Using the regularity properties of $(\u^*,P^*,\varphi^*,\mu^*)$, $(\bm{w}^*,\pi^*,\rho^*,\zeta^*)$ and $(\vv^h,q^h,\xi^h,\eta^h)$, we can further verify that
\begin{align}
\begin{cases}
\bm{g}_1\in L^2(0,T;\mathbf{H}^1(\Omega))\cap L^\infty(0,T;\mathbf{L}^2(\Omega)),\quad \mathrm{curl}\,\bm{g}_1\in L^4(0,T;L^2(\Omega)),\\
g_2\in L^2(0,T;L^{4/3}(\Omega)),\quad g_3\in L^2(0,T;H^1(\Omega)).
\end{cases}
\label{ggg1}
\end{align}
Indeed, it is straightforward to check that
\begin{align*}
&\|\bm{g}_1\|^2_{L^2(0,T;H^1(\Omega))}\\
& \quad \leq
 \int_0^T \|\nu'(\vps)\|_{L^\infty(\Omega)}^2\|\xi^h\|_{L^4(\Omega)}^2 \|\bm{w}^*\|_{L^4(\Omega)}^2+ \|\rho^*\|_{L^4(\Omega)}^2\|\nabla \xi^h\|_{L^4(\Omega)}^2\,\mathrm{d}t\non\\
 &\qquad + \int_0^T \|\nu''(\vps)\|_{L^\infty(\Omega)}^2\|\nabla \vps\|_{L^\infty(\Omega)}^2\|\xi^h\|_{L^4(\Omega)}^2 \|\bm{w}^*\|_{L^4(\Omega)}^2\,\mathrm{d}t\\
 &\qquad +\int_0^T  \|\nu'(\vps)\|_{L^\infty(\Omega)}^2\|\nabla \xi^h\|_{L^4(\Omega)}^2 \|\bm{w}^*\|_{L^4(\Omega)}^2\,\mathrm{d}t \\
 &\qquad + \int_0^T \|\nu'(\vps)\|_{L^\infty(\Omega)}^2\| \xi^h\|_{L^\infty(\Omega)}^2 \|\nabla \bm{w}^*\|_{L^2(\Omega)}^2\,\mathrm{d}t\\
 &\qquad + \int_0^T \|\rho^*\|_{L^4(\Omega)}^2\|\xi^h\|_{W^{2,4}(\Omega)}^2
 +\|\nabla \rho^*\|_{L^4(\Omega)}^2\|\nabla \xi^h\|_{L^4(\Omega)}^2\,\mathrm{d}t\\
 &\quad \leq C\|h\|_{L^2(Q)}^2,
\end{align*}
\begin{align*}
\|\bm{g}_1\|^2_{L^\infty(0,T;L^2(\Omega))}
&  \leq C\|\xi^h\|_{C([0,T];L^\infty(\Omega))}^2\|\bm{w}^*\|_{L^\infty(0,T;L^2(\Omega))}^2
\\
&\quad + C\|\rho^*\|_{C([0,T];L^4(\Omega))}^2\|\nabla \xi^h\|_{C([0,T];L^4(\Omega))}^2\\
 &  \leq C\|h\|_{L^2(Q)}^2,
\end{align*}
\begin{align*}
& \|\mathrm{curl}\,\bm{g}_1\|^4_{L^4(0,T;L^2(\Omega))} \\
& \quad   \leq C\|\xi^h\|_{C([0,T];L^\infty(\Omega))}^4 \int_0^T\|\bm{w}^*\|_{H^1(\Omega)}^4\,\mathrm{d}t\\
&\qquad +C\|\bm{w}^*\|_{L^\infty(0,T;L^2(\Omega))}^4
\|\xi^h\|_{C([0,T];H^1(\Omega))}^2
\int_0^T \|\xi^h\|_{H^3(\Omega)}^2 \,\mathrm{d}t\\
&\qquad +C\|\nabla \rho^*\|_{C([0,T];H^1(\Omega))}^4\|\xi^h\|_{C([0,T];H^1(\Omega))}^2
\int_0^T \|\xi^h\|_{H^3(\Omega)}^2 \,\mathrm{d}t\\
 &\quad \leq C\|h\|_{L^2(Q)}^4,
\end{align*}
\begin{align*}
&\|g_2\|_{L^2(0,T;L^{4/3}(\Omega)))}^2 \\
&\quad \leq \int_0^T \|\vv^h\|_{L^4(\Omega)}^2\|\nabla\rho^*\|_{L^2(\Omega)}^2
+  \|\Psi^{(3)}(\vps)\|_{L^\infty(\Omega)}^2 \|\xi^h\|_{L^4(\Omega)}^2\|\zeta^*\|_{L^2(\Omega)}^2\,\mathrm{d}t
\\
&\qquad +\int_0^T \|\bm{w}^*\|_{L^2(\Omega)}\|\bm{w}^*\|_{H^1(\Omega)} \|\eta^h\|_{L^2(\Omega)}\|\Delta \eta^h\|_{L^2(\Omega)}\,\mathrm{d}t\\
&\qquad +  \int_0^T \|\nu''(\vps)\|_{L^\infty(\Omega)}^2\|\xi^h\|_{L^\infty(\Omega)} ^2 \|\u^*\|_{L^2(\Omega)}^2\|\bm{w}^*\|_{L^4(\Omega)}^2\,\mathrm{d}t\\
&\qquad +\int_0^T \|\nu'(\vps)\|_{L^\infty(\Omega)}^2 \|\vv^h\|_{L^4(\Omega)}^2\|\bm{w}^*\|_{L^2(\Omega)}^2\,\mathrm{d}t +\alpha_2\int_0^T \|\xi^h\|_{L^2(\Omega)}^2\,\mathrm{d}t
\\
&\quad \leq C\|h\|_{L^2(Q)}^2,
\end{align*}
and
\begin{align*}
\|g_3\|_{L^2(0,T;H^1(\Omega))}^2
&  \leq  \int_0^T \|\nabla \bm{w}^*\|_{L^2(\Omega)}^2\|\nabla \xi^h\|_{L^2(\Omega)}\|\nabla \xi^h\|_{H^2(\Omega)}  \,\mathrm{d}t\non\\
&\quad + C\int_0^T \|\bm{w}^*\|_{L^2(\Omega)}^2\|\nabla \xi^h\|_{W^{1,\infty}(\Omega)}^2\,\mathrm{d}t\non\\
&\leq C\|h\|_{L^2(Q)}^2.
\end{align*}
Combining the above estimates, we arrive at the conclusion \eqref{ggg1}.

Thanks to \eqref{ggg1}, we are able to apply Lemma \ref{adexe} and conclude that the linear problem \eqref{asuz}--\eqref{dg4}
  admits a unique weak solution $(\widetilde{\bm{w}}^h,\widetilde{\pi}^h,\widetilde{\rho}^h,\widetilde{\zeta}^h)$ with the desired regularity. Moreover, it holds
  \begin{align}
&\|\widetilde{\rho}^h\|^2_{C([0,T];H^1(\Omega))\cap L^2(0,T;H^3(\Omega))\cap H^1(0,T;(H^1(\Omega))')} +\|\widetilde{\zeta}^h\|_{L^2(0,T;H^1(\Omega))}^2 \non \\
&\qquad +\|\widetilde{\bm{w}}^h\|_{L^\infty(0,T;L^2(\Omega))\cap L^4(0,T;H^1(\Omega))}^2 + \|\widetilde{\pi}^h\|_{ L^2(0,T;H^2(\Omega))}^2\non\\
&\quad \leq C \|\bm{g}_1\|^2_{L^2(0,T;H^1(\Omega))} + \|\bm{g}_1\|^2_{L^\infty(0,T;L^2(\Omega))}+ \|\mathrm{curl}\,\bm{g}_1\|^2_{L^4(0,T;L^2(\Omega))} \non \\
&\qquad + C\|g_2\|^2_{L^2(0,T;L^{4/3}(\Omega))} +  C\|g_3\|^2_{L^2(0,T;H^1(\Omega))} + C\|g_4\|^2_{H^1(\Omega)} \non\\
&\quad \leq C\|h\|_{L^2(Q)}^2.
\label{estiLinedd}
\end{align}

  Next, we show that $(\widetilde{\bm{w}}^h,\widetilde{\pi}^h,\widetilde{\rho}^h,\widetilde{\zeta}^h)$ gives the Fr\'{e}chet derivative of $\mathcal{T}$.
  Similar to Proposition \ref{1stFD}, since $R^*\in \widetilde{\mathcal{U}}$, there is some $\lambda>0$ sufficiently small such that $R^*+h\in \widetilde{\mathcal{U}}$,
  whenever $h\in \mathcal{U}$ and $\|h\|_{\mathcal{U}}< \lambda$. Below we shall only consider such small perturbations $h$.
For $R^*+h$, we denote by $(\u^h,P^h,\vph,\muh)$ its associated  state
and by
$(\bm{w}^h,\pi^h,\rho^h,\zeta^h)$ its associated adjoint state.
Define the differences
\begin{align*}
&\widehat{y}^h=\rho^h-\rho^*-\widetilde{\rho}^h, \quad \widehat{z}^h=\zeta^h-\zeta^*-\widetilde{\zeta}^h,\quad \bm{a}^h=\bm{w}^h-\bm{w}^*-\widetilde{\bm{w}}^h,
\quad b^h=\pi^h-\pi^*-\widetilde{\pi}^h.
\end{align*}
For all admissible perturbations $h$,  we can check that
\begin{align*}
 \widehat{y}^h &\in \mathcal{Z},\quad \widehat{z}^h\in L^2(0,T;H^1(\Omega)),\non\\[1mm]
\bm{a}^h&\in L^2(0,T;\mathbf{H}^1(\Omega)\cap\mathbf{H}_\sigma),\quad b^h\in L^2(0,T;H^2(\Omega)\cap L_0^2(\Omega)).
\end{align*}
Then $(\bm{a}^h, b^h, \yh, \zh)$ is a weak solution to the following problem:
\begin{align}
\label{asuy}
& \nu(\vps)\bm{a}^h =\nabla b^h - \widehat{y}^h\nabla\vps +\widehat{\bm{g}}_1, &\mbox{a.e. in $Q$,}\\[1mm]
\label{aspy}
&\mathrm{div} \,\bm{a}^h =0,&\mbox{a.e. in  $Q$,}\\[1mm]
\label{asmuy}
& \widehat{z}^h = \Delta \widehat{y}^h + \bm{a}^h \cdot \nabla \vps + \widehat{g}_3, &\mbox{a.e. in  $Q$,}\\[1mm]
\label{asbcy}
&\dn \widehat{y}^h =\bm{a}^h\cdot{\bf n}= 0,&\mbox{a.e. on \,$\Sigma$,}\\[1mm]
\label{asiniy}
&\widehat{y}^h|_{t=T} = \widehat{g}_4,&\mbox{a.e. in $\Omega$,}
\end{align}
and
\begin{align}
\label{asphiy}
&-\langle\dt \widehat{y}^h, \psi\rangle_{(H^1)',H^1}
-\int_\Omega (\bus\cdot\nabla \widehat{y}^h)\psi\,\mathrm{d}x
-\int_\Omega \nabla  \widehat{z}^h \cdot\nabla \psi\,\mathrm{d}x
-\int_\Omega \Psi''(\vps) \widehat{z}^h\psi\,\mathrm{d}x \non\\
&\qquad \qquad +\int_\Omega  (\bm{a}^h \cdot \nabla \mus)\psi\,\mathrm{d}x
+\int_\Omega (\nu'(\vp^*) \bus\cdot \bm{a}^h)\psi\,\mathrm{d}x = \int_\Omega \widehat{g}_2\psi\,\mathrm{d}x,
\end{align}
 for almost all $t\in(0,T)$ and all $\psi\in H^1(\Omega)$. Using Taylor's formula, it is straightforward to check that
 \begin{align*}
 \widehat{\bm{g}}_1&= - (\bm{w}^h-\bm{w}^*)\int_0^1 \nu'(s\vp^h+(1-s)\vps)(\vph-\vps)\,\mathrm{d}s\\
 &\quad -\bm{w}^* \nu'(\vps)(\vph-\vps-\xi^h)\\
 &\quad -\frac12 \bm{w}^* \int_0^1\int_0^1 \nu''(sz\vph+(1-sz)\vps)(\vph-\vps)^2\mathrm{d}s\mathrm{d}z\\
 &\quad -(\rho^h-\rho^*)\nabla (\vph-\vps) -\rho^*\nabla(\vph-\vps-\xi^h),\\
 \widehat{g}_2&= \alpha_2(\vph-\vps-\xi^h) + (\uh-\bus)\cdot\nabla(\rho^h-\rho^*)
+(\uh-\bus-\bm{v}^h)\cdot\nabla \rho^*\\
&\quad + (\zeta^h-\zeta^*)\int_0^1\Psi^{(3)}(s\vph+(1-s)\vps)(\vph-\vps)\,\mathrm{d}s\\
&\quad + \zeta^* \Psi^{(3)}(\vps)(\vph-\vps-\xi^h) \\
&\quad +\frac12\zeta^*\int_0^1\int_0^1\Psi^{(4)} (sz\vph+(1-sz)\vps)(\vph-\vps)^2\,\mathrm{d}s\mathrm{d}z\\
&\quad - (\bm{w}^h-\bm{w}^*)\cdot\nabla (\mu^h-\mu^*)
-\bm{w}^*\cdot\nabla (\mu^h-\mu^*-\eta^h)\\
&\quad -\bus\cdot(\bm{w}^h-\bm{w}^*) \int_0^1\nu''(s\vph+(1-s)\vps)(\vph-\vps)\,\mathrm{d}s\\
&\quad - \nu''(\vps)(\vph-\vps-\xi^h)(\bus\cdot\bm{w}^*)\\
&\quad - \frac12 (\bus \cdot \bm{w}^*) \int_0^1\int_0^1 \nu^{(3)}(sz\vph+(1-sz)\vps)(\vph-\vps)^2\mathrm{d}s\mathrm{d}z\\
&\quad -\nu'(\vps)(\uh-\bus)\cdot(\bm{w}^h-\bm{w}^*)
-\nu'(\vps)\big(\uh-\bus-\bm{v}^h\big)\cdot\bm{w}^*,\\
\widehat{g}_3&=(\bm{w}^h-\bm{w}^*)\cdot\nabla(\vph-\vps) +\bm{w}^*\cdot\nabla(\vph-\vps-\xi^h),\\
\widehat{g}_4&= \alpha_1(\vph(T)-\vps(T)-\xi^h(T)).
\end{align*}

Applying Lemma \ref{adexe}, we can derive the following estimate
\begin{align}
&\|\widetilde{y}^h\|^2_{C([0,T];H^1(\Omega))\cap L^2(0,T;H^3(\Omega))\cap H^1(0,T;(H^1(\Omega))')} \non\\
&\quad \leq C \Big(\|\widehat{\bm{g}}_1\|^2_{L^2(0,T;H^1(\Omega))}
+ \|\widehat{g}_2\|^2_{L^2(0,T;L^{4/3}(\Omega))} +  \|\widehat{g}_3\|^2_{L^2(0,T;H^1(\Omega))} + \|\widehat{g}_4\|^2_{H^1(\Omega)}\Big)\non\\
&\quad \leq C\|h\|_{L^2(Q)}^4,
\label{estiLined}
\end{align}
for any $h\in \mathcal{U}\subset L^2(Q)$ with $\|h\|_{\mathcal{U}}< \lambda$.
To obtain \eqref{estiLined}, we note that the estimates \eqref{K1}--\eqref{K2} hold
for both $(\bus,P^*,\vps,\mus)$ and $(\uh,\ph,\vph,\muh)$. Besides, we recall that the control-to-costate operator is Lipschitz continuous (see \eqref{stabuLip}), namely,
\begin{align}
&\|\rho^h-\rho^*\|_{C([0,T];H^1(\Omega))\cap L^2(0,T;H^3(\Omega))\cap H^1(0,T;(H^1(\Omega))')}
 + \|\zeta^h-\zeta^*\|_{L^2(0,T;H^1(\Omega))}\non\\
&\qquad +\|\bm{w}^h-\bm{w}^*\|
_{L^\infty(0,T;L^2(\Omega))\cap L^4(0,T;H^1(\Omega))} + \|\pi^h-\pi^*\|_{L^2(0,T;H^2(\Omega))}\non\\
&\quad \leq C\|h\|_{L^2(Q)},
\label{stabu4}
\end{align}
for some $C>0$ independent of $h$.
Combining the above estimates with \eqref{DDD1} and \eqref{estiLineaj}, we can deduce that
\begin{align*}
  \|\widehat{\bm{g}}_1\|_{L^2(0,T;H^1(\Omega))}^2
&   \leq C\|\vph-\vps\|_{C([0,T];L^\infty(\Omega))}^2
\int_0^T \|\bm{w}^h-\bm{w}^*\|_{H^1(\Omega)}^2\,\mathrm{d}t \\
&\quad + C\|\nabla(\vph-\vps)\|_{C([0,T];L^4(\Omega))}^2
\int_0^T \|\bm{w}^h-\bm{w}^*\|_{L^4(\Omega)}^2\,\mathrm{d}t\\
&\quad + C\|\vph-\vps-\xi^h\|_{C([0,T];L^\infty(\Omega))}^2
\int_0^T \| \bm{w}^*\|_{H^1(\Omega)}^2\,\mathrm{d}t\\
&\quad + C \|\nabla (\vph-\vps-\xi^h)\|_{C([0,T];L^4(\Omega))}^2
\int_0^T \| \bm{w}^*\|_{L^4(\Omega)}^2\,\mathrm{d}t\non\\
&\quad + C \|\vph-\vps\|_{C([0,T];L^\infty(\Omega))}^4\int_0^T \|\bm{w}^*\|_{H^1(\Omega)}^2\,\mathrm{d}t \\
&\quad + C\|\nabla(\vph-\vps)\|_{C([0,T];L^8(\Omega))}^4
\int_0^T \|\bm{w}^*\|_{L^4(\Omega)}^2\,\mathrm{d}t\\
&\quad +C \|\rho^h-\rho^*\|^2_{C([0,T];H^1(\Omega))} \int_0^T \|\nabla(\vph-\vps)\|_{L^\infty(\Omega)}^2\,\mathrm{d}t\non\\
&\quad +C \|\rho^h-\rho^*\|^2_{C([0,T];L^4(\Omega))} \int_0^T \|\nabla(\vph-\vps)\|_{W^{1,4}(\Omega)}^2\,\mathrm{d}t\\
&\quad + C \|\rho^*\|_{C([0,T];H^1(\Omega))}^2 \int_0^T \|\nabla(\vph-\vps-\xi^h)\|_{L^\infty(\Omega)}^2\,\mathrm{d}t\non\\
&\quad + C\|\rho^*\|_{C([0,T];L^4(\Omega))}^2 \int_0^T \|\nabla(\vph-\vps-\xi^h)\|_{W^{1,4}(\Omega)}^2\,\mathrm{d}t\non\\
&\leq C\|h\|_{L^2(Q)}^4,
\end{align*}
\begin{align*}
\|\widehat{g}_2\|_{L^2(0,T;L^{4/3}(\Omega))}^2
& \leq C\int_0^T \|\vph-\vps-\xi^h\|_{L^2(\Omega)}^2\,\mathrm{d}t\\
&\quad + C \|\nabla(\rho^h-\rho^*)\|_{C([0,T];L^2(\Omega))}^2 \int_0^T \|\uh-\bus\|_{L^4(\Omega)}^2\,\mathrm{d}t\non\\
&\quad +C\|\nabla \rho^*\|_{C([0,T];L^2(\Omega))}^2\int_0^T  \|\uh-\bus-\bm{v}^h\|_{L^4(\Omega)}^2\,\mathrm{d}t\\
&\quad +C\|\vph-\vps\|_{C([0,T];L^4(\Omega))}^2\int_0^T \|\zeta^h-\zeta^*\|_{L^2(\Omega)}^2\,\mathrm{d}t\\
&\quad +C \|\vph-\vps-\xi^h\|_{C([0,T];L^4(\Omega))}^2\int_0^T \| \zeta^* \|_{L^2(\Omega)}^2\,\mathrm{d}t\\
&\quad +C \|\vph-\vps\|_{C([0,T];L^8(\Omega))}^4\int_0^T \| \zeta^* \|_{L^4(\Omega)}^2\,\mathrm{d}t\\
&\quad +C \|\bm{w}^h-\bm{w}^*\|_{L^\infty(0,T;L^2(\Omega))}^2 \int_0^T\|\nabla (\mu^h-\mu^*)\|_{L^4(\Omega)}^2\,\mathrm{d}t\\
&\quad +C  \|\bm{w}^*\|_{L^\infty(0,T;L^2(\Omega))}^2 \int_0^T\|\nabla (\mu^h-\mu^*-\eta^h)\|_{L^4(\Omega)}^2\,\mathrm{d}t\\
&\quad +C \|\bm{w}^h-\bm{w}^*\|_{L^\infty(0,T;L^2(\Omega))}^2 \|\vph-\vps\|_{C([0,T];L^\infty(\Omega))}^2\int_0^T \|\bus \|_{L^4(\Omega)}^2\,\mathrm{d}t\\
&\quad +C \|\vph-\vps-\xi^h\|_{C([0,T];L^\infty(\Omega))}^2 \|\bm{w}^*\|_{L^\infty(0,T;L^2(\Omega))}^2 \int_0^T \|\bus \|_{L^4(\Omega)}^2\,\mathrm{d}t\\
&\quad +C \|\vph-\vps\|_{C([0,T];L^\infty(\Omega))}^4 \|\bm{w}^*\|_{L^\infty(0,T;L^2(\Omega))}^2 \int_0^T \|\bus \|_{L^4(\Omega)}^2\,\mathrm{d}t\\
&\quad +C \|\bm{w}^h-\bm{w}^*\|_{L^\infty(0,T;L^2(\Omega))}^2 \int_0^T\|\uh-\bus\|_{L^4(\Omega)}^2\,\mathrm{d}t\\
&\quad +C \|\bm{w}^*\|_{L^\infty(0,T;L^2(\Omega))}^2 \int_0^T\|\uh-\bus-\bm{v}^h\|_{L^4(\Omega)}^2\,\mathrm{d}t\\
&\leq C\|h\|_{L^2(Q)}^4,
\end{align*}
\begin{align*}
\|\widehat{g}_3\|_{L^2(0,T;H^1(\Omega))}^2
&\leq C\int_0^T \|\bm{w}^h-\bm{w}^*\|_{H^1(\Omega)}^2 \|\nabla(\vph-\vps)\|_{L^2(\Omega)} \|\nabla(\vph-\vps)\|_{H^2(\Omega)}\,\mathrm{d}t\\
&\quad +C \|\bm{w}^h-\bm{w}^*\|_{L^\infty(0,T;L^2(\Omega))}^2
\int_0^T \|\nabla(\vph-\vps)\|_{W^{1,\infty}(\Omega)}^2\,\mathrm{d}t\\
&\quad +C\int_0^T \|\bm{w}^*\|_{H^1(\Omega)}^2\|\nabla(\vph-\vps-\xi^h)\|_{L^2(\Omega)} \|\nabla(\vph-\vps-\xi^h)\|_{H^2(\Omega)}\,\mathrm{d}t\\
&\quad +C\|\bm{w}^*\|_{L^\infty(0,T;L^2(\Omega))}^2
\int_0^T \| \nabla(\vph-\vps-\xi^h)\|_{W^{1,\infty}(\Omega)}^2\,\mathrm{d}t\\
&\leq C\|h\|_{L^2(Q)}^4,
\end{align*}
and
\begin{align*}
\|\widehat{g}_4\|_{H^1(\Omega)}^2&\leq C\|\vph-\vps-\xi^h\|_{C([0,T];H^1(\Omega))}^2\leq C\|h\|_{L^2(Q)}^4.
\end{align*}
Collecting the above estimates, we arrive at \eqref{estiLined}.
As a consequence, it holds
$$
\frac{\|\mathcal{T}(R^*+h)-\mathcal{T}(R^*) -\widetilde{\rho}^h\|_{\mathcal{Z}}}{\|h\|_{\mathcal{U}}} =\frac{\|\widehat{y}^h\|_{\mathcal{Z}}}{\|h\|_{\mathcal{U}}}
\leq C\|h\|_{\mathcal{U}}\to 0,\quad \text{as}\ \ \ \|h\|_{\mathcal{U}}\to 0.
$$
This completes the proof of the assertion (1).

Next, we prove assertion (2), that is, the Lipschitz continuity of the Fr\'{e}chet derivative of $\mathcal{T}$.

For two given controls $R^*, R^\sharp\in \widetilde{\mathcal{U}}$,
 we denote by $(\bus,P^*,\vp^*,\mu^*)$, $(\u^\sharp,P^\sharp,\vp^\sharp,\mu^\sharp)$ their associate states, by $(\vv^*,q^*,\xi^*,\eta^*)$,
 $(\vv^\sharp,q^\sharp,\xi^\sharp,\eta^\sharp)$ the corresponding solutions to the linear system \eqref{lsphi}--\eqref{lsini} at $(\u^*,P^*,\varphi^*,\mu^*)$
 and $(\u^\sharp,P^\sharp,\vp^\sharp,\mu^\sharp)$ with $\bm{f}_1=\bm{0}$, $f_2=h$, $f_3=0$, respectively,
 and by $(\widetilde{\bm{w}}^*,\widetilde{\pi}^*,\widetilde{\rho}^*,\widetilde{\zeta}^*)$,
 $(\widetilde{\bm{w}}^\sharp,\widetilde{\pi}^\sharp,\widetilde{\rho}^\sharp,\widetilde{\zeta}^\sharp)$
 the solutions to the linear system \eqref{asuz}--\eqref{dg4} corresponding to $R^*$ and $R^\sharp$.
 Define the differences
 $$
 \widetilde{\bm{w}}=\widetilde{\bm{w}}^*-\widetilde{\bm{w}}^\sharp,
 \quad \widetilde{\pi}=\widetilde{\pi}^*-\widetilde{\pi}^\sharp,
 \quad \widetilde{\rho}=\widetilde{\rho}^*-\widetilde{\rho}^\sharp,
 \quad \widetilde{\zeta}=\widetilde{\zeta}^*-\widetilde{\zeta}^\sharp.
 $$
 We find that $(\widetilde{\bm{w}},\widetilde{\pi},\widetilde{\rho},\widetilde{\zeta})$ is a solution to the following problem
 \begin{align*}
& \nu(\vps)\widetilde{\bm{w}} =\nabla \widetilde{\pi} - \widetilde{\rho}\nabla\vps +\widetilde{\bm{g}}_1,&\mbox{a.e. in $Q$,}\\[1mm]
&\mathrm{div} \,\widetilde{\bm{w}} =0,&\mbox{a.e. in  $Q$,}\\[1mm]
& \widetilde{\zeta} = \Delta \widetilde{\rho} + \widetilde{\bm{w}} \cdot \nabla \vps + \widetilde{g}_3, &\mbox{a.e. in  $Q$,}\\[1mm]
&\dn \widetilde{\rho} =\widetilde{\bm{w}}\cdot{\bf n}=  0, &\mbox{a.e. on $\Sigma$,}\\[1mm]
&\widetilde{\rho}|_{t=T} = \widetilde{g}_4,&\mbox{a.e. in $\Omega$,}
\end{align*}
and
\begin{align*}
&-\langle\dt \widetilde{\rho}, \psi\rangle_{(H^1)',H^1}
-\int_\Omega (\bus\cdot\nabla \widetilde{\rho})\psi\,\mathrm{d}x -\int_\Omega \nabla  \widetilde{\zeta} \cdot\nabla \psi\,\mathrm{d}x
-\int_\Omega \Psi''(\vps) \widetilde{\zeta}\psi\,\mathrm{d}x \non\\
&\qquad \qquad +\int_\Omega  (\widetilde{\bm{w}} \cdot \nabla \mus)\psi\,\mathrm{d}x
+\int_\Omega (\nu'(\vp^*) \bus\cdot \widetilde{\bm{w}})\psi\,\mathrm{d}x = \int_\Omega \widetilde{g}_2\psi\,\mathrm{d}x,
\end{align*}
 for almost all $t\in(0,T)$ and all $\psi\in H^1(\Omega)$. Here, we have
\begin{align*}
\widetilde{\bm{g}}_1&=-\widetilde{\bm{w}}^\sharp\int_0^1\nu'(s\vp^*+(1-s)\vp^\sharp)(\vp^*-\vp^\sharp)\,\mathrm{d}s
-\widetilde{\rho}^\sharp\nabla(\vp^*-\vp^\sharp)\\
&\quad -\nu'(\vp^*)\xi^*(\bm{w}^*-\bm{w}^\sharp)
-\nu'(\vp^*)(\xi^*-\xi^\sharp)\bm{w}^\sharp\\
&\quad -\xi^\sharp\bm{w}^\sharp\int_0^1\nu''(s\vp^*+(1-s)\vp^\sharp)(\vp^*-\vp^\sharp)\,\mathrm{d}s\\
&\quad -(\rho^*-\rho^\sharp)\nabla \xi^*-\rho^\sharp\nabla (\xi^*-\xi^\sharp),\\
\widetilde{g}_2&= (\bm{u}^*-\bm{u}^\sharp)\cdot\nabla \widetilde{\rho}^\sharp
 +\widetilde{\xi}^\sharp\int_0^1\Psi^{(3)}(s\vp^*+(1-s)\vp^\sharp)(\vp^*-\vp^\sharp)\,\mathrm{d}s\\
&\quad -\widetilde{\bm{w}}^\sharp\cdot\nabla(\mu^*-\mu^\sharp)
-\nu'(\vp^*)(\bm{u}^*-\bm{u}^\sharp)\cdot\widetilde{\bm{w}}^\sharp\\
&\quad - \bm{u}^\sharp\cdot \widetilde{\bm{w}}^\sharp \int_0^1\nu''(s\vp^*+(1-s)\vp^\sharp)(\vp^*-\vp^\sharp)\,\mathrm{d}s\\
&\quad + \bm{v}^*\cdot\nabla (\rho^*-\rho^\sharp) + (\bm{v}^*-\bm{v}^\sharp)\cdot\nabla \rho^\sharp\\
&\quad + \Psi^{(3)}(\vp^*)\xi^*(\zeta^*-\zeta^\sharp) + \Psi^{(3)}(\vp^*)(\xi^*-\xi^\sharp)\zeta^\sharp\\
&\quad + \xi^\sharp\zeta^\sharp\int_0^1\Psi^{(4)}(s\vp^*+(1-s)\vp^\sharp)(\vp^*-\vp^\sharp)\,\mathrm{d}s\\
&\quad -\bm{w}^*\cdot\nabla(\eta^*-\eta^\sharp)-(\bm{w}^*-\bm{w}^\sharp)\cdot\nabla \eta^\sharp
- \nu''(\vp^*)\xi^*\bm{u}^*\cdot(\bm{w}^*-\bm{w}^\sharp)\\
&\quad - \nu''(\vp^*)\xi^*(\bm{u}^*-\bm{u}^\sharp)\cdot \bm{w}^\sharp - \nu''(\vp^*)(\xi^*-\xi^\sharp)\bm{u}^\sharp \cdot \bm{w}^\sharp \\
&\quad - \xi^\sharp(\bm{u}^\sharp\cdot\bm{w}^\sharp) \int_0^1\nu^{(3)}(s\vp^*+(1-s)\vp^\sharp)(\vp^*-\vp^\sharp)\,\mathrm{d}s\\
&\quad -\nu'(\vp^*)\bm{v}^*\cdot(\bm{w}^*-\bm{w}^\sharp) -\nu'(\vp^*)(\bm{v}^*-\bm{v}^\sharp)\bm{w}^\sharp\\
&\quad -(\bm{v}^\sharp\cdot \bm{w}^\sharp)\int_0^1\nu''(s\vp^*+(1-s)\vp^\sharp)(\vp^*-\vp^\sharp)\,\mathrm{d}s +\alpha_2(\xi^*-\xi^\sharp),\\
\widetilde{g}_3&= \widetilde{\bm{w}}^\sharp\cdot\nabla(\vp^*-\vp^\sharp) + (\bm{w}^*-\bm{w}^\sharp)\cdot\nabla \xi^*+\bm{w}^\sharp\cdot\nabla(\xi^*-\xi^\sharp),\\
\widetilde{g}_4&=\alpha_1(\xi^*(T)-\xi^\sharp(T)).
\end{align*}

Applying Lemma \ref{adexe}, we can derive the following estimate
\begin{align}
&\|\widetilde{\rho}^*-\widetilde{\rho}^\sharp\|^2_{C([0,T];H^1(\Omega))\cap L^2(0,T;H^3(\Omega))\cap H^1(0,T;(H^1(\Omega))')} \non\\
&\quad \leq C \Big(\|\widetilde{\bm{g}}_1\|^2_{L^2(0,T;H^1(\Omega))}
+ \|\widetilde{g}_2\|^2_{L^2(0,T;L^{4/3}(\Omega))} +  \|\widetilde{g}_3\|^2_{L^2(0,T;H^1(\Omega))} + \|\widetilde{g}_4\|^2_{H^1(\Omega)}\Big)\non\\
&\quad\leq  C\|R^*-R^\sharp\|^2_{L^2(Q)}\|h\|_{L^2(Q)}^2.
\label{estiLinez}
\end{align}
Indeed, using the estimates \eqref{stabu1}, \eqref{K1}, \eqref{K2},  \eqref{estiLine}, \eqref{DeLipa}, \eqref{estiLineaj}, \eqref{stabuLip} and \eqref{estiLinedd}, we find
\begin{align*}
&\|\widetilde{\bm{g}}_1\|_{L^2(0,T;H^1(\Omega))}^2 \\
&\quad \leq C\|\widetilde{\bm{w}}^\sharp\|_{L^\infty(0,T;L^2(\Omega))}^2 \int_0^T\|\vp^*-\vp^\sharp\|_{W^{1,\infty}(\Omega)}^2\,\mathrm{d}t
+ C\|\vp^*-\vp^\sharp\|_{C([0,T];L^\infty(\Omega))}^2 \int_0^T \|\widetilde{\bm{w}}^\sharp\|_{H^1(\Omega)}^2\,\mathrm{d}t\\
&\qquad + C\|\widetilde{\rho}^\sharp\|_{C([0,T];H^1(\Omega))}^2 \int_0^T\|\nabla(\vp^*-\vp^\sharp)\|_{L^\infty(\Omega)}^2\,\mathrm{d}t
+ C\|\vp^*-\vp^\sharp\|_{C([0,T];H^2(\Omega))}^2  \int_0^T\|\widetilde{\rho}^\sharp\|_{L^\infty(\Omega)}^2\,\mathrm{d}t\\
&\qquad +C  \|\bm{w}^*-\bm{w}^\sharp\|_{L^\infty(0,T;L^2(\Omega))}^2\int_0^T \|\xi^*\|_{W^{1,\infty}(\Omega)}^2\,\mathrm{d}t
+C \|\xi^*\|_{C([0,T];L^\infty(\Omega))}^2\int_0^T \|\bm{w}^*-\bm{w}^\sharp\|_{H^1(\Omega)}^2\,\mathrm{d}t\\
&\qquad +C \|\bm{w}^\sharp\|_{L^\infty(0,T;L^2(\Omega))}^2 \int_0^T\|\xi^*-\xi^\sharp\|_{W^{1,\infty}(\Omega)}^2\,\mathrm{d}t
+C \|\xi^*-\xi^\sharp\|_{C([0,T];L^\infty(\Omega))}^2 \int_0^T\|\bm{w}^\sharp\|_{H^1(\Omega)}^2\,\mathrm{d}t\\
&\qquad +C \|\xi^\sharp\|_{C([0,T];L^\infty(\Omega))}^2 \|\bm{w}^\sharp\|_{L^\infty(0,T;L^2(\Omega))}^2 \int_0^T\|\vp^*-\vp^\sharp\|_{W^{1,\infty}(\Omega)}^2\,\mathrm{d}t \\
&\qquad  + C \|\bm{w}^\sharp\|_{L^\infty(0,T;L^2(\Omega))}^2 \|\vp^*-\vp^\sharp\|_{C([0,T];L^\infty(\Omega))}^2\int_0^T \|\nabla \xi^\sharp\|_{L^\infty(\Omega)}^2\,\mathrm{d}t\\
&\qquad +C  \|\xi^\sharp\|_{C([0,T];L^\infty(\Omega))}^2 \|\vp^*-\vp^\sharp\|_{C([0,T];L^\infty(\Omega))}^2 \int_0^T \|\bm{w}^\sharp\|_{H^1(\Omega)}^2\,\mathrm{d}t\\
&\qquad +C \|\rho^*-\rho^\sharp\|_{C([0,T];H^1(\Omega))}^2 \int_0^T \| \nabla \xi^*\|_{L^\infty(\Omega)}^2\,\mathrm{d}t
 +C \| \nabla \xi^*\|_{C([0,T];H^1(\Omega))}^2\int_0^T \|\rho^*-\rho^\sharp\|_{L^\infty(\Omega)}^2\,\mathrm{d}t\\
&\qquad +C \|\rho^\sharp\|_{C([0,T];H^1(\Omega))}^2\int_0^T \|\nabla (\xi^*-\xi^\sharp)\|_{L^\infty(\Omega)}^2\,\mathrm{d}t \\
&\qquad
 +C\| \nabla (\xi^*-\xi^\sharp)\|_{C([0,T];H^1(\Omega))}^2 \int_0^T\|\rho^\sharp\|_{L^\infty(\Omega)}^2\,\mathrm{d}t\\
& \quad  \leq C\|R^*-R^\sharp\|^2_{L^2(Q)}\|h\|_{L^2(Q)}^2,
\end{align*}
\begin{align*}
& \|\widetilde{g}_2\|_{L^2(0,T;L^{4/3}(\Omega))}^2 \\
&\quad \leq C \|\nabla \widetilde{\rho}^\sharp\|_{C([0,T];L^2(\Omega))}^2\int_0^T \|\bm{u}^*-\bm{u}^\sharp\|_{L^4(\Omega)}^2\,\mathrm{d}t  + C \|\widetilde{\xi}^\sharp\|_{C([0,T];L^2(\Omega))}^2 \int_0^T\|\vp^*-\vp^\sharp\|_{L^4(\Omega)}^2\,\mathrm{d}t\\
&\qquad +C\|\widetilde{\bm{w}}^\sharp\|_{L^\infty(0,T;L^2(\Omega))}^2 \int_0^T\|\nabla(\mu^*-\mu^\sharp)\|_{L^4(\Omega)}^2\,\mathrm{d}t  +C\|\widetilde{\bm{w}}^\sharp\|_{L^\infty(0,T;L^2(\Omega))}^2 \int_0^T \|\bm{u}^*-\bm{u}^\sharp\|_{L^4(\Omega)}^2\,\mathrm{d}t\\
&\qquad +C\|\widetilde{\bm{w}}^\sharp\|_{L^\infty(0,T;L^2(\Omega))}^2 \|\vp^*-\vp^\sharp\|_{C([0,T];L^\infty(\Omega))}^2 \int_0^T\|\bm{u}^\sharp\|_{L^4(\Omega)}^2\,\mathrm{d}t\\
&\qquad +C \|\nabla (\rho^*-\rho^\sharp)\|_{C([0,T];L^2(\Omega))}^2 \int_0^T\| \bm{v}^*\|_{L^4(\Omega)}^2\,\mathrm{d}t
+C  \|\nabla \rho^\sharp\|_{C([0,T];L^2(\Omega))}^2   \int_0^T\|\bm{v}^*-\bm{v}^\sharp\|_{L^4(\Omega)}^2\,\mathrm{d}t\\
&\qquad +C \|\xi^*\|_{C([0,T];L^2(\Omega))}^2 \int_0^T\|\zeta^*-\zeta^\sharp\|_{L^4(\Omega)}^2\,\mathrm{d}t
+C \|\xi^*-\xi^\sharp\|_{C([0,T];L^2(\Omega))}^2\int_0^T \|\zeta^\sharp\|_{L^4(\Omega)}^2\,\mathrm{d}t\\
&\qquad +C\|\xi^\sharp\|_{C([0,T];L^2(\Omega))}^2 \|\vp^*-\vp^\sharp\|_{C([0,T];L^\infty(\Omega))}^2\int_0^T \|\zeta^\sharp\|_{L^4(\Omega)}^2\,\mathrm{d}t\\
& \qquad +C \|\bm{w}^*\|_{L^\infty(0,T;L^2(\Omega))}^2\int_0^T \|\nabla(\eta^*-\eta^\sharp)\|_{L^4(\Omega)}^2\,\mathrm{d}t \\
&\qquad
 +C \|\bm{w}^*-\bm{w}^\sharp\|_{L^\infty(0,T;L^2(\Omega))}^2\int_0^T \|\nabla \eta^\sharp\|_{L^4(\Omega)}^2\,\mathrm{d}t\\
 &\qquad +C \|\bm{w}^*-\bm{w}^\sharp\|_{L^\infty(0,T;L^2(\Omega))}^2 \|\xi^*\|_{C([0,T];L^\infty(\Omega))}^2
 \int_0^T \| \bm{u}^*\|_{L^4(\Omega)}^2\,\mathrm{d}t\\
 &\qquad +C \|\bm{w}^\sharp\|_{L^\infty(0,T;L^2(\Omega))}^2\| \xi^*\|_{C([0,T];L^\infty(\Omega))}^2
 \int_0^T \|\bm{u}^*-\bm{u}^\sharp\|_{L^4(\Omega)}^2\,\mathrm{d}t\\
 &\qquad  +C \|\bm{w}^\sharp\|_{L^\infty(0,T;L^2(\Omega))}^2 \|\xi^*-\xi^\sharp\|_{C([0,T];L^\infty(\Omega))}^2
 \int_0^T \|\bm{u}^\sharp \|_{L^4(\Omega)}^2\,\mathrm{d}t\\
 &\qquad +C \|\bm{w}^\sharp\|_{L^\infty(0,T;L^2(\Omega))}^2 \| \xi^\sharp\|_{C([0,T];L^\infty(\Omega))}^2
 \|\vp^*-\vp^\sharp\|_{C([0,T];L^\infty(\Omega))}^2  \int_0^T\|\bm{u}^\sharp\|_{L^4(\Omega)}^2\,\mathrm{d}t\\
 &\qquad +C \|\bm{w}^*-\bm{w}^\sharp\|_{L^\infty(0,T;L^2(\Omega))}^2 \int_0^T \| \bm{v}^*\|_{L^4(\Omega)}^2\,\mathrm{d}t
  +C \|\bm{w}^\sharp\|_{L^\infty(0,T;L^2(\Omega))}^2\int_0^T \|\bm{v}^*-\bm{v}^\sharp \|_{L^4(\Omega)}^2\,\mathrm{d}t\\
  &\qquad  +C \|\bm{w}^\sharp\|_{L^\infty(0,T;L^2(\Omega))}^2 \|\vp^*-\vp^\sharp\|_{C([0,T];L^\infty(\Omega))}^2
  \int_0^T  \|\bm{v}^\sharp \|_{L^4(\Omega)}^2 \,\mathrm{d}t  +\alpha_2\|\xi^*-\xi^\sharp\|_{L^2(Q)}^2\\
  &\quad \leq C\|R^*-R^\sharp\|^2_{L^2(Q)}\|h\|_{L^2(Q)}^2,
\end{align*}
\begin{align*}
\|\widetilde{g}_3\|_{L^2(0,T;H^1(\Omega))}^2
&\leq C\|\widetilde{\bm{w}}^\sharp\|_{L^\infty(0,T;L^2(\Omega))}^2
\int_0^T \|\nabla(\vp^*-\vp^\sharp)\|_{W^{1,\infty}(\Omega)}^2 \,\mathrm{d}t\\
&\quad +C \|\nabla(\vp^*-\vp^\sharp)\|_{C([0,T];L^2(\Omega))} \int_0^T \|\nabla \widetilde{\bm{w}}^\sharp\|_{L^2(\Omega)}^2  \|\nabla(\vp^*-\vp^\sharp)\|_{H^2(\Omega)} \,\mathrm{d}t\\
&\quad +C \|\bm{w}^*-\bm{w}^\sharp\|_{L^\infty(0,T;L^2(\Omega))}^2 \int_0^T\| \nabla \xi^*\|_{W^{1,\infty}(\Omega)}^2 \,\mathrm{d}t\\
&\quad +C \| \nabla \xi^*\|_{C([0,T];L^2(\Omega))}\int_0^T\|\bm{w}^*-\bm{w}^\sharp\|_{H^1(\Omega)}^2  \| \nabla \xi^*\|_{H^2(\Omega)}  \,\mathrm{d}t\\
&\quad + C\|\bm{w}^\sharp\|_{L^\infty(0,T;L^2(\Omega))}^2
\int_0^T \|\nabla(\xi^*-\xi^\sharp)\|_{W^{1,\infty}(\Omega)}^2 \,\mathrm{d}t\\
&\quad +C \|\nabla(\xi^*-\xi^\sharp)\|_{C([0,T];L^2(\Omega))} \int_0^T \|\nabla \bm{w}^\sharp\|_{L^2(\Omega)}^2  \|\nabla(\xi^*-\xi^\sharp)\|_{H^2(\Omega)} \,\mathrm{d}t\\
&  \leq C\|R^*-R^\sharp\|^2_{L^2(Q)}\|h\|_{L^2(Q)}^2,
\end{align*}
and
\begin{align*}
\|\widetilde{g}_4\|_{H^1(\Omega)}^2\leq C\|\xi^*-\xi^\sharp\|_{C([0,T];H^1(\Omega))}^2
\leq  C\|R^*-R^\sharp\|^2_{L^2(Q)}\|h\|_{L^2(Q)}^2.
\end{align*}
Collecting the above estimates, we obtain \eqref{estiLinez}, which easily yields the conclusion \eqref{DeLipT}.
\end{proof}

\subsection{A second-order sufficient condition}

Since the control-to-state operator $\mathcal{S}$ and the control-to-costate operator $\mathcal{T}$ are continuously Fr\'{e}chet differentiable (recall Proposition \ref{1stFD}
and Proposition  \ref{Diffcotoco}), the cost functional $\mathcal{J}$ as well as the reduced cost functional  $\widehat{\mathcal{J}}$ are twice continuously differentiable
in $\mathcal{U}$ thanks to the chain rule. With the help of the adjoint state, for a given control $R^*\in \widetilde{\mathcal{U}}$, the first and second Fr\'{e}chet derivatives of
$\widehat{\mathcal{J}}$ at $R^*$ can be calculated. Indeed, we have
$$
\widehat{\mathcal{J}}'(R^*)h=
\int_Q\left(\rho^* + \beta R^*\right)h\,\mathrm{d}x\mathrm{d}t,\quad \forall\,
h\in \mathcal{U},
$$
where $\rho^*$ is the associated adjoint state of $R^*$, and
$$
\widehat{\mathcal{J}}''(R^*)[h_1,h_2]=\int_Q \big(\beta h_1+\widetilde{\rho}_{R^*}^{h_1}\big)h_2\,\mathrm{d}x\mathrm{d}t,\quad \forall\,
h_1, h_2\in \mathcal{U},
$$
where $\widetilde{\rho}_{R^*}^{h_1}$ is determined as in \eqref{Dtt} with $h=h_1$.

To derive the second-order sufficient condition, we apply the idea in \cite{CRT08} and introduce the concept of cone of critical directions.
\begin{definition}
Let $R^*\in \mathcal{U}_{\mathrm{ad}}$ and set
$$
A_0(R^*)=\{(x,t)\in Q\ |\ |\rho^*(x,t) + \beta R^*(x,t)|>0\},
$$
where $\rho^*$ is the associated adjoint state of $R^*$.
The \textbf{cone of critical directions}, denoted by $\mathcal{C}(R^*)$, is defined as
$$
\mathcal{C}(R^*):=\{h\in  \mathcal{U}\cap L^\infty(Q)\ |\ h\ \text{ satisfies }\ \eqref{CC}\},
$$
where \eqref{CC} is given by
\begin{equation}\label{CC}
h(x,t)
\begin{cases}
\geq 0,\quad \text{if}\ \ (x,t)\notin A_0(R^*),\ R^*(x,t)=R_{\mathrm{min}}(x,t),\\
\leq 0,\quad \text{if}\ \ (x,t)\notin A_0(R^*),\ R^*(x,t)=R_{\mathrm{max}}(x,t),\\
=0,\quad \text{if}\ \ (x,t)\in A_0(R^*),
\end{cases}
\end{equation}
for almost all $(x,t)\in Q$.
\end{definition}

The main result of this section reads as follows.

\begin{theorem}\label{2ndSC}
Assume that $\Omega\subset \mathbb{R}^2$ is a bounded domain with smooth boundary $\partial\Omega$, $T>0$ and assumptions $\mathbf{(S1)}$--$\mathbf{(S5)}$ are satisfied
with $\beta>0$. Let $R^*\in\mathcal{U}_{\mathrm{ad}}$ be any control satisfying the variational inequality \eqref{vug2}. Moreover, we assume that
\begin{align}
\widehat{\mathcal{J}}''(R^*)[h,h]=\int_Q \big(\beta h+\widetilde{\rho}_{R^*}^{h}\big)h\,\mathrm{d}x\mathrm{d}t>0,\quad \forall\, h\in \mathcal{C}(R^*)\setminus \{0\}.\label{2ndScon}
\end{align}
Then there exists constants $\lambda,\,\theta>0$ such that the following inequality holds
\begin{align}
\widehat{\mathcal{J}}(R)\geq \widehat{\mathcal{J}}(R^*)+\theta\|R-R^*\|_{L^2(Q)}^2,\quad \forall\, R\in \mathcal{U}_{\mathrm{ad}} \ \ \text{with}\ \ \|R-R^*\|_{L^2(Q)}<\lambda.
\label{quadgro}
\end{align}
As a consequence, $R^*$ is a strict local minimizer of $\widehat{\mathcal{J}}$ on the set $\mathcal{U}_{\mathrm{ad}}$.
\end{theorem}

\begin{proof}
The proof follows from a contradiction argument in the spirit of \cite[Theorem 4.1]{CRT08}. Suppose that $R^*$ does not fulfill \eqref{quadgro}.
Then there exists a sequence of controls $\{R_k\}\subset \mathcal{U}_{\mathrm{ad}}$ satisfying $R_k\neq R^*$, $R_k\to R^*$ strongly in $L^2(Q)$ such that
\begin{align}
\widehat{\mathcal{J}}(R_k)< \widehat{\mathcal{J}}(R^*)+\frac{1}{k}\|R_k-R^*\|_{L^2(Q)}^2,\qquad  \forall\, k\in \mathbb{Z}^+.
\label{contri1}
\end{align}
Define
$$
\widetilde{h}_k=\frac{R_k-R^*}{\|R_k-R^*\|_{L^2(Q)}}.
$$
Since $\|\widetilde{h}_k\|_{L^2(Q)}=1$, we can extract a subsequence (not relabelled for simplicity) such that
$$
\widetilde{h}_k\to \widetilde{h}\quad \text{weakly in}\ L^2(Q)\ \  \text{for some}\ \widetilde{h}\in L^2(Q).
$$

Let $\rho_k$ be the associated adjoint state of $R_k$. Using Propositions \ref{StoCo},  \ref{adexe2}, and \ref{Lipcotoco}, we can deduce that
\begin{align*}
& \left|\int_Q\left(\rho_k + \beta R_k\right)\widetilde{h}_k\,\mathrm{d}x\mathrm{d}t
- \int_Q\left(\rho^* + \beta R^*\right)\widetilde{h}\,\mathrm{d}x\mathrm{d}t\right|\\
&\quad \leq \left|\int_Q\left((\rho_k-\rho^*) + \beta (R_k-R^*)\right)\widetilde{h}_k\,\mathrm{d}x\mathrm{d}t\right|
+\left|\int_Q\left(\rho^* + \beta R^*\right)(\widetilde{h}_k-\widetilde{h})\,\mathrm{d}x\mathrm{d}t\right| \to 0,
\end{align*}
as $k\to+\infty$, which implies
\begin{align}
\lim_{k\to+\infty} \widehat{\mathcal{J}}'(R_k)\widetilde{h}_k =\widehat{\mathcal{J}}'(R^*)\widetilde{h}.
\label{1stCov}
\end{align}
Using the Newton-Leibniz formula
$$
\widehat{\mathcal{J}}(R_k) = \widehat{\mathcal{J}}(R^*)
+ \int_0^1 \widehat{\mathcal{J}}'(s R_k+(1-s) R^*) (R_k-R^*) \,\mathrm{d}s,
$$
we infer from \eqref{contri1} that
\begin{align*}
\int_0^1 \widehat{\mathcal{J}}'(s R_k+(1-s) R^*) \widetilde{h}_k \,\mathrm{d}s < \frac{1}{k}\|R_k-R^*\|_{L^2(Q)} \to 0, \quad \text{as}\ k\to+\infty.
\end{align*}
This fact, combined with \eqref{1stCov} and Proposition \ref{Lipcotoco}, easily yields
$\widehat{\mathcal{J}}'(R^*)\widetilde{h}\leq 0$. On the other hand, since $R_k\in \mathcal{U}_{\mathrm{ad}}$, it follows from \eqref{vug2} that
$\widehat{\mathcal{J}}'(R^*)\widetilde{h}_k\geq 0$. This further implies
$\widehat{\mathcal{J}}'(R^*)\widetilde{h}\geq 0$.
As a consequence, we find
\begin{align}
\widehat{\mathcal{J}}'(R^*)\widetilde{h}=0. \label{1stCovb}
\end{align}
Then, following the same argument in Step 2 of the proof of \cite[Theorem 4.1]{CRT08}, we can conclude
$$\widetilde{h}\in \mathcal{C}(R^*).$$

Next, from Proposition \ref{Diffcotoco} and the compact embedding $\mathcal{Z}\hookrightarrow\hookrightarrow L^2(Q)$, we can deduce that, up to a subsequence
(not relabelled hereafter for simplicity), it holds
$$
\|\widetilde{\rho}_{R^*}^{\widetilde{h}_k} -\widetilde{\rho}_{R^*}^{\widetilde{h}}\|_{L^2(\Omega)}\to 0,\quad \text{as}\ k\to+\infty.
$$
As a result, we have (again up to a subsequence)
\begin{align*}
\left|\int_Q \widetilde{\rho}_{R^*}^{\widetilde{h}_k} \widetilde{h}_k\,\mathrm{d}x\mathrm{d}t
-\int_Q \widetilde{\rho}_{R^*}^{\widetilde{h}} \widetilde{h}\,\mathrm{d}x\mathrm{d}t\right|
&\leq \left|\int_Q \big(\widetilde{\rho}_{R^*}^{\widetilde{h}_k} -\widetilde{\rho}_{R^*}^{\widetilde{h}} \big)\widetilde{h}_k\,\mathrm{d}x\mathrm{d}t\right|
+\left|\int_Q \widetilde{\rho}_{R^*}^{\widetilde{h}} (\widetilde{h}_k- \widetilde{h}\big)\,\mathrm{d}x\mathrm{d}t\right|\to  0,
\end{align*}
as $k\to+\infty$. Hence, it holds
\begin{align}
\lim_{k\to+\infty} \widehat{\mathcal{J}}''(R^*)[\widetilde{h}_k,\widetilde{h}_k] =\widehat{\mathcal{J}}''(R^*)[\widetilde{h},\widetilde{h}].
\label{2ndCov}
\end{align}
Using a second-order Taylor's expansion, we find
\begin{align*}
\widehat{\mathcal{J}}(R_k)&  = \widehat{\mathcal{J}}(R^*) + \widehat{\mathcal{J}}'(R^*)(R_k-R^*) \\
&\quad + \frac12\|R_k-R^*\|_{L^2(Q)}^2\int_0^1\int_0^1 \widehat{\mathcal{J}}''(sz R_k+(1-sz) R^*) [\widetilde{h}_k,\widetilde{h}_k]\,\mathrm{d}s\mathrm{d}z.
\end{align*}
From \eqref{vug2} and \eqref{contri1} we infer that
\begin{align}
&   \int_0^1\int_0^1 \widehat{\mathcal{J}}''(sz R_k+(1-sz) R^*) [\widetilde{h}_k,\widetilde{h}_k]\,\mathrm{d}s\mathrm{d}z  < \frac{2}{k}.\label{2ndCovb}
\end{align}
Besides, the Lipschitz continuity property (see Proposition \ref{Diffcotoco}) implies
$$
\left|\int_0^1\int_0^1 \widehat{\mathcal{J}}''(sz R_k+(1-sz) R^*) [\widetilde{h}_k,\widetilde{h}_k]\,\mathrm{d}s\mathrm{d}z
-\widehat{\mathcal{J}}''(R^*)[\widetilde{h}_k,\widetilde{h}_k]\right|\to  0,\quad \text{as}\ k\to+\infty.
$$
From the above fact, \eqref{2ndCov} and \eqref{2ndCovb}, we easily deduce that
\begin{align*}
\widehat{\mathcal{J}}''(R^*)[\widetilde{h},\widetilde{h}]
& = \limsup_{k\to+\infty} \widehat{\mathcal{J}}''(R^*)[\widetilde{h}_k,\widetilde{h}_k]\\
&= \limsup_{k\to+\infty}  \int_0^1\int_0^1 \widehat{\mathcal{J}}''(sz R_k+(1-sz) R^*) [\widetilde{h}_k,\widetilde{h}_k]\,\mathrm{d}s\mathrm{d}z  \\
&\quad + \lim_{k\to+\infty}\left(\widehat{\mathcal{J}}''(R^*)[\widetilde{h}_k,\widetilde{h}_k]
- \int_0^1\int_0^1 \widehat{\mathcal{J}}''(sz R_k+(1-sz) R^*) [\widetilde{h}_k,\widetilde{h}_k]\,\mathrm{d}s\mathrm{d}z \right)\\
&\leq 0.
\end{align*}

Recalling condition \eqref{2ndScon} and the fact that $\widetilde{h}\in \mathcal{C}(R^*)$, we find $\widetilde{h}=0$ so that $\widetilde{h}_k\to 0$ weakly in $L^2(Q)$.
Hence, $\widetilde{\rho}_{R^*}^{\widetilde{h}}=0$ and, up to a subsequence, we have
$$
\|\widetilde{\rho}_{R^*}^{\widetilde{h}_k}\|_{L^2(\Omega)}\to 0,\quad \text{as}\ k\to+\infty.
$$
Then, we find that
\begin{align*}
0&<\beta = \beta \lim_{k\to+\infty} \|\widetilde{h}_k\|_{L^2(\Omega)}^2
 \\
&= \lim_{k\to+\infty}\widehat{\mathcal{J}}''(R^*)[\widetilde{h}_k,\widetilde{h}_k]
-\lim_{k\to+\infty}  \int_Q \widetilde{\rho}_{R^*}^{\widetilde{h}_k} \widetilde{h}_k\,\mathrm{d}x\mathrm{d}t \non\\
&= \widehat{\mathcal{J}}''(R^*)[\widetilde{h},\widetilde{h}]\leq 0,
\end{align*}
which leads to a contradiction. The proof of Theorem \ref{2ndSC} is complete.
\end{proof}

\section{Appendix}
\label{appe}\setcounter{equation}{0}
In this appendix, we prove Theorem \ref{str-well} on the existence and uniqueness of a strong solution to problem \eqref{bdini}--\eqref{CHHS}.

We apply a strategy similar to that for the existence of weak solutions described in Section \ref{exeweak}
(see \cite{JWZ,GGW,Gio2020,SW21}). Here, the crucial issue is to derive uniform higher-order estimates for the approximate solutions. Below we sketch the main steps and
point out the necessary modifications due to the presence of mass source terms.
For simplicity, we denote the norm of $L^2(\Omega)$ (or $(L^2(\Omega))^2$) by $\|\cdot\|$, and the norms of $L^p(\Omega)$, $W^{m,p}(\Omega)$, $H^m(\Omega)$
by $\|\cdot\|_{L^p}$, $\|\cdot\|_{W^{m,p}}$ and $\|\cdot\|_{H^m}$, respectively.

\smallskip

\textbf{Step 1. The approximating problem.} Let  $\varepsilon\in (0,\kappa)$ with $\kappa$ being the constant given in $\mathbf{(A1)'}$. As in \cite{GGW}, we introduce a
family of regular potentials $\lbrace \Psi_\varepsilon \rbrace$ that approximates the original singular potential $\Psi$ by setting
\be
\Psi_\varepsilon(s)=F_\varepsilon(s)-\frac{\Theta_0}{2}s^2,\quad \forall\, s\in \mathbb{R},\label{defF}
\ee
where
\be
F_\varepsilon(s)=
\begin{cases}
 \displaystyle{\sum_{j=0}^4 \frac{1}{j!}}
 F^{(j)}(1-\varepsilon) \left[s-(1-\varepsilon)\right]^j,
 \qquad\qquad \!\!  \forall\,s\geq 1-\varepsilon,\\
 F(s), \qquad \qquad \qquad \qquad \qquad \qquad \qquad \quad
 \forall\, s\in[-1+\varepsilon, 1-\varepsilon],\\
 \displaystyle{\sum_{j=0}^4
 \frac{1}{j!}} F^{(j)}(-1+\varepsilon)\left[ s-(-1+\varepsilon)\right]^j,
 \qquad\ \forall\, s\leq -1+\varepsilon.
 \end{cases}
 \label{defF1s}
\ee
Thanks to the definition of $\Psi_\varepsilon$,
one easily finds (see \cite{Gio2020}) some $\overline{\kappa}\in (0,\kappa]$ such that, for any $\varepsilon\in (0,\overline{\kappa})$, the approximating
function $\Psi_\varepsilon$  given by \eqref{defF} satisfies
 $\Psi_\varepsilon\in C^{4}(\mathbb{R})$ and
\begin{equation}
\gamma_1s^4-\gamma_2 \leq \Psi_\varepsilon(s), \quad -\alpha\leq \Psi''_\varepsilon(s)\leq L,\quad \forall \, s\in \mathbb{R},
\label{psiep}
\end{equation}
where $\gamma_1$, $\gamma_2$ are positive constants
independent of $\varepsilon$, the constant $\alpha$ is given in $\mathbf{(A1)}$ and $L$ is a positive constant that may depend on $\varepsilon$. Moreover, we have
\begin{align}
& \Psi_{\varepsilon}(s)\leq \Psi(s),\quad \forall\, s\in [-1,1]\ \ \ \ \text{and}\ \ \ |\Psi'_{\varepsilon}(s)| \leq |\Psi'(s)|, \quad \forall\, s\in (-1,1). \nonumber
\end{align}
For every $\varepsilon\in (0,\overline{\kappa})$, we study the following approximating problem:
\be
\begin{cases}
\label{A1CHHS}
\nu(\vp_\varepsilon) \u_\varepsilon= -\nabla P_\varepsilon+ \mu_\varepsilon \nabla \vp_\varepsilon,\\
\text{div} \u_\varepsilon=S, \\
\partial_t \vp_\varepsilon + \mathrm{div}\,(\vp_\varepsilon \u_\varepsilon)   =
 \Delta \mu_\varepsilon + S+ R,\\
\mu_\varepsilon= - \Delta \vp_\varepsilon +
 \Psi'_\varepsilon(\vp_\varepsilon),
\end{cases}
 \quad \text{ in } Q,
\ee
subject to the initial and boundary conditions
\be
\label{bdini2}
\begin{cases}
\u_\varepsilon \cdot \n= \partial_\n \mu_\varepsilon
=\partial_\n \vp_\varepsilon=0, \quad\ \text{on\ } \Sigma,\\
\vp_\varepsilon|_{t=0}=\vp_0, \quad\ \ \qquad \qquad \qquad \text{in } \Omega.
\end{cases}
\ee

When the strong solutions are concerned, we need to pay some attention to the ``initial value" of the approximated chemical potential
$\mu_{\varepsilon,0} \triangleq - \Delta \vp_0 +  \Psi'_\varepsilon(\vp_0)$ (see \cite{GMT,Gio2020}).
To this end, we consider the following Neumann problem with a singular nonlinearity:
\be
\begin{cases}
-\Delta u+F^{\prime}(u)=f,\quad \text { in } \Omega, \\
\partial_{\mathbf{n}} u=0, \qquad \qquad\quad\
\text { on } \partial \Omega,
\end{cases}
\label{ph1}
\ee
where the function $F$ satisfies the assumptions $\mathbf{(A1)}$ and $\mathbf{(A1)'}$. Then the following lemma holds
\begin{lemma} \label{esing}
Let $\Omega\subset \mathbb{R}^2$ be a bounded domain with smooth boundary $\partial\Omega$. For any $f\in L^2(\Omega)$, problem \eqref{ph1} admits a unique strong solution $u\in H^2(\Omega)$ and $F^{\prime}(u)\in L^2(\Omega)$.
If $f \in H^1(\Omega)$, then $\|\Delta u\| \leq C\|\nabla u\|^{\frac{1}{2}}\|\nabla f\|^{\frac{1}{2}}$ for some positive constant $C$ independent of $u$.
For any $q\in [2,+\infty)$, there exists a positive constant $C=C(q,\Omega)$ such that
\begin{align}
\|u\|_{W^{2, q}}+\|F^{\prime}(u)\|_{L^{q}} &\leq C\left(1+\|f\|_{H^1}\right),\label{ph2}\\
\|F^{\prime \prime}(u)\|_{L^{q}}
&\leq C\Big(1+{e}^{C\|f\|_{H^1}^{2}}\Big).
\label{ph3}
\end{align}
Moreover, $F'(u)\in W^{1,q}(\Omega)$ and there exists a constant $\delta\in (0,1)$ such that
\be
\|u\|_{L^{\infty}} \leq 1-\delta.\label{sep2}
\ee
\end{lemma}
\begin{remark}
The results presented in Lemma \ref{esing} can be found in \cite{A2009,GGM2017,GGW,GMT}, and in particular, the strict separation property \eqref{sep2} was proved in
\cite[Lemma 3.2]{HW21}.
\end{remark}

Now, for the initial datum $\varphi_0$ given in Theorem \ref{str-well}, we infer from Lemma \ref{esing} that there exists some $\widetilde{\delta}\in (0,1)$ such that
$\|\varphi_0\|_{L^{\infty}} \leq 1-\widetilde{\delta}$, i.e., the initial phase function is indeed strictly separated from the pure states $\pm 1$. The constant $\widetilde{\delta}$
depends on $\|\widetilde{\mu}_0\|_{H^1}$ and $\Omega$. Then, by the elliptic estimate, we also find that $\varphi_0\in H^3(\Omega)$.
As a consequence, different from \cite[Section 5]{Gio2020}, here we do not need to introduce approximations for the initial data (in dimension two). Indeed,
for every $\varepsilon\in (0, \min\{\overline{\kappa}, (1/2)\widetilde{\delta}\})$, we have (recall \eqref{defF1s})
$$
\widetilde{\mu}_{\varepsilon,0} \triangleq - \Delta \vp_0 +  F'_\varepsilon(\vp_0) =  - \Delta \vp_0 +  F'(\vp_0)=\widetilde{\mu}_{0},
$$
which yields the $\varepsilon$-independent estimate  $\|\widetilde{\mu}_{\varepsilon,0}\|_{H^1}=\|\widetilde{\mu}_0\|_{H^1}$.

Existence and uniqueness of global strong solutions to the approximating problem \eqref{A1CHHS}--\eqref{bdini2} can be proven by using the standard arguments similar to those in \cite{JWZ,WW12,WZ13}
via a suitable Faedo--Galerkin scheme. We only state the result and skip its proof.

\begin{proposition}
Let the assumptions in Theorem \ref{str-well} be satisfied. For every $\varepsilon\in (0, \min\{\overline{\kappa}, (1/2)\widetilde{\delta}\})$,  problem \eqref{A1CHHS} admits a
unique global strong solution $(\u_\varepsilon, P_\varepsilon, \varphi_\varepsilon, \mu_\varepsilon)$ on $[0,T]$ such that
\begin{align*}
& \u_\varepsilon\in C([0,T]; \mathbf{H}^1(\Omega))\cap L^2(0,T; \mathbf{H}^3(\Omega)),\\
& P_\varepsilon\in C([0,T];V_0)\cap L^2(0,T;H^4(\Omega)),\\
& \varphi_\varepsilon \in C([0,T]; H^3(\Omega))\cap L^2(0,T; H^5(\Omega))\cap H^1(0,T;H^1(\Omega)),\\
& \mu_\varepsilon \in C([0,T];H^1(\Omega))\cap L^2(0,T;H^3(\Omega)).
\end{align*}
The strong solution satisfies the system \eqref{A1CHHS} almost everywhere in $Q$. Moreover, $\u_\varepsilon \cdot \n= \partial_\n \mu_\varepsilon=\partial_\n \vp_\varepsilon=0$
almost everywhere on $\Sigma$ and $\varphi(\cdot,0)=\varphi_0$ in $\Omega$.
\end{proposition}

\textbf{Step 2. A priori estimates.}
We derive uniform estimates
for $(\u_\varepsilon, P_\varepsilon, \vp_\varepsilon, \mu_\varepsilon)$ that are independent of the approximating parameter
$\varepsilon\in (0, \min\{\overline{\kappa}, (1/2)\widetilde{\delta}\})$.

\medskip

\textbf{First estimate}.
Integrating the third equation of \eqref{A1CHHS} over $\Omega$, we have
$$
\int_\Omega \vp_\varepsilon(t,x)\,\d x =\int_\Omega \vp_0(x)\,\d x+\int_\Omega R(t,x)\,\d x, \quad \forall\, t\in [0,T].
$$
Then, from assumption $\mathbf{(A3)}$ and  \eqref{AR-1}
we deduce that
\begin{equation}
|\overline{\vp_\varepsilon(t)}|\leq 1-\delta_0,\quad \forall\, t\in [0,T].
\label{mean}
\end{equation}

\textbf{Second estimate}. The pressure $P_\varepsilon$ can be written as follows (see e.g., \cite{Gio2020,JWZ})
\begin{align}
P_\varepsilon&= \mathcal{N}\mathrm{div}\,(\nu(\vp_\varepsilon)\u_\varepsilon) -\mathcal{N}\mathrm{div}\,(\mu_\varepsilon\nabla\vp_\varepsilon)\notag\\
&=\mathcal{N}\mathrm{div}\,(\nu(\vp_\varepsilon)\u_\varepsilon) -\mathcal{N}\mathrm{div}\,((\mu_\varepsilon-\overline{\mu_\varepsilon}) \nabla\vp_\varepsilon)
+ \overline{\mu_\varepsilon}(\vp_\varepsilon-\overline{\vp_\varepsilon}).
\label{Pe}
\end{align}
Using $\mathbf{(A2)}$, H\"{o}lder's inequality and \eqref{mean}, we find
\begin{align*}
\|P_\varepsilon\|
&\leq C\big(\|\mathcal{N}\mathrm{div}\,(\nu(\vp_\varepsilon)\u_\varepsilon)\|+
\|\mathcal{N}\mathrm{div}\,((\mu_\varepsilon-\overline{\mu_\varepsilon}) \nabla\vp_\varepsilon)\|+ |\overline{\mu_\varepsilon}| \|\vp_\varepsilon-\overline{\vp_\varepsilon}\|\big)\\
&\leq C\|\nu(\vp_\varepsilon)\u_\varepsilon\|_{L^{3/2}} + C\|(\mu_\varepsilon-\overline{\mu_\varepsilon}) \nabla\vp_\varepsilon\|_{L^{6/5}}
+ |\overline{\mu_\varepsilon}|(1+\|\vp_\varepsilon\|)\\
&\leq C\|\u_\varepsilon\| + \|\mu_\varepsilon-\overline{\mu_\varepsilon}\|_{L^3} \|\nabla\vp_\varepsilon\| + |\overline{\mu_\varepsilon}|(1+\|\vp_\varepsilon\|).
\end{align*}

\textbf{Third estimate}.
Thanks to $\mathbf{(A1)}$ and  \eqref{mean}, we find the following standard estimate (see e.g., \cite{GGW})
$$
\|F_\varepsilon(\vp_\varepsilon)\|_{L^1} \leq C\int_\Omega
(\varphi_\varepsilon-\overline{\varphi_\varepsilon}) (F'_\varepsilon(\varphi_\varepsilon)-\overline{F'_\varepsilon(\varphi_\varepsilon)})\, \d x +C,
$$
based on which we can further obtain (see \cite{MA2020})
\begin{align*}
\|F_\varepsilon(\vp_\varepsilon)\|_{L^1}
&\leq C(\| \mu_\varepsilon - \overline{\mu_\varepsilon}\|
\|\vp_\varepsilon-\overline{\vp_\varepsilon}\| +\|\vp_\varepsilon-\overline{\vp_\varepsilon}\|\|\vp_\varepsilon\|)+C\\
&\leq C \| \mu_\varepsilon - \overline{\mu_\varepsilon}\| \|\vp_\varepsilon\| + C(1+\|\vp_\varepsilon\|^2).
\end{align*}
As a consequence, we get
\begin{equation}
|\overline{\mu_\varepsilon}|
\leq \|F_\varepsilon(\vp_\varepsilon)\|_{L^1} +\Theta_0|\Omega|^{-1}\left|\int_\Omega \vp_\varepsilon \,\d x\right|
\leq  C \| \mu_\varepsilon - \overline{\mu_\varepsilon}\| \|\vp_\varepsilon\| + C(1+\|\vp_\varepsilon\|^2).
\label{meanmu}
\end{equation}

\textbf{Fourth estimate}.
Multiplying the first equation of \eqref{A1CHHS} by $\u_\varepsilon$ and the third equation by $\mu_\varepsilon$, adding the resultants together and integrating over $\Omega$,
we obtain
\begin{align}
&\frac{\d}{\d t}\mathcal{E}_\varepsilon(t) + \int_\Omega \nu(\vp_\varepsilon)|\uu_\varepsilon|^2\,\d x+ \|\nabla \mu_\varepsilon\|^2\notag\\
&\quad =\int_\Omega \mu_\varepsilon S(1-\vp_\varepsilon)\, \d x +\int_\Omega R\mu_\varepsilon \,\d x+ \int_\Omega P_\varepsilon S\,\d x \notag\\
&\quad =\int_\Omega (\mu_\varepsilon - \overline{\mu_\varepsilon}) S(1-\vp_\varepsilon)\, \d x
+ \int_\Omega R\mu_\varepsilon \,\d x
+ \int_\Omega S \mathcal{N}\mathrm{div}\,(\nu(\vp_\varepsilon)\u_\varepsilon) \, \d x  \notag\\
&\qquad -\int_\Omega S \mathcal{N}\mathrm{div}\,((\mu_\varepsilon-\overline{\mu_\varepsilon}) \nabla\vp_\varepsilon)\, \d x\notag\\
&\quad =: J_1+J_2+J_3+J_4,
\label{aBEL}
\end{align}
where
\begin{align}
\mathcal{E}_\varepsilon(t)=\frac12 \|\nabla \vp_\varepsilon(t)\|^2 + \int_\Omega \Psi_\varepsilon (\vp_\varepsilon(t))\, \d x.
\label{Ee}
\end{align}
In the above computation, we have used \eqref{Pe} and the assumption $\overline{S}=0$.

Let us estimate the four terms on the right-hand side of \eqref{aBEL}. Using the Poincar\'{e}-Wirtinger inequality and Young's inequality, we get
\begin{align*}
J_1 & = \int_\Omega (\mu_\varepsilon - \overline{\mu_\varepsilon}) S(1-\vp_\varepsilon)\, \mathrm{d}x
\notag\\
& \leq \|\mu_\varepsilon - \overline{\mu_\varepsilon}\|_{L^4} \|S\|\|1-\vp_\varepsilon\|_{L^4} \\
&\leq \|\mu_\varepsilon - \overline{\mu_\varepsilon}\|_{H^1}\|S\| (1+\|\vp_\varepsilon\|_{H^1})  \\
& \leq C\|\nabla \mu_\varepsilon \| \|S\|(1+\|\vp_\varepsilon-\overline{\vp_\varepsilon}\|_{H^1})\\
&\leq \frac{1}{6} \|\nabla \mu_\varepsilon \| ^2 +C\|S\|^2(1+\|\nabla \vp_\varepsilon\|^2).
\end{align*}
For $J_2$, we infer from \eqref{meanmu} that
\begin{align*}
J_2
&\leq \|R\|\|\mu_\varepsilon-\overline{\mu_\varepsilon}\| +\|R\|_{L^1}|\mu_\varepsilon|\\
&\leq C\|R\|\|\nabla \mu_\varepsilon\|
+ C\|R\|\big(\| \mu_\varepsilon - \overline{\mu_\varepsilon}\| \|\vp_\varepsilon\| + C(1+\|\vp_\varepsilon\|^2)\big) \\
&\leq  \frac{1}{6} \|\nabla \mu_\varepsilon \| ^2
+ C(1+\|R\|^2)(1+\|\vp_\varepsilon\|^2).
\end{align*}
On the other hand, for $J_3$ and $J_4$, it holds
\begin{align*}
J_3
&\leq \|\mathcal{N}\mathrm{div}\,(\nu(\vp_\varepsilon)\u_\varepsilon)\|
\|S\| \leq \|\nu(\vp_\varepsilon)\|_{L^\infty}\|\u_\varepsilon\|\|S\| \\
&\leq C\nu^*\|\u_\varepsilon\|\|S\| \leq
\frac{1}{2}\nu_*\|\u_\varepsilon\|^2+ C\|S\|^2,\\
J_4&\leq \|\mathcal{N}\mathrm{div}\,((\mu_\varepsilon -\overline{\mu_\varepsilon}) \nabla\vp_\varepsilon)\|\|S\|\\
&\leq \|\mu_\varepsilon-\overline{\mu_\varepsilon}\|_{L^3}\| \nabla\vp_\varepsilon\|\|S\|\\
&\leq \frac{1}{6} \|\nabla \mu_\varepsilon \| ^2
+ C\|S\|^2\|\nabla \vp_\varepsilon\|^2.
\end{align*}
Collecting the above estimates, from \eqref{aBEL} we deduce that
\begin{align}
&\frac{\d}{\d t}\mathcal{E}_\varepsilon(t) + \frac12 \int_\Omega \nu(\vp_\varepsilon)|\uu_\varepsilon|^2\,\d x+ \frac12 \|\nabla \mu_\varepsilon\|^2
 \leq  C(1+\|R\|^2+\|S\|^2)(1+\|\nabla \vp_\varepsilon\|^2),
\label{aBEL1}
\end{align}
where $C>0$ is independent of $\varepsilon$. Observing that $
\|\nabla \vp_\varepsilon\|^2\leq 2(\mathcal{E}_\varepsilon+\gamma_2|\Omega|),
$ we further obtain
\begin{align}
&\frac{\d}{\d t}\mathcal{E}_\varepsilon(t) + \frac12 \int_\Omega \nu(\vp_\varepsilon)|\uu_\varepsilon|^2\,\d x
+ \frac12 \|\nabla \mu_\varepsilon\|^2 \leq  C(1+\|R\|^2+\|S\|^2)(1+\mathcal{E}_\varepsilon(t)).
\label{aBeL2}
\end{align}
Thus, from Gronwall's lemma it follows  that
\be
\frac{1}{2} \| \nabla \vp_\varepsilon(t)\|^2
+ \int_{\Omega} \Psi_\varepsilon(\vp_\varepsilon(t)) \, \d x
\leq C_T, \quad \forall\, t\in [0,T],
\label{fres}
\ee
and
\be
\label{fres2}
\int_0^T \left(\| \u_\varepsilon(t) \|^2 +
\| \nabla \mu_\varepsilon(t)\|^2\right) \,  \d t\leq
C_T,
\ee
where the constant $C_T$ may depend on the initial energy
$\mathcal{E}_\varepsilon(0)$, $\nu_*$, $\nu^*$, $\gamma_2$, $\Omega$ and $T$. In particular, from the argument in Step 1, we see that
$\mathcal{E}_\varepsilon(0)=(1/2) \|\nabla \vp_0\|^2 + \int_\Omega \Psi (\vp_0)\, \d x$, that is, the initial energy is independent of $\varepsilon$.\medskip

\textbf{Fifth estimate}.
Arguing as in \cite{GGW}, we test the fourth equation in \eqref{A1CHHS} with $-\Delta \vp_\varepsilon$ and we obtain
$$
\| \Delta \vp_\varepsilon\|^2
-( \Psi_\varepsilon'(\vp_\varepsilon),
\Delta \vp_\varepsilon )=- ( \mu_\varepsilon,\Delta \vp_\varepsilon ).
$$
Exploiting the integration by parts and the homogeneous Neumann boundary condition
for $\vp_\varepsilon$, we get
$$
\| \Delta \vp_\varepsilon \|^2
+( \Psi_\varepsilon''(\vp_\varepsilon )\nabla \vp_\varepsilon ,
\nabla \vp_\varepsilon )=(\nabla \mu_\varepsilon ,
\nabla \vp_\varepsilon ).
$$
Hence, we deduce from \eqref{psiep} that
$
\| \Delta \vp_\varepsilon \|^2
\leq \alpha \| \nabla \vp_\varepsilon\|^2
+ \| \nabla \mu_\varepsilon \|\| \nabla \vp_\varepsilon\|
$.
This implies
\begin{align}
\|\vp_\varepsilon\|_{H^2}^2\leq C(1+\|\nabla \mu_\varepsilon\|).\label{H2}
\end{align}
Taking the square of both sides and integrating in time, we deduce from \eqref{fres2}--\eqref{H2} that
\begin{align}
\label{estdeltavp0}
&\int_0^T \| \vp_\varepsilon (t)\|_{H^2}^4 \, \d t
\leq C_T.
\end{align}
Besides, recalling \eqref{meanmu} and using \eqref{fres}--\eqref{fres2}, we easily see that
\begin{align}
\int_0^T \|\mu_\varepsilon(t)\|_{H^1}^2 \,\d t
\leq C\int_0^T (\|\nabla \mu_\varepsilon(t)\|^2 +|\overline{\mu_\varepsilon}(t)|^2)\,\d t\leq C_T.
\label{L2H1mu}
\end{align}
On the other hand, by the same arguments used in \cite[Section 5]{Gio2020} (see also Lemma \ref{esing}), we have
\begin{align}
\|F_\varepsilon'(\vp_\varepsilon)\|_{L^p} \leq C(1+\|\nabla \mu_\varepsilon\|),\quad \|\vp_\varepsilon\|_{W^{2,p}}\leq C(1+\|\nabla \mu_\varepsilon\|),
\label{W2p}
\end{align}
for any $p\in [1,+\infty)$. As a consequence,
it holds
$$
\int_0^T \|F_\varepsilon'(\vp_\varepsilon)(t)\|_{L^p}^2\,\d t\leq C_T,\quad \int_0^T \|\vp_\varepsilon(t)\|_{W^{2,p}}^2 \,\d t\leq C_T.
$$

\textbf{Sixth estimate}.
We proceed to derive some higher-order estimates for the approximate solution. To this end, we first infer from Lemma \ref{Bo} that, since $S\in C([0,T];L_0^2(\Omega))$
(cf. $\mathbf{(A3)}$), there exists  $\vv_S\in C([0,T];\mathbf{H}^1(\Omega))$ satisfying
$$
\mathrm{div}\,\vv_S=S\quad \text{in}\ \Omega,\quad \vv=\mathbf{0}\quad \text{on}\ \partial \Omega,
$$
for $t\in[0,T]$. Then, setting
$$
\widetilde{\u}_\varepsilon=\u_\varepsilon-\vv_S,
$$
we can rewrite the approximate problem \eqref{A1CHHS} in the following way:
\be
\begin{cases}
\label{A2CHHS}
\nu(\vp_\varepsilon) \tu_\varepsilon = -\nu(\vp_\ve)\vv_S
-\nabla P_\varepsilon
+ \mu_\varepsilon \nabla \vp_\varepsilon,\\
\text{div} \tu_\varepsilon=0, \\
\partial_t \vp_\varepsilon + \mathrm{div}\,(\vp_\varepsilon \tu_\varepsilon)   =
 \Delta \mu_\varepsilon + S+ R -\mathrm{div}\,(\vp_\varepsilon \vv_S) ,\\
\mu_\varepsilon= - \Delta \vp_\varepsilon +
 \Psi'_\varepsilon(\vp_\varepsilon),
\end{cases}
 \quad \text{ in } Q,
\ee
subject to the initial and boundary conditions
\be
\label{bdini2a}
\begin{cases}
\tu_\varepsilon \cdot \n= \partial_\n \mu_\varepsilon
=\partial_\n \vp_\varepsilon=0, \quad\ \text{on\ } \Sigma,\\
\vp_\varepsilon(\cdot,0)=\vp_0, \qquad\qquad \quad \qquad\,   \text{in } \Omega.
\end{cases}
\ee

Differentiating (formally) the first equation of \eqref{A2CHHS} with respect to time, multiplying the resultant by $\tu_\varepsilon$ and integrating over $\Omega$, we obtain
\begin{align}
\frac12\frac{\d}{\d t}\int_\Omega \nu(\vp_\varepsilon)|\tu_\varepsilon|^2\, \d x
&=
\langle \partial_t\mu_\varepsilon,  \nabla \vp_\varepsilon \cdot \tu_\varepsilon \rangle_{(H^1)',H^1} +\int_\Omega \mu_\ve \nabla \partial_t\vp_\ve \cdot \tu_\ve \,\d x\notag\\
&\quad -\frac12\int_\Omega \nu'(\vp_\ve)\partial_t\vp_\ve |\tu_\ve|^2 \,\d x -
\int_\Omega \partial_t (\nu(\vp_\ve)\vv_S) \cdot \tu_\ve \,\d x.
\label{ut1}
\end{align}
We remark that the above computation can be rigorously justified by a standard approximation procedure using the time difference operator $\partial_t^h f=h^{-1}[f(t+h)-f(t)]$ and
then letting $h$ go to $0$
(see e.g., \cite[pp. 21]{Gio2020}). In a similar manner, we have
\begin{align}
&\frac{\d}{\d t}\int_\Omega \mathrm{div}\,(\vp_\ve \vv_S)\mu_\ve\, \d x= \langle \partial_t (\mathrm{div}\,(\vp_\ve \vv_S)),\mu_\ve\rangle_{(H^1)',H^1}
+
\langle  \partial_t\mu_\ve, \mathrm{div}\,(\vp_\varepsilon \vv_S) \rangle_{(H^1)',H^1},
\label{vpv1}\\
& \frac{\d}{\d t}\int_\Omega (S+R) \mu_\ve \,\d x
= \langle \partial_t\mu_\ve, S+R \rangle_{(H^1)',H^1}
+ \langle \partial_t(S+R),\mu_\ve \rangle_{(H^1)',H^1}.
\label{srmu1}
\end{align}
Furthermore, multiplying the third equation of \eqref{A2CHHS} by $\partial_t \mu_\ve=-\Delta \partial_t \vp_\ve+ \Psi_\varepsilon''(\vp_\ve) \partial_t\vp_\ve$ and integrating
over $\Omega$, after integration by parts, we get
\begin{align}
&\frac{1}{2}\frac{\d}{\d t} \|\nabla \mu_\ve\|^2 +\|\nabla \partial_t \vp_\ve\|^2+\int_\Omega F_\ve''(\vp_\ve) (\partial_t\vp_\ve)^2 \,\d x\notag\\
&\quad = -\langle \partial_t\mu_\ve, \tu_\ve \cdot \nabla \vp_\ve \rangle_{(H^1)',H^1} +\Theta_0\|\partial_t\vp_\ve\|^2 +\langle  \partial_t\mu_\ve, S+R\rangle_{(H^1)',H^1} \notag\\
&\qquad -\langle  \partial_t\mu_\ve,\mathrm{div}\,(\vp_\varepsilon \vv_S) \rangle_{(H^1)',H^1}.\label{mut1}
\end{align}
Adding \eqref{ut1}--\eqref{mut1} together yields
\begin{align}
&\frac{\d}{\d t}\Lambda_\ve(t) +\|\nabla \partial_t \vp_\ve\|^2 +\int_\Omega F_\ve''(\vp_\ve) (\partial_t\vp_\ve)^2 \,\d x\notag\\
&\quad = \int_\Omega \mu_\ve \nabla \partial_t\vp_\ve \cdot \tu_\ve\, \d x -\frac12\int_\Omega \nu'(\vp_\ve)\partial_t\vp_\ve |\tu_\ve|^2 \,\d x -
\int_\Omega \partial_t (\nu(\vp_\ve)\vv_S) \cdot \tu_\ve \,\d x\notag\\
&\qquad +\Theta_0\|\partial_t\vp_\ve\|^2
- \langle \partial_t(S+R), \mu_\ve\rangle_{(H^1)',H^1} +\langle \partial_t (\mathrm{div}\,(\vp_\ve \vv_S)),\mu_\ve \rangle_{(H^1)',H^1}\notag\\
&=: \sum_{i=5}^{10} J_i,
\label{umut1}
\end{align}
where
\begin{align}
\Lambda_\ve(t)= \frac12\int_\Omega \nu(\vp_\varepsilon)|\tu_\varepsilon|^2\, \d x+ \frac12\|\nabla \mu_\ve\|^2+ \int_\Omega \mathrm{div}\,(\vp_\ve \vv_S)\mu_\ve\, \d x
-\int_\Omega (S+R) \mu_\ve\, \d x.
\label{Lambda1}
\end{align}
From \eqref{mean}, \eqref{meanmu}, \eqref{fres} and the Poincar\'{e}-Wirtinger inequality we see that
\begin{align*}
\left|\int_\Omega \mathrm{div}\,(\vp_\ve \vv_S)\mu_\ve \,\d x\right|
&\leq (\|\vp_\ve \,\mathrm{div}\, \vv_S\|_{L^\frac43} + \|\nabla \vp_\ve\cdot\vv_S\|_{L^\frac43})\|\mu_\ve\|_{L^4}\\
&\leq C(\|\vp_\ve\|_{L^4}\|S\|
+ \|\nabla \vp_\ve\|\|\vv_S\|_{L^4}) \|\mu_\ve\|^\frac12\|\mu_\ve\|_{H^1}^\frac12\\
&\leq C\|S\|(1+\|\nabla \mu_\ve\|)\\
&\leq \frac18\|\nabla \mu_\ve\|^2+C(1+\|S\|^2),
\end{align*}
and
\begin{align*}
\left|\int_\Omega (S+R) \mu_\ve\, \d x\right|
&\leq (\|S\|+\|R\|)\|\mu_\ve\|\nonumber\\
&\leq C(\|S\|+\|R\|)(1+\|\nabla \mu_\ve\|)\notag\\
&\leq \frac18\|\nabla \mu_\ve\|^2 +C(1+\|S\|^2+\|R\|^2) .
\end{align*}
From $\mathbf{(A3)}$  we also note that
$S, R\in C([0,T]; L^2(\Omega))$ (see \cite{Si87}). As a consequence, setting
$$
\Lambda^*_\ve(t)=\Lambda_\ve(t) + S^*\quad \text{with} \quad S^*=C(1+\|S\|^2_{L^\infty(0,T;L^2(\Omega))}
+\|R\|^2_{L^\infty(0,T;L^2(\Omega))}),
$$
for some sufficiently large $C>0$, it follows that
\begin{align}
\frac12\nu_* \|\tu_\varepsilon\|^2 + \frac14 \|\nabla \mu_\ve\|^2 +1
\leq \Lambda^*_\ve(t)
\leq \frac12\nu^* \|\tu_\varepsilon\|^2 + \|\nabla \mu_\ve\|^2 + 2S^*.
\label{Lambda2}
\end{align}

In what follows, we estimate the right-hand side of \eqref{umut1} in terms of the quantity $\Lambda^*_\ve$. By a simple modification of the argument in \cite[(5.25)]{Gio2020},
we deduce
\begin{align}
\|\partial_t\vp_\varepsilon\|_{(H^1)'}
\leq C
\|\u_\varepsilon\| \ln^\frac12(C+C\|\nabla \mu_\varepsilon\|)
+\|\nabla \mu_\varepsilon\|+\|S\|+\|R\|.
\label{phit1}
\end{align}
From the identity
$$
\nu(\vp_\varepsilon)\mathrm{curl}\,\u_\varepsilon
+\nu'(\vp_\varepsilon)\nabla \vp_\varepsilon \cdot \u_\varepsilon^\perp =\nabla \mu_\varepsilon \cdot (\nabla \vp_\varepsilon)^\perp,
$$
where $\vv^\perp=(v_2,-v_1)^T$, for any $\vv=(v_1,v_2)^T$, we can conclude from \eqref{BGW}, the assumption $\mathbf{(A2)}$ and \eqref{W2p} that (see \cite[pp. 21]{Gio2020})
\begin{align*}
\|\mathrm{curl}\,\u_\varepsilon\|\leq
\nu_*^{-1}C(\|\u_\varepsilon\|+\|\nabla \mu_\varepsilon\|)(1+\|\nabla \mu_\varepsilon\|)^\frac12 \ln^\frac12(C+C\|\nabla \mu_\varepsilon\|),
\end{align*}
which implies
\begin{align*}
\|\mathrm{curl}\,\widetilde{\u}_\varepsilon\|\leq
\nu_*^{-1}C(\|\widetilde{\u}_\varepsilon\|+\|\vv_S\| +\|\nabla \mu_\varepsilon\|)(1+\|\nabla \mu_\varepsilon\|)^\frac12 \ln^\frac12(C+C\|\nabla \mu_\varepsilon\|)
+\|\mathrm{curl}\, \vv_S\|.
\end{align*}
Thus, recalling \eqref{rot}, we have
\begin{align}
\|\widetilde{\u}_\varepsilon\|_{H^1}
\leq C(\|\widetilde{\u}_\varepsilon\|+\|\nabla \mu_\varepsilon\|+\|S\|)(1+\|\nabla \mu_\varepsilon\|)^\frac12 \ln^\frac12(C+C\|\nabla \mu_\varepsilon\|)+C\|S\|.
\label{ucurl}
\end{align}
From the definition of $\Lambda^*_\ve$ and \eqref{phit1}, \eqref{ucurl}, we have
$$
\|\partial_t\vp_\ve\|_{(H^1)'}\leq C(\Lambda^*_\ve)^\frac12\ln^\frac12(C+C\Lambda^*_\ve),\quad
\|\tu_\ve\|_{H^1}\leq C(\Lambda^*_\ve)^\frac34\ln^\frac12(C+C\Lambda^*_\ve).
$$
On the other hand, using interpolation
\begin{align*}
\|\partial_t\vp_\ve\|^2 &\leq \|\partial_t\vp_\ve\|_{(H^1)'} \|\partial_t\vp_\ve\|_{H^1}
=\|\partial_t\vp_\ve\|_{(H^1)'}(\|\partial_t\vp_\ve\|+\|\nabla \partial_t\vp_\ve\|)\\
&\leq \frac12\|\partial_t\vp_\ve\|^2 +\frac12\|\partial_t\vp_\ve\|_{(H^1)'}^2+ \|\partial_t\vp_\ve\|_{(H^1)'} \|\nabla \partial_t\vp_\ve\|,
\end{align*}
we also have
\begin{align}
\|\partial_t\vp_\ve\|
\leq 2(\|\partial_t\vp_\ve\|_{(H^1)'} +\|\partial_t\vp_\ve\|_{(H^1)'}^\frac12 \|\nabla \partial_t\vp_\ve\|^\frac12).
\label{L2phit}
\end{align}

After the above preparations, recalling that $\mathrm{div}\,\tu_\ve=0$, we can estimate the term $J_5$ like in \cite[(5.38)--(5.43)]{Gio2020} with suitable modifications. More precisely, we have
\begin{align*}
J_5
&\leq \|\tu_\ve\|\|\partial_t\vp_\ve\|_{L^4}\|\nabla \mu_\ve\|_{L^4}
\\
&\leq C\|\tu_\ve\| \|\partial_t\vp_\ve\|^\frac12 \|\partial_t\vp_\ve\|_{H^1}^\frac12
\|\nabla \mu_\ve\|^\frac12\|\mu_\ve\|_{H^2}^\frac12
\\
&\leq C\|\tu_\ve\| (\|\partial_t\vp_\ve\|+\|\partial_t\vp_\ve\|^\frac12 \|\nabla \partial_t\vp_\ve\|^\frac12)
\|\nabla \mu_\ve\|^\frac12\|\mu_\ve\|_{H^2}^\frac12
\\
&\leq C\|\tu_\ve\| (\|\partial_t\vp_\ve\|_{(H^1)'} +\|\partial_t\vp_\ve\|_{(H^1)'}^\frac12 \|\nabla \partial_t\vp_\ve\|^\frac12
+\|\partial_t\vp_\ve\|_{(H^1)'}^\frac14 \|\nabla \partial_t\vp_\ve\|^\frac34)\non\\
&\quad \times
\|\nabla \mu_\ve\|^\frac12
(\|\mu_\ve\|^\frac12+\|\Delta\mu_\ve\|^\frac12)
\\
&\leq C\|\tu_\ve\| (\|\partial_t\vp_\ve\|_{(H^1)'} +\|\partial_t\vp_\ve\|_{(H^1)'}^\frac14 \|\nabla \partial_t\vp_\ve\|^\frac34)
\|\nabla \mu_\ve\|^\frac12
(\|\mu_\ve\|^\frac12+\|\Delta\mu_\ve\|^\frac12)
\\
&\leq
C\|\tu_\ve\| \|\partial_t\vp_\ve\|_{(H^1)'} \|\nabla \mu_\ve\|^\frac12
(1+\|\nabla \mu_\ve\|^\frac12+\|S\|^\frac12+\|R\|^\frac12)\non\\
&\quad
+ C\|\tu_\ve\| \|\partial_t\vp_\ve\|_{(H^1)'} \|\nabla \mu_\ve\|^\frac12\|\partial_t\vp_\ve\|^\frac12
\\
&\quad
+ C\|\tu_\ve\| \|\partial_t\vp_\ve\|_{(H^1)'} \|\nabla \mu_\ve\|^\frac12 \|\tu_\ve\cdot\nabla\vp_\ve\|^\frac12
+ C\|\tu_\ve\| \|\partial_t\vp_\ve\|_{(H^1)'} \|\nabla \mu_\ve\|^\frac12\|\mathrm{div}\,(\vp_\ve\vv_S)\|^\frac12
\\
&\quad
+ C\|\tu_\ve\|\|\partial_t\vp_\ve\|_{(H^1)'}^\frac14 \|\nabla \partial_t\vp_\ve\|^\frac34\|\nabla \mu_\ve\|^\frac12(1+\|\nabla \mu_\ve\|^\frac12+
\|S\|^\frac12+\|R\|^\frac12)
\\
&\quad
+ C\|\tu_\ve\|\|\partial_t\vp_\ve\|_{(H^1)'}^\frac14 \|\nabla \partial_t\vp_\ve\|^\frac34\|\nabla \mu_\ve\|^\frac12\|\partial_t\vp_\ve\|^\frac12
\\
&\quad
+C\|\tu_\ve\|\|\partial_t\vp_\ve\|_{(H^1)'}^\frac14 \|\nabla \partial_t\vp_\ve\|^\frac34\|\nabla \mu_\ve\|^\frac12\|\tu_\ve\cdot\nabla\vp_\ve\|^\frac12
\\
&\quad +C\|\tu_\ve\|\|\partial_t\vp_\ve\|_{(H^1)'}^\frac14 \|\nabla \partial_t\vp_\ve\|^\frac34\|\nabla \mu_\ve\|^\frac12 \|\mathrm{div}\,(\vp_\ve\vv_S)\|^\frac12
\\
&=: \sum_{i=1}^{8}J_5^{(i)},
\end{align*}
where we have used the third equation in \eqref{A2CHHS} to compute the term $\Delta \mu_\ve$. We find that
\begin{align*}
J_5^{(1)}&\leq C (\Lambda^*_\ve)^\frac32\ln^\frac12(C+C\Lambda^*_\ve),\\
J_5^{(2)}&\leq  C\|\tu_\ve\| \|\partial_t\vp_\ve\|_{(H^1)'}\|\nabla \mu_\ve\|^\frac12(\|\partial_t\vp_\ve\|_{(H^1)'} ^\frac12
+\|\partial_t\vp_\ve\|_{(H^1)'}^\frac14 \|\nabla \partial_t\vp_\ve\|^\frac14)\\
&\leq \frac{1}{24} \|\nabla \partial_t\vp_\ve\|^2
+ C(\Lambda^*_\ve)^\frac32\ln^\frac34(C+C\Lambda^*_\ve)
+ C(\Lambda^*_\ve)^\frac{11}{7}\ln^\frac57(C+C\Lambda^*_\ve)\\
&\leq \frac{1}{24} \|\nabla \partial_t\vp_\ve\|^2
+ C(\Lambda^*_\ve)^\frac{11}{7}\ln^\frac34(C+C\Lambda^*_\ve),\\
J_5^{(3)}&\leq
C\|\tu_\ve\|^\frac32 \|\partial_t\vp_\ve\|_{(H^1)'} \|\nabla \mu_\ve\|^\frac12 \|\nabla\vp_\ve\|_{\mathbf{L}^\infty}^\frac12\\
&\leq C (\Lambda^*_\ve)^\frac32\ln^\frac12(C+C\Lambda^*_\ve)(1+\|\nabla \vp_\ve\|_{H^1}\ln^\frac12(e+\|\nabla \vp_\ve\|_{W^{1,3}}))^\frac12\\
&\leq C (\Lambda^*_\ve)^\frac{13}{8}\ln^\frac34(C+C\Lambda^*_\ve),
\end{align*}
where in the estimate for $J_5^{(3)}$ we have used inequality \eqref{BGW} with $f=\partial_{x_i} \vp_\ve$, $i=1,2$, and the estimates \eqref{H2}, \eqref{W2p}.
Arguing similarly, we have
\begin{align*}
J_5^{(4)}&\leq C\|\tu_\ve\| \|\partial_t\vp_\ve\|_{(H^1)'} \|\nabla \mu_\ve\|^\frac12(\|\vp_\ve S\|^\frac12+\|\vv_S\cdot\nabla \vp_\ve\|^\frac12)\\
&\leq C\|\tu_\ve\| \|\partial_t\vp_\ve\|_{(H^1)'} \|\nabla \mu_\ve\|^\frac12\|S\|^\frac12\|\vp_\ve\|_{W^{1,\infty}}^\frac12\\
&\leq C(\Lambda^*_\ve)^\frac54\ln^\frac12(C+C\Lambda^*_\ve)\|S\|^\frac12
(1+\|\vp_\ve\|_{H^2}\ln^\frac12(e+\|\vp_\ve\|_{W^{2,3}}))^\frac12\\
&\leq C(\Lambda^*_\ve)^\frac{13}{8}\ln^\frac34(C+C\Lambda^*_\ve).
\end{align*}
Concerning the left terms involving $\nabla \partial_t\vp_\ve$, we use Young's inequality to obtain
\begin{align*}
J_5^{(5)}&\leq \frac{1}{24} \|\nabla \partial_t\vp_\ve\|^2 +  C\|\tu_\ve\|^\frac85\|\partial_t\vp_\ve\|_{(H^1)'}^\frac25 \|\nabla \mu_\ve\|^\frac45(1+\|\nabla \mu_\ve\|^\frac45+
\|S\|^\frac45+\|R\|^\frac45)\\
&\leq \frac{1}{24} \|\nabla \partial_t\vp_\ve\|^2 + C(\Lambda^*_\ve)^\frac95\ln^\frac15(C+C\Lambda^*_\ve),\\
J_5^{(6)}&\leq \frac{1}{24} \|\nabla \partial_t\vp_\ve\|^2 + C\|\tu_\ve\|^\frac85\|\partial_t \vp_\ve\|_{(H^1)'}^\frac25 \|\nabla \mu_\ve\|^\frac45\|\partial_t\vp_\ve\|^\frac45\\
&\leq \frac{1}{24} \|\nabla \partial_t\vp_\ve\|^2
+ C\|\tu_\ve\|^\frac85\|\nabla \mu_\ve\|^\frac45
(\|\partial_t\vp_\ve\|_{(H^1)'} ^\frac65 +\|\partial_t\vp_\ve\|_{(H^1)'}^\frac45 \|\nabla \partial_t\vp_\ve\|^\frac25)
\\
&\leq \frac{1}{12} \|\nabla \partial_t\vp_\ve\|^2
+ C(\Lambda^*_\ve)^\frac95\ln^\frac35(C+C\Lambda^*_\ve)
+ C(\Lambda^*_\ve)^2\ln^\frac12(C+C\Lambda^*_\ve)\\
&\leq \frac{1}{12} \|\nabla \partial_t\vp_\ve\|^2
+C(\Lambda^*_\ve)^2\ln^\frac35(C+C\Lambda^*_\ve),\\
J_5^{(7)}&\leq \frac{1}{24} \|\nabla \partial_t\vp_\ve\|^2
+ C\|\tu_\ve\|^\frac{12}{5} \|\partial_t\vp_\ve\|_{(H^1)'}^\frac25 \|\nabla \mu_\ve\|^\frac45\|\nabla\vp_\ve\|_{L^\infty}^\frac45\\
&\leq \frac{1}{24} \|\nabla \partial_t\vp_\ve\|^2
+ C(\Lambda^*_\ve)^\frac95\ln^\frac15(C+C\Lambda^*_\ve) (1+\|\nabla \vp_\ve\|_{H^1}\ln^\frac12(e+\|\nabla \vp_\ve\|_{W^{1,3}}))^\frac45\\
&\leq \frac{1}{24} \|\nabla \partial_t\vp_\ve\|^2 + C(\Lambda^*_\ve)^2\ln^\frac35(C+C\Lambda^*_\ve),\\
J_5^{(8)} &\leq
\frac{1}{24} \|\nabla \partial_t\vp_\ve\|^2
+C\|\tu_\ve\|^\frac85\|\partial_t\vp_\ve\|_{(H^1)'}^\frac25 \|\nabla \mu_\ve\|^\frac45 (\|\vp_\ve S\|^\frac45+\|\vv_S\cdot\nabla \vp_\ve\|^\frac45)\\
&\leq \frac{1}{24} \|\nabla \partial_t\vp_\ve\|^2
+C\|\tu_\ve\|^\frac85\|\partial_t\vp_\ve\|_{(H^1)'}^\frac25\|\nabla \mu_\ve\|^\frac45  \|S\|^\frac45\|\vp_\ve\|_{W^{1,\infty}}^\frac45\\
&\leq \frac{1}{24} \|\nabla \partial_t\vp_\ve\|^2
+C\|\tu_\ve\|^\frac85\|\partial_t\vp_\ve\|_{(H^1)'}^\frac25\|\nabla \mu_\ve\|^\frac45  \|S\|^\frac45(1+\|\vp_\ve\|_{H^2} \ln^\frac12(e+\|\vp_\ve\|_{W^{2,3}}))^\frac45\\
&\leq \frac{1}{24} \|\nabla \partial_t\vp_\ve\|^2
+ C(\Lambda^*_\ve)^\frac85\ln^\frac35(C+C\Lambda^*_\ve).
\end{align*}
Collecting the above estimates, we arrive at
\begin{align*}
J_5&\leq \frac{1}{4}\|\nabla \partial_t\vp_\ve\|^2
+ C(\Lambda^*_\ve)^2\ln(C+C\Lambda^*_\ve).
\end{align*}

Next, we estimate $J_6$ by using \eqref{phit1}--\eqref{L2phit}. This gives
\begin{align*}
J_6&\leq \nu^*\|\partial_t\vp_\ve\|\|\tu_\ve\|_{L^4}^2\\
&\leq C(\|\partial_t\vp_\ve\|_{(H^1)'} +\|\partial_t\vp_\ve\|_{(H^1)'}^\frac12 \|\nabla \partial_t\vp_\ve\|^\frac12)\|\tu_\ve\|\|\tu_\ve\|_{H^1}\\
&\leq  \frac{1}{16} \|\nabla \partial_t\vp_\ve\|^2
+ C\|\partial_t\vp_\ve\|_{(H^1)'} \|\tu_\ve\|\|\tu_\ve\|_{H^1}+C  \|\partial_t\vp_\ve\|_{(H^1)'}^\frac23 \|\tu_\ve\|^\frac43 \|\tu_\ve\|_{H^1}^\frac43\\
&\leq \frac{1}{16} \|\nabla \partial_t\vp_\ve\|^2
+ C(\Lambda^*_\ve)^\frac74\ln(C+C\Lambda^*_\ve)
+ C (\Lambda^*_\ve)^2\ln(C+C\Lambda^*_\ve)\\
&\leq \frac{1}{16} \|\nabla \partial_t\vp_\ve\|^2
+ C (\Lambda^*_\ve)^2\ln(C+C\Lambda^*_\ve).
\end{align*}
Besides, it follows that
\begin{align*}
J_7&\leq \nu^*\|\partial_t\vp_\ve\|_{L^4} \|\vv_S\|_{L^4}\|\tu_\ve\|
+ \nu^* \|\partial_t\vv_S\|\|\tu_\ve\|\\
&\leq C(\|\partial_t\vp_\ve\|_{(H^1)'} +\|\partial_t\vp_\ve\|_{(H^1)'}^\frac14 \|\nabla \partial_t\vp_\ve\|^\frac34)\|S\|\|\tu_\ve\|
+C \|\partial_t S\|_{(H^1)'}\|\tu_\ve\|\\
&\leq \frac{1}{16} \|\nabla \partial_t\vp_\ve\|^2
+ C\Lambda^*_\ve\ln^\frac12(C+C\Lambda^*_\ve)
+ C\|S\|^\frac{8}{5} \Lambda^*_\ve \ln^\frac{1}{5}(C+C\Lambda^*_\ve)\\
&\quad +C\Lambda^*_\ve+\|\partial_t S\|_{(H^1)'}^2\\
&\leq \frac{1}{16} \|\nabla \partial_t\vp_\ve\|^2 +C\Lambda^*_\ve\ln^\frac12(C+C\Lambda^*_\ve)+\|\partial_t S\|_{(H^1)'}^2.
\end{align*}
From \eqref{L2phit}, it is straightforward to check that
\begin{align*}
J_8&\leq \frac{1}{16} \|\nabla \partial_t\vp_\ve\|^2 +
C \Lambda^*_\ve\ln(C+C\Lambda^*_\ve).
\end{align*}
The term $J_9$ can be estimated as
\begin{align*}
J_9&\leq  (\|\partial_t S\|_{(H^1)'}+\|\partial_t R\|_{(H^1)'})\|\mu_\ve\|_{H^1}\\
&\leq C(\|\partial_t S\|_{(H^1)'}+\|\partial_t R\|_{(H^1)'})(1+\|\nabla \mu_\ve\|)\\
&\leq C\Lambda^*_\ve + \|\partial_t S\|_{(H^1)'}^2+\|\partial_t R\|_{(H^1)'}^2,
\end{align*}
while for $J_{10}$, we get
\begin{align*}
J_{10}&\leq (\|\nabla \vp_\ve\cdot \partial_t \vv_S\|_{(H^1)'}+\|\nabla\partial_t\vp_\ve\cdot \vv_S\|_{(H^1)'}
+\|\partial_t\vp_\ve S\|_{(H^1)'}+\|\vp_\ve \partial_t S\|_{(H^1)'}) \|\mu_\ve\|_{H^1}\\
&\leq
C\big(\|\nabla \vp_\ve\|_{L^4}\|\partial_t \vv_S\|+
\|\nabla\partial_t\vp_\ve\|\|\vv_S\|_{L^4}\big) (1+\|\nabla \mu_\ve\|)\\
&\quad + C\big(\|S\|\|\partial_t\vp_\ve\|_{L^4} +\|\vp_\ve\|_{H^1}\|\partial_t S\|\big)(1+\|\nabla \mu_\ve\|)\\
&\leq  C\|\partial_t S\|(\|\vp_\ve\|_{H^1} +\|\vp_\ve\|_{H^1}^\frac12\|\vp_\ve\|_{H^2}^\frac12)(1+\|\nabla \mu_\ve\|)\\
&\quad + C\|S\|(1+\|\nabla \mu_\ve\|)(\|\partial_t\vp_\ve\|_{(H^1)'}+\|\nabla\partial_t\vp_\ve\|)\\
&\leq \frac{1}{16} \|\nabla \partial_t\vp_\ve\|^2 +C(\Lambda^*_\ve)^\frac54 + C(\Lambda^*_\ve)^2\ln(C+C\Lambda^*_\ve)+C\|S\|^2\Lambda^*_\ve + \|\partial_t S\|^2+\|S\|^2\\
&\leq \frac{1}{16} \|\nabla \partial_t\vp_\ve\|^2
+ C(\Lambda^*_\ve)^2\ln(C+C\Lambda^*_\ve) + C\Lambda^*_\ve+ \|\partial_t S\|^2.
\end{align*}

Collecting the above estimates and recalling $\mathbf{(A3)}$, we derive from \eqref{umut1} the following differential inequality
\begin{equation*}
\frac{\d}{\d t}\Lambda_\ve^*(t) + \frac12\|\nabla \partial_t \vp_\ve\|^2+\int_\Omega F_\ve''(\vp_\ve) (\partial_t\vp_\ve)^2 \,\d x
\leq C(\Lambda^*_\ve)^2\ln(C+C\Lambda^*_\ve) + C(\|\partial_t S\|^2+\|\partial_t R\|_{(H^1)'}^2).
\end{equation*}
Thanks to the time continuity of the approximate solution and $\vv_S$, we see that
\begin{align*}
\|\sqrt{\nu(\vp_\ve(0))}\tu_\ve(0)\|
&\leq C\|\mu_\ve(0)\|_V\|\vp_\ve(0)\|_{L^\infty}+C\|\vv_S(0)\|\\
&\leq C(1+\|\widetilde{\mu}_0\|_{H^1}+\|S(0)\|),
\end{align*}
which is bounded. As a consequence, the initial data for $\Lambda^*_\ve$ is bounded by
$$
1\leq \Lambda_\ve^*(0)\leq C(1+\|\widetilde{\mu}_0\|_{H^1}^2+\|S(0)\|^2)+2S^*.
$$
Hence, from Lemma \ref{gron}, $\mathbf{(A3)}$ and the fact $\Lambda^*_\ve\in L^1(0,T)$ (recall \eqref{fres2}), we can conclude that
\begin{align}
\Lambda_\ve^*(t)\leq C_T,\quad \forall\, t\in [0,T], \non
\end{align}
where $C_T>0$ depends on   $\|S\|_{L^2(0,T;H^1(\Omega))}$, $\|\partial_t S\|_{L^2(0,T; L^2(\Omega))}$, $\|R\|_{L^2(0,T;H^1(\Omega))}$,
$\|\partial_t R\|_{L^2(0,T; (H^1(\Omega))')}$, $\|\widetilde{\mu}_0\|_{H^1}$, $\mathcal{E}_\varepsilon(0)$, $\Omega$, $T$ and coefficients of the system,
but it is independent of $\ve$. Recalling the definition of $\Lambda_\ve^*$, we infer from the above estimate that
\begin{align}
\|\mu_\ve(t)\|_{H^1}+\|\widetilde{\u}_\ve(t)\|\leq C_T,\quad \forall\, t\in[0,T],\non
\end{align}
as well as
\begin{align}
\int_0^T\|\nabla \partial_t \vp_\ve(t)\|^2\,\d t\leq C_T. \non
\end{align}
As a consequence, it follows from \eqref{W2p}, \eqref{phit1}, \eqref{ucurl} that
\begin{align}
\|\u_\ve(t)\|_{H^1}+\|\vp_\ve(t)\|_{W^{2,p}} +\|\partial_t \vp_\ve(t)\|_{(H^1)'}+\|F'_\ve(\vp_\ve)(t)\|_{L^p}\leq C_T,
\label{es-high}
\end{align}
for any $t\in [0,T]$ and $p\in [2,+\infty)$. The $H^2$-estimate for the pressure $P_\varepsilon$ can be derived from Lemma \ref{press} and \eqref{es-high} (cf. \cite{Gio2020}):
 $$
\|P_\ve(t)\|_{H^2}\leq C_T,\quad \forall\, t\in[0,T].
$$
Finally, we infer from the third equation of \eqref{A2CHHS}, $\mathbf{(A3)}$, \eqref{es-high} and a standard elliptic estimate that
\begin{align*}
\|\mu_\varepsilon\|_{H^3}
&\leq C(\|\partial_t \vp_\varepsilon\|_{H^1} + \|\mathrm{div}\,(\vp_\varepsilon \u_\varepsilon)\|_{H^1} + \|S\|_{H^1}+ \|R\|_{H^1}
+ \|\mu_\ve\|)\\
&\leq C(\|\nabla \partial_t \vp_\ve\|+\|\mu_\ve\|) + C(\|S\|_{H^1}+ \|R\|_{H^1}),
\end{align*}
almost everywhere in $(0,T)$. This yields
$$
\int_0^T\|\mu_\ve(t)\|_{H^3}^2\,\d t\leq C_T.
$$

\textbf{Step 3. Existence.}
Based on the $\ve$-independent estimates obtained in the previous step, we can pass to the limit as $\varepsilon\to 0^+$ and obtain the existence of a global strong solution
$(\u,P,\varphi,\mu)$ to problem \eqref{bdini}--\eqref{CHHS} with corresponding regularity properties via a standard compactness argument (cf. \cite{GGW}).
Further regularity of the solution $(\u,P,\varphi,\mu)$ can be obtained. To this end, from $\mu\in L^\infty(0,T;H^1(\Omega))$ and Lemma \ref{esing}, we can deduce that
$F''(\varphi)\in L^\infty(0,T;L^p(\Omega))$ for any $p\in [2,+\infty)$, and the strict separation property \eqref{sep1} (cf. \cite{HW21}, see also an alternative argument
for the Cahn-Hilliard equation in \cite{GGG22}). Then from the equation $\mu=-\Delta \varphi+\Psi'(\varphi)$, $\mathbf{(A1)}$ and the elliptic estimate, we find that
$\varphi\in L^\infty(0,T;H^3(\Omega))$. Recalling $\mathbf{(A1)'}$ and using \eqref{sep1}, we further obtain $F'(\varphi)\in L^2(0,T;H^3(\Omega))$. Applying the elliptic
estimate again, we can conclude
$\varphi\in L^2(0,T;H^5(\Omega))$. For the chemical potential $\mu$, similar to \cite{Gio2020}, we get $\partial_t\mu \in L^2(0,T;(H^1(\Omega))')$. Finally, using standard
interpolation arguments, we can derive the time continuity $\varphi\in C([0,T];H^3(\Omega))$ and $\mu\in C([0,T];H^1(\Omega))$.
\medskip

\textbf{Step 4. Uniqueness.}
Uniqueness of the strong solution is a direct consequence of the continuous dependence estimate obtained in Proposition \ref{conti-es1}.

The proof of Theorem \ref{str-well} is complete. \hfill $\square$

\section*{Declarations}
\noindent \textbf{Acknowledgement.} 
The authors thank the reviewer for helpful suggestions that improved the presentation of this paper.
C. Cavaterra and M. Grasselli are members of Gruppo Nazionale per l'Ana\-li\-si Matematica, la Probabilit\`{a} e le loro Applicazioni (GNAMPA), Istituto Nazionale di Alta Matematica (INdAM). Moreover, their research is part of the activities of ``Dipartimento di Eccellenza 2023-2027'' of Universit\`a degli Studi di Milano (C. Cavaterra) and Politecnico di Milano (M. Grasselli).
H. Wu is a member of Key Laboratory of Mathematics for Nonlinear Sciences (Fudan University), Ministry of Education.\medskip

\noindent \textbf{Funding.}
C. Cavaterra and M. Grasselli were partially supported by MIUR-PRIN Grant 2020F3NCPX ``Mathematics for Industry 4.0 (Math4I4)''.
C. Cavaterra's work has also been partially supported by MIUR-PRIN Grant 2022 ``Partial differential equations and related geometric-functional inequalities''. H. Wu was partially supported by National Natural Science Foundation of China under Grant number 12071084.

\medskip

\noindent \textbf{Competing interests.}
The authors declare that they have no conflict of interest.


\end{document}